\documentclass{compositio}
\usepackage{amsfonts}
\usepackage{amsmath,amsthm}
\usepackage{graphics}
\usepackage{amssymb}
\usepackage[all]{xy}

\newcommand{\NZ}{\mathbb{N}} %natN|rliche Zahlen
\newcommand{\CZ}{{\mathbb{C}}} %komplexe Zahlen
\newcommand{\C}{\mathbb{C}_{\infty}} %C im Funktionenkoerperfall
\newcommand{\CS}{{\mathbb{C}_{\infty}^*}}
\newcommand{\AZ}{\mathbb{A}} %affiner Raum, Adele
\newcommand{\PZ}{\mathbb{P}} %projektiver Raum

\newcommand{\FZ}{{\mathbb{F}}}
\newcommand{\FC}{F_{\infty}} %Vervollstaendigung von F
\newcommand{\id}{{\mathbb{I}}}
\newcommand{\pp}{\mathfrak{p}} %Primideal p
\newcommand{\mm}{\mathfrak{m}} %Maximalideal m
\newcommand{\nn}{\mathfrak{n}}
\newcommand{\qq}{\mathfrak{q}} %Primideal q
\newcommand{\Gal}{\mathrm{Gal}}

\newcommand{\Hom}{\mathrm{Hom}}
\newcommand{\End}{\mathrm{End}}
\newcommand{\SL}{\mathrm{SL}}
\newcommand{\sep}{\mathrm{sep}}

\newcommand{\incl}{\iota_{F,\,b}^{F'}}
\newcommand{\inclz}{\iota_{F,\,b'}^{F''}}
\newcommand{\incla}{\iota_{F,\,b_1}^{F'}}
\newcommand{\inclb}{\iota_{F,\,b_2}^{F''}}
\newcommand{\level}{{\cal K}'}
\newcommand{\Mat}{\mathrm{Mat}} %Matrixring

\newcommand{\Aut}{\mathrm{Aut}}
\renewcommand{\id}{\mathrm{id}}
\def\GL{{\mathop{\rm GL}\nolimits}}
\let\phi\varphi

\newtheorem{theorem}{Theorem}[subsection]
\newtheorem{theorem*}{Theorem}
\newtheorem{satz}[theorem]{Proposition}
\newtheorem{lemma}[theorem]{Lemma}
\newtheorem{korollar}[theorem]{Corollary}
\newtheorem{definition}[theorem]{Definition}
\newtheorem{conjecture}[theorem]{Conjecture}
\newtheorem{conjecture*}[theorem*]{Conjecture}
\newtheorem{definition*}[theorem*]{Definition}

\numberwithin{equation}{subsection}

\renewenvironment{enumerate}{\begin{list}{\arabic{ctr}.}{\usecounter{ctr}\setlength{\topsep}{3mm}
  \setlength{\leftmargin}{12mm}\setlength{\labelsep}{3mm}\setlength{\labelwidth}{10mm}}}{\end{list}}
  
\newenvironment{properties}{\begin{list}{(\roman{ctr})}{\usecounter{ctr}\setlength{\topsep}{0mm}
  \setlength{\leftmargin}{12mm}\setlength{\labelsep}{5mm}\setlength{\labelwidth}{10mm}}}{\end{list}}

\newenvironment{propertiesabc}{\begin{list}{(\alph{ctr})}{\usecounter{ctr}\setlength{\topsep}{0mm}
  \setlength{\leftmargin}{12mm}\setlength{\labelsep}{5mm}\setlength{\labelwidth}{10mm}}}{\end{list}}

\begin{document}
\newcounter{ctr}

\title{The Andr\'e-Oort Conjecture for \\Drinfeld Modular Varieties}

\author{Patrik Hubschmid}
\email{patrik.hubschmid@iwr.uni-heidelberg.de}
\address{Interdisciplinary Center for Scientific Computing\\Universit\"at Heidelberg\\Im Neuenheimer Feld 368\\69120 Heidelberg\\Germany}

\thanks{The work presented in this paper has been partially funded by
  the SNSF, Switzerland (project no. 117638).}
\begin{abstract}
We consider the analogue of the Andr\'e-Oort conjecture for Drinfeld
modular varieties which was formulated by Breuer. We prove this
analogue for special points with separable reflex field over the base
field by adapting methods which were used by Klingler and Yafaev to
prove the Andr\'e-Oort conjecture under the generalized Riemann
hypothesis in the classical case. Our result extends results of Breuer
showing the correctness of the analogue for special points lying in a
curve and for special points having a certain behaviour at a fixed set
of primes.
\end{abstract}
\maketitle

% \tableofcontents
\section*{Introduction}
%\addcontentsline{toc}{section}{Introduction}
\subsection*{The Andr\'e-Oort conjecture}
The Andr\'e-Oort conjecture asserts that every irreducible component
of the Zariski closure of a set of special points in a Shimura variety
is a special subvariety. There has been remarkable progress on this
conjecture recently.

In the late 1990's, Edixhoven proved the conjecture for Hilbert
modular surfaces and
products of modular curves assuming the generalized Riemann hypothesis
(GRH) in~\cite{Ed1} and~\cite{Ed2}. Both proofs exploit the Galois
action on special points and use geometric properties of Hecke
correspondences. In the special case of a product of two modular curves,
Andr\'e~\cite{An} gave a proof without assuming GRH.

These methods were extended in~\cite{EdYa} by Edixhoven and Yafaev to
prove the conjecture for curves in general Shimura varieties 
containing infinitely many special
points all lying in the same Hecke orbit. Subsequently,
Yafaev~\cite{Ya} also proved the conjecture for general curves
assuming GRH.

Recently, Klingler, Ullmo and Yafaev have announced a proof of the full
Andr\'e-Oort conjecture assuming GRH, see~\cite{KlYa}
and~\cite{UlYa}. Their methods use a combination of the methods of
Edixhoven and Yafaev and equidistribution results of Clozel
and Ullmo~\cite{ClUl} established by methods from ergodic theory.

For a more detailed exposition of results 
concerning the Andr\'e-Oort
conjecture for Shimura varieties, we refer to the survey 
article of Noot~\cite{No}.

\subsection*{Drinfeld modular varieties}
\emph{Drinfeld modular varieties} are a natural analogue of Shimura
varieties in the function field case. They can be interpreted as
moduli spaces for Drinfeld $A$-modules over a global function 
field $F$ of a given rank~$r$ with ${\cal K}$-level structure, 
where $A$ is the ring of elements of $F$ that are regular outside 
of a fixed place $\infty$ and ${\cal K} \subset \GL_r(\AZ_F^f)$ is a
compact open subgroup with $\AZ_F^f$ the ring of 
adeles of $F$ outside $\infty$.

%As for Shimura varieties, there is an analytic description of a Drinfeld
%modular variety~$S$ as a double quotient. Let $\C$ be the completion 
%of an algebraic closure of the completion $F_{\infty}$ of $F$ and let
%$\AZ_F^f$ be the adeles of $F$ outside of $\infty$. There is a
%natural rigid-analytic isomorphism
%\begin{equation} \label{eq:rigid}
%S(\C) \cong \GL_r(F) \setminus (\Omega_F^r \times \GL_r(\AZ_F^f) /
%{\cal K}), \end{equation}
%where ${\cal K} \subset \GL_r(\AZ_F^f)$ is an open compact subgroup,
%called \emph{level},
%and $\Omega_F^r$ denotes Drinfeld's upper half-space obtained by removing all
%$F_{\infty}$-rational hyperplanes from $\PZ^{r-1}(\C)$. In this
%situation, the Drinfeld modular variety $S$ is of dimension~$r-1$ and
%$S$ is denoted by~$S_{F,\cal K}^r$.

%Also along the same lines as for Shimura varieties, one can define
%\emph{Hecke correspondences} on Drinfeld modular varieties. These are finite
%algebraic correspondences defined over the base field~$F$.

%\subsection*{Special subvarieties and Drinfeld modular subvarieties}
One can define \emph{special subvarieties} of a Drinfeld
modular variety~$S = S_{F,\cal K}^r$ parametrising Drinfeld 
$A$-modules of rank $r$ 
in analogy to the case of Shimura varieties. For each finite 
extension $F'$ of $F$ of degree $\frac{r}{r'}$ with only one place 
above $\infty$ and integral closure $A'$ of $A$ in $F'$, the restriction of
Drinfeld $A'$-modules to $A$ gives a morphism from the
moduli space of Drinfeld $A'$-modules of rank~$r'$ (with a certain
level structure) to $S$ defined over $F'$. These morphisms are
analogues of morphisms induced by a Shimura
subdatum. A special subvariety~$V$ is defined to be a geometrically irreducible
component of a Hecke translate of the image of such a morphism. 
%If
%this morphism is associated to the 
%field extension~$F' / F$, we say that $F'$ is the \emph{reflex field}
%of~$V$. 
A \emph{special point} is a special subvariety of dimension~$0$.

In fact, we can interpret each special subvariety~$V$ as a geometrically 
irreducible component of a \emph{Drinfeld modular subvariety} which is 
the union of Galois conjugates of $V$ over the corresponding
extension~$F'$ of $F$. A Drinfeld modular
subvariety~$X$ is the image of the composition of an above morphism 
defined by the restriction of 
Drinfeld~$A'$-modules to $A$ with a morphism given 
by a Hecke correspondence. Such a composition, called \emph{inclusion
morphism}, is associated to an extension $F' / F$ of the above type 
and an $\AZ_F^f$-linear isomorphism 
$b: (\AZ_F^f)^r \stackrel{\sim}{\rightarrow} (\AZ_{F'}^f)^{r'}$ encoding the
involved Hecke correspondence. We say that $F'$ is the \emph{reflex field} 
of $X$ and its geometrically irreducible components.

In~\cite{Pi}, Pink constructs the \emph{Satake
  compactification}~$\overline{S}_{F,\cal K}^r$ of a Drinfeld
modular variety~$S_{F,\cal K}^r$. It is characterized up to unique isomorphism
by a certain universal property. If ${\cal K}$ is sufficiently small 
in a certain sense, there is a natural ample invertible
sheaf~${\cal L}_{F,\cal K}^r$ on $\overline{S}_{F,\cal K}^r$. This
allows us to define the \emph{degree} of a subvariety of $S_{F,\cal K}^r$ 
as the degree of its Zariski closure in $\overline{S}_{F,\cal K}^r$ 
with respect to~${\cal L}_{F,\cal K}^r$. The degree of a subvariety can be seen 
as a measure for the ``complexity'' of the subvariety.

\subsection*{Andr\'e-Oort Conjecture for Drinfeld modular varieties}
The following analogue of the Andr\'e-Oort conjecture was formulated
by Breuer in~\cite{Br2}:

\begin{conjecture*}\label{conj*:AO}
Let $S = S_{F,\cal K}^r$ be a Drinfeld modular variety and $\Sigma$ 
a set of special points in $S$. Then each 
irreducible component over~$\C$ of the Zariski closure of $\Sigma$ 
is a special subvariety of~$S$.
\end{conjecture*}

Breuer~\cite{Br2} proved this analogue in two cases. Firstly, when the 
Zariski closure of $\Sigma$ is a curve, and secondly when all special
points in~$\Sigma$ have a certain behaviour at a fixed set of primes. %The
%second case includes the situation where all elements in~$\Sigma$ lie
%in the same Hecke orbit. 
Earlier in~\cite{Br1}, he proved an analogue of the
Andr\'e-Oort conjecture for products of modular curves in odd
characteristic. These proofs use an adaptation of the
methods of Edixhoven and Yafaev in~\cite{EdYa}, \cite{Ed2}
and~\cite{Ya}. The results are unconditional because GRH holds over
function fields.

In this thesis, we extend the arguments of Breuer using an adaptation 
of the methods
of Klingler and Yafaev in~\cite{KlYa}. Our main result is the
following theorem:

\begin{theorem*} \label{th*:main}
Conjecture~\ref{conj*:AO} is true if the reflex fields of all special points
in~$\Sigma$ are separable over~$F$.
\end{theorem*}

Since the reflex field of a
special point in a Drinfeld modular variety~$S_{F,\cal K}^r$ is of 
degree~$r$ over $F$, special
points with inseparable reflex field over $F$ can only occur if $r$ is
divisible by $p = \mathrm{char}(F)$. So Theorem~\ref{th*:main} implies

\begin{theorem*}
Conjecture~\ref{conj*:AO} is true if $r$ is not a multiple of $p = 
\mathrm{char}(F)$.
\end{theorem*}

%\subsection*{Drinfeld modular subvarieties}

%%%%%%%%%%%%%%%%%%%%%%%%%%%%%% old version %%%%%%%%%%%%%%%%%%%%%%%%%%%%%
%In~\cite{KlYa}, Klingler and Yafaev deal with sets~$\Sigma$ of special
%subvarieties of a certain dimension which are Zariski dense in a
%subvariety of a Shimura variety and perform an induction
%on the dimension of the elements of~$\Sigma$. We perform a similar
%induction, however it turned out to be more natural to work with 
%Galois orbits of special subvarieties over their reflex field
%instead of special subvarieties itself. We call these Galois
%orbits \emph{Drinfeld modular subvarieties} because
%we can interpret them as the image of an injective and finite morphism
%from some lower rank Drinfeld modular variety~$S_{F',\cal K'}^{r'}$ 
%into $S_{F,\cal K}^r$.

%%%%%%%%%%%%%%%%%%%%%%%%%%%%%%%%%%%%%%%%%%%%%%%%%%%%%%%%%%%%%%%
\subsection*{Sketch of the proof of Theorem~\ref{th*:main}}
\subsubsection*{First reductions}
We need to show that a geometrically irreducible subvariety~$Z$ of
$S_{F,\cal K}^r$ containing a Zariski dense subset of special points 
with separable reflex field over $F$ is a special subvariety. An 
induction argument shows that it is enough to show the following 
crucial statement:
\begin{theorem*} \label{th*:reduced}
Let $\Sigma$ be a set of Drinfeld modular subvarieties of $S$ of 
dimension~$d$ whose union is Zariski dense in a 
subvariety $Z \subset S$ of dimension~$> d$ 
which is defined and irreducible over $F$. Then, for
almost all $X \in \Sigma$, there is a Drinfeld modular subvariety $X'$ of $S$
with $X \subsetneq X' \subset Z$.
\end{theorem*}
In~\cite{KlYa}, Klingler and Yafaev perform the same induction, however they
work with special subvarieties instead of certain unions of their
Galois conjugates (Drinfeld modular subvarieties in our case).

In the proof of this statement, we can assume w.l.o.g. that

\begin{itemize}
\item
${\cal K}$ is sufficiently small such that the degree of
subvarieties of $S = S_{F,\cal K}^r$ is defined, and

\item
$Z$ is \emph{Hodge generic}, i.e., no geometrically
irreducible component of $Z$ is contained in a proper Drinfeld modular
subvariety of $S$.
\end{itemize}

\subsubsection*{Degree of Drinfeld modular subvarieties}
We give a classification of the Drinfeld modular subvarieties
of~$S$ and then use it to show the following unboundedness result:

\begin{theorem*} \label{th*:unbounded}
If $\Sigma$ is an infinite set of Drinfeld modular subvarieties of $S$,
then $\deg X$ is unbounded as $X$ varies over $\Sigma$.
\end{theorem*}

Note that, for a special subvariety~$V$ which is a geometrically irreducible
component of a Drinfeld modular subvariety~$X$, the union of the Galois
conjugates of $V$ over its reflex field is equal to $X$. Therefore,
$\deg X$ measures both the degree of~$V$ and the number of 
Galois conjugates of $V$. So our
unboundedness statement tells us that it is not possible that
in an infinite family of special subvarieties $V$, the degrees and
the number of Galois conjugates of $V$ are both bounded. Since we can
exclude this case, we
only need an adaptation of the Galois-theoretic and geometric methods
in~\cite{KlYa} and do not need equidistribution results as in~\cite{ClUl}.

\subsubsection*{Geometric criterion}
We deduce a geometric criterion
for $Z$ being equal to~$S$. It is a key ingredient of our proof 
of Theorem~\ref{th*:reduced} and says that $Z$ is equal to the
whole of $S$ provided that $Z$ is contained in a suitable Hecke
translate~$T_{g_{\pp}}Z$ of itself. A similar
geometric criterion appears in the proof of Klingler and Yafaev in the
classical case.

\begin{theorem*} \label{th*:geom}
Suppose that ${\cal K} = {\cal K}_{\pp} \times {\cal
  K}^{(\pp)}$ with ${\cal K}_{\pp} \subset \GL_r(F_{\pp})$ and assume
that $Z \subset T_{g_{\pp}}Z$ for some $g_{\pp} \in
\GL_r(F_{\pp})$ and $Z$ Hodge-generic and irreducible over $F$. If, for
all $k_1,k_2 \in {\cal K}_{\pp}$, the cyclic subgroup of
$\mathrm{PGL}_r(F_{\pp})$ generated by the image of $k_1 \cdot g_{\pp}
\cdot k_2$ is unbounded, then $Z = S$.
\end{theorem*}

The proof of this theorem is based on two results:

\begin{properties}
\item Zariski density

We define the $T_{h_{\pp}} + T_{h_{\pp}^{-1}}$-orbit of a geometric point $x \in
S(\C)$ to be the smallest subset of $S(\C)$ containing $x$ 
which is invariant under
$T_{h_{\pp}}$ and $T_{h_{\pp}^{-1}}$. We show
that the $T_{h_{\pp}} + T_{h_{\pp}^{-1}}$-orbit of an arbitrary point 
$x \in S(\C)$ is Zariski dense
in the geometrically irreducible component of $S$ containing $x$ provided
that $h_{\pp} \in \GL_r(F_{\pp})$ is chosen such that the cyclic subgroup of
$\mathrm{PGL}_r(F_{\pp})$ generated by the image of $h_{\pp}$ is unbounded.

\item 
A result of Pink~\cite[Theorem 0.1]{Pi3} on the Galois representations 
associated to Drinfeld modules implies that the image of the
arithmetic \'etale fundamental group of a geometrically irreducible
component of $Z$ is open in $\GL_r(F_{\pp})$, see Theorem 4
in~\cite{BrPi}. Here we need our assumption that $Z$
is Hodge-generic.
\end{properties}

\subsubsection*{Induction}
Our final step of the proof of Theorem~\ref{th*:reduced} consists of
an induction which uses a Hecke correspondence with specific
properties. By induction we show the following statement:

\begin{theorem*} \label{th*:induction}
Let $X$ be a Drinfeld modular subvariety of $S$ associated to $F' / F$ and 
$b: (\AZ_F^f)^r \stackrel{\sim}{\rightarrow} (\AZ_{F'}^f)^{r'}$ and
assume that $X$ is contained in a Hodge-generic subvariety 
$Z \subset S$ which is irreducible over $F$.

Suppose that $T_{g_{\pp}}$ is a Hecke correspondence localized at a
prime~$\pp$ with the following properties:
\begin{properties}
\item
$g_{\pp}$ is defined by some $g'_{\pp'} \in \GL_{r'}(F'_{\pp'})$ where $\pp'$
is a prime of $F'$ lying over~$\pp$, i.e., 
$g_{\pp} = b^{-1} \circ g'_{\pp'} \circ b$.

\item
$g_{\pp}$ satisfies the unboundedness condition in
Theorem~\ref{th*:geom}, i.e., $\cal K = {\cal K}_{\pp} \times {\cal K}^{(\pp)}$ 
with ${\cal K}_{\pp} \subset \GL_r(F_{\pp})$ and,
for all $k_1,k_2 \in {\cal K}_{\pp}$,
the cyclic subgroup of $\mathrm{PGL}_r(F_{\pp})$ generated by the 
image of $k_1 \cdot g_{\pp} \cdot k_2$ is unbounded,

\item
If $\iota: S' \rightarrow S$ is an inclusion morphism with $X \subset \iota(S')$,
then the Hecke correspondence $T'$ on $S'$ defined by $g'_{\pp'}$
satisfies (ii) and $\deg T' = \deg T_{g_{\pp}}$,

\item
$\deg X > \deg(T_{g_{\pp}})^{2^s-1} \cdot (\deg Z)^{2^s}$ for $s := \dim Z - \dim X$.
\end{properties}
Then there is a Drinfeld modular subvariety $X'$ of $S$ with $X
\subsetneq X' \subset Z$.
\end{theorem*}

We perform an induction over $s := \dim Z - \dim
X$. Property (i) implies that $X \subset T_{g_{\pp}}X$, in particular we therefore
have
\[ X \subset Z \cap T_{g_{\pp}}Z. \]
The lower bound (iv) for $\deg X$ now says that $X$ cannot be a union
of geometrically irreducible components of $Z \cap
T_{g_{\pp}}Z$. Therefore we find an irreducible component~$Z'$ over $F$ of
$Z \cap T_{g_{\pp}}Z$ with $X \subset Z'$ and $\dim Z' > \dim
X$. There are two cases:

If $Z' = Z$, we have $Z \subset T_{g_{\pp}}Z$ and conclude by
Theorem~\ref{th*:geom} that $Z = S$, so the conclusion of 
Theorem~\ref{th*:reduced} is true with $X' = S$.

If $Z' \subsetneq Z$, then $\dim Z' < \dim Z$ because $Z$ is
irreducible over $F$. We replace $Z$ by $Z'$ and apply
the induction hypothesis. In this step, it is possible that $Z'$ is not 
Hodge-generic any more. In this case, we replace $S$ by a
smaller Drinfeld modular variety~$S'$ and show that (i)-(iv) are
still valid in $S'$ using our property (iii).

\subsubsection*{Choice of a suitable Hecke correspondence}
To finish the proof of Theorem~\ref{th*:reduced}, by Theorem~\ref{th*:induction}
we need to show that, for almost
all $X \in \Sigma$, there is a Hecke correspondence $T_{g_{\pp}}$
localized at a prime~$\pp$ with the properties (i)-(iv). To construct
such a $T_{g_{\pp}}$ for a $X \in \Sigma$, we need the 
prime~$\pp$ to satisfy specific conditions under which we call 
the prime \emph{good} for $X$:

\begin{definition*} \label{def*:goodprime}
Let $X$ be a Drinfeld modular subvariety of $S_{F,\cal K}^r$
associated to $F' / F$ and $b: (\AZ_F^f)^r \stackrel{\sim}{\rightarrow}
(\AZ_{F'}^f)^{r'}$. A prime~$\pp$ of $F$ is called \emph{good} 
for $X \subset S_{F,\cal K}^r$ if there is an $s_{\pp} \in
\GL_r(F_{\pp})$ such that the following holds for the $A_{\pp}$-lattice
$\Lambda_{\pp} := s_{\pp} \cdot A_{\pp}^r$:

\begin{propertiesabc}
\item
${\cal K} = {\cal K}_{\pp} \times {\cal K}^{(\pp)}$ where ${\cal
  K}_{\pp} = s_{\pp}{\cal K}(\pp)s_{\pp}^{-1}$ for the principal
congruence subgroup ${\cal K}(\pp)$ of $\GL_r(A_{\pp})$, 

\item
$b_{\pp}(\Lambda_{\pp})$ is an $A' \otimes_A A_{\pp}$-submodule of $(A'
\otimes_A A_{\pp})^{r'}$, 

\item
there exists a prime~$\pp'$ of $F'$ above $\pp$ with local degree~$1$ over~$F$.
\end{propertiesabc}
\end{definition*}

\begin{theorem*} \label{th*:Hecke}
If $\pp$ is a good prime for a Drinfeld modular subvariety $X \subset
S_{F,\cal K}^r$, then there is a Hecke correspondence $T_{g_{\pp}}$
localized at $\pp$ satisfying (i)-(iii) from
Theorem~\ref{th*:induction} with
\[ \deg T_{g_{\pp}} = |k(\pp)|^{r-1}, \]
where $k(\pp)$ denotes the residue field of $\pp$.
\end{theorem*}
We show this theorem by defining
\[ g_{\pp} :=
s_{\pp}\mathrm{diag}(\pi_{\pp}^{-1},1,\ldots,1)s_{\pp}^{-1} \]
for a uniformizer $\pi_{\pp}$ at $\pp$. In the proof, it is crucial that ${\cal
  K}_{\pp}$ is \emph{not} a maximal compact subgroup of
$\GL_r(F_{\pp})$, which is guaranteed by condition (a) in the 
definition of good prime. Otherwise we are not able to satisfy 
the unboundedness condition (ii). 

However, (a) is a very strict
condition on the prime~$\pp$: For a fixed level~$\cal K$ it can only
be satisfied at most at a finite set of primes because $\cal K$ is
maximal compact at almost all primes. Since (b) and (c) are both
satisfied only for an infinite set of primes of density smaller
than one, for a fixed level~$\cal K$, in general we cannot find a
prime~$\pp$ satisfying (a)-(c). We get rid of this problem by 
starting with a prime~$\pp$
for which there is an $s_{\pp} \in \GL_r(F_{\pp})$ such that

\emph{(a')}\ \ \ ${\cal K} = s_{\pp}\GL_r(A_{\pp})s_{\pp}^{-1} \times {\cal K}^{(\pp)}$

and also (b) and (c) are satisfied. With an effective version 
of \v{C}ebotarev's theorem which relies on the correctness of GRH 
for function fields we can show that such a prime satisfying
an upper bound for $|k(\pp)|$ exists provided that $\deg X$ is large enough. 

In this situation we
consider the Drinfeld modular variety $\tilde{S} :=
S_{F,\tilde{\cal K}}^r$ with $\tilde{\cal K} = s_{\pp}{\cal
  K}(\pp)s_{\pp}^{-1} \times {\cal K}^{(\pp)}$
which is a finite cover of $S = S_{F,\cal K}^r$. The conditions
(a)-(c) from Definition~\ref{def*:goodprime} are 
satisfied for some Drinfeld modular subvariety $\tilde{X}$ of 
$\tilde{S}$ lying over $X$, i.e., $\pp$ is a good prime for $\tilde{X}
\subset \tilde{S}$. By Theorem~\ref{th*:Hecke}, we then find a 
Hecke correspondence~$T_{g_{\pp}}$ on $\tilde{S}$
localized at~$\pp$ satisfying (i)-(iv) from
Theorem~\ref{th*:induction} for $\tilde{X} \subset
\tilde{S}$ where (iv) is ensured by the above upper bound for
$|k(\pp)|$.

Since $\deg X$ is unbounded as $X$ ranges over $\Sigma$ by
Theorem~\ref{th*:unbounded}, this works for almost all $X \in
\Sigma$. For these $X$ Theorem~\ref{th*:induction} gives a 
Drinfeld modular subvariety~$\tilde{X'}$ of $\tilde{S}$ 
with $\tilde{X} \subsetneq \tilde{X'} \subsetneq \tilde{Z}$. 
The image $X' \subset S$ of $\tilde{X'}$ under the covering map 
$\tilde{S} \rightarrow S$ then satisfies the conclusion of 
Theorem~\ref{th*:reduced}. 

\subsection*{Difficulties in the inseparable case}
Unfortunately, the above methods do not work in the inseparable case,
i.e., if $\Sigma$ in Theorem~\ref{th*:reduced} contains Drinfeld
modular subvarieties of $S$ with inseparable reflex field. This is
caused by the fact that every prime ramifies in an inseparable field
extension. Therefore, for a Drinfeld modular subvariety with
inseparable reflex field, there is no prime for which condition (c) in
Definition~\ref{def*:goodprime} is satisfied. So we cannot apply
Theorem~\ref{th*:Hecke} to find a Hecke correspondence satisfying
(i)-(iii) from Theorem~\ref{th*:induction}.

Also other approaches to find such Hecke correspondences fail. For
example, if $X$ is a Drinfeld modular subvariety of dimension~$0$ 
with purely inseparable reflex field~$F' / F$ and $\pp$ any prime of
$F$, then a Hecke correspondence $T_{g_{\pp}}$ localized at $\pp$ 
satisfying (i) of Theorem~\ref{th*:induction} does not satisfy the
unboundedness condition~(ii) in Theorem~\ref{th*:induction}: 
%Indeed, since $\pp$ totally ramifies in
%$F'$, there is a unique prime~$\pp'$ of $F'$ above $\pp$, and 
Indeed, in this case there is exactly one prime $\pp'$ of $F'$ above
$\pp$ with ramification index~$r$ and, if $\pi_{\pp'}
\in F'_{\pp'}$ is a uniformizer, then
$1,\pi_{\pp'},\ldots,\pi_{\pp'}^{r-1}$ is an $F_{\pp}$-basis of
$F'_{\pp'}$. Therefore, if $g_{\pp} \in \GL_r(F_{\pp})$ is 
defined by $g'_{\pp'} = \pi_{\pp'}^k \in
\GL_1(F'_{\pp'})$ as in (i) of Theorem~\ref{th*:induction}, then $g_{\pp}$ is a conjugate of the matrix
\[ \left( \begin{array}{cccc} 
&&& \pi_{\pp} \\ 1 &&& \\ & \ddots && \\ &&1& \end{array}\right)^k \in \GL_r(F_{\pp}) \]
for $\pi_{\pp} := \pi_{\pp'}^r$. Its $r$-th power is a scalar matrix,
hence the cyclic subgroup of $\mathrm{PGL}_r(F_{\pp})$ generated by
the image of $g_{\pp}$ is bounded and we cannot apply our geometric
criterion (Theorem~\ref{th*:geom}) for the Hecke correspondence $T_{g_{\pp}}$.

\subsection*{Organization of the paper}
After discussing preliminaries in Section~\ref{ch:preliminaries}, we 
define \emph{Drinfeld modular varieties}
for arbitrary level ${\cal K} \subset \GL_r(\AZ_F^f)$ as quotients of
fine moduli schemes of Drinfeld modules in Section~\ref{ch:DMV}.

In Section~\ref{ch:morphisms}, we first define \emph{projection morphisms}
and \emph{Hecke correspondences} on Drinfeld modular varieties. Then we
define \emph{inclusion morphisms} of Drinfeld modular
varieties which allow us to
define \emph{Drinfeld modular subvarieties} and \emph{special
  subvarieties} of a Drinfeld modular variety~$S$. 
Subsequently, we show various properties of these morphisms, we give a 
classification of the Drinfeld modular subvarieties of~$S$ and 
describe the Galois action on the sets of Drinfeld modular
subvarieties and irreducible components of $S$.

In Section~\ref{ch:degree}, we define the \emph{degree} of subvarieties of a
Drinfeld modular variety using the Satake compactification constructed
in~\cite{Pi} and discuss some of its properties. We then show our
unboundedness statement for the degree of Drinfeld modular subvarieties 
(Theorem~\ref{th*:unbounded}).

The next two sections are devoted to the proof of our geometric
criterion for being a Drinfeld modular
subvariety (Theorem~\ref{th*:geom}). 
Section~\ref{ch:density} deals with Zariski density of
$T_g + T_{g^{-1}}$-orbits and in Section~\ref{ch:geom} we give the
proof of the actual criterion.

In Section~\ref{ch:choice}, we first define \emph{good primes} for
Drinfeld modular subvarieties. We then explain, for a fixed
Drinfeld modular subvariety, how we can find a suitable Hecke 
correspondence at a good prime as in Theorem~\ref{th*:Hecke}. The 
last subsection of Section~\ref{ch:choice}
is devoted to find a good prime~$\pp$ satisfying an upper bound 
for $|k(\pp)|$ for a given Drinfeld modular subvariety after passing
to a finite cover of $S$.

In Section~\ref{ch:induction}, we finally conclude the proof of 
Theorem~\ref{th*:reduced} by proving Theorem~\ref{th*:induction} and
applying the results of the previous sections. Here we also explain why
Theorem~\ref{th*:reduced} implies our main result (Theorem~\ref{th*:main}).

%\setcounter{section}{-1}
%\tableofcontents
\section{Preliminaries} \label{ch:preliminaries}
\subsection{Notation and conventions}
\label{sec:conv}
%\addcontentsline{toc}{subsection}{Notation and conventions}
%\setcounter{subsection}{1}

The following notation and conventions will be used throughout this paper:

\begin{itemize}
%\item
%$|M|$ denotes the cardinality of a set $M$.
\item
$\FZ_q$ denotes a fixed finite field with $q$ elements.
\item
For an $\FZ_q$-algebra $R$, we denote by $R\{\tau\}$ 
the ring of non-commutative polynomials in
the variable $\tau$ with coefficients in $R$ and the commutator rule
$\tau \lambda = \lambda^q \tau$ for $\lambda \in R$.
\item
$F$ always denotes a global function field of characteristic $p$ with field of
constants~$\FZ_q$ and $\infty$ a fixed place of $F$.
\item
For a pair $(F,\infty)$, we use the following notation:

\begin{tabular}{ll}
$A$ & ring of elements of $F$ regular outside $\infty$ \\
$F_{\pp}$ & completion of $F$ at a place~$\pp$ \\
$A_{\pp}$ & discrete valuation ring of $F_{\pp}$ \\
$k(\pp)$ & residue field of $\pp$\\
$\C$ & completion of an algebraic closure of $F_{\infty}$ \\
%$F^{\sep}$ & separable closure of $F$ inside $\C$ \\
$\AZ_F^f$ & ring of finite adeles of $F$ (i.e., adeles outside $\infty$)\\
$\AZ_F^{f,\,\pp}$ & ring of finite adeles of $F$ outside $\pp$ (i.e., adeles outside $\pp$ and $\infty$)\\
$\hat{A}$ & profinite completion $\prod_{\pp \neq \infty}A_{\pp}$ of $A$ \\
$\mathrm{Cl}(F)$ & class group of $A$ \\
%$g(F)$ & genus of $F$ \\
%$\Omega^r_F$ & Drinfeld's upper half-space of dimension $r-1$ over $F$
\end{tabular}
\item
A place $\pp \neq \infty$ of $F$ is said to be a \emph{prime} of
$F$. We
identify a prime $\pp$ of~$F$ with a prime ideal of $A$.

\item
For a place $\pp$ and a finite extension $F'$ of $F$ , we set $F'_{\pp} := F' \otimes_F
F_{\pp}$ and $A'_{\pp} := A' \otimes_A A_{\pp}$. 
We identify $F'_{\pp}$ with $\prod_{\pp' |
  \pp}{F'_{\pp'}}$ and $A'_{\pp}$ with $\prod_{\pp' |
  \pp}{A'_{\pp'}}$ via the canonical isomorphisms. For a second finite extension $F''$ of $F$, 
we use the analogous conventions and notations.

\item
For a subfield $K \subset \C$ we denote by $K^{\sep}$ the separable
and by $\overline{K}$ the algebraic closure of $K$ in $\C$. Each
$K$-automorphism of $K^{\sep}$ has a unique continuation to a
$K$-automorphism of $\overline{K}$. Therefore, we can and do identify
the absolute Galois group $G_K := \Gal(K^{\sep} / K)$ with the
automorphism group $\mathrm{Aut}_K(\overline{K})$.
\end{itemize}

For the formulation of algebro-geometric results, we use the following
conventions:

\begin{itemize}
\item
Unless otherwise stated, {\it variety} means a reduced separated scheme of
finite type over~$\C$ and {\it subvariety} means a reduced closed
subscheme of a variety. We identify the set $X(\C)$ of $\C$-valued points
of a variety~$X$ with the set of its closed points.

\item
For a subfield $K \subset \C$, a variety $X$ together with a
scheme~$X_0$ of finite type over $K$ and an isomorphism of schemes
$\alpha_X: X_{0,\,\C} \stackrel{\sim}{\rightarrow} X$ is called a {\it
variety over~$K$}. We often write $X$ in place of $(X, X_0, \alpha_X)$
and identify $X_{0,\,\C}$ with $X$ via $\alpha_X$ if this leads
to no confusion. Such a variety $X$ is called \emph{$K$-irreducible} if
$X_0$ is irreducible.
Note that a variety over $K$ is also a 
variety over~$K'$ if $K \subset K' \subset \C$.

\item
Let $X'=X'_{0,\,\C}$ and $X = X_{0,\,\C}$ be two varieties over $K$. A
morphism $X' \rightarrow X$ is said to be {\it defined over $K$} if it
is the base extension to~$\C$ of a morphism $X'_0 \rightarrow X_0$ of schemes 
over~$K$.

\item
For a variety $X$ over $K$, a subvariety $X'
  \hookrightarrow X$ is said to be \emph{defined over $K$} if $X'$ 
is a variety over $K$ and the closed
immersion $X' \hookrightarrow X$ is defined over $K$.

%For a variety $X = X_{0,\,\C}$ over $K$, a subvariety $X'$ together
%with a closed subscheme $X'_0$ of $X_0$ and an isomorphism 
%$X'_{0,\,\C} \stackrel{\sim}{\rightarrow} X'$ such that 
%$X'_{0,\,\C} \stackrel{\sim}{\rightarrow} X' \hookrightarrow X = X_{0,\,\C}$
%is the base extension to~$\C$ of $X'_0 \hookrightarrow X_0$ is said to be 
%{\it defined over $K$}. Again, we often only write $X'$ and identify
%$X'_{0,\,\C}$ and $X'$ via the given isomorphism.

\item
For a variety $X = X_{0,\,\C}$ over $K$ and a subfield $K' \subset \C$
containing $K$, we denote by $X(K')$ the set of $K'$-valued points of
$X_0$. Note that $X(K')$ is naturally a subset of the set of closed points of
$X$, in fact it is equal to the set of closed points of $X$ defined over
$K'$, see, e.g., p. 26 of~\cite{Bo}.

\item
The \emph{degree} of a finite surjective morphism $X \rightarrow Y$ of 
irreducible varieties is defined to be the degree of the 
extension of the function fields $\C(X) / \C(Y)$. We say that a finite
surjective morphism $f: X \rightarrow Y$ of (not necessarily irreducible)
varieties is of degree~$d$ if for each irreducible component $Z$ of $Y$
\begin{equation} \label{eq:finitedegree}
\sum_{\text{irr. components}\ X_i\ \text{of}\ f^{-1}(Z)}
\deg(f|X_i: X_i \rightarrow Z) = d. \end{equation}
%A general finite morphism $f: X \rightarrow Y$ of
%varieties is said to be of degree~$d$ if 
%$f: X \rightarrow f(X)$ is of degree $d$
%(note that $f(X)$ is a subvariety of $Y$ because finite morphisms are
%proper).
\end{itemize}
%\begin{satz} \label{prop:finitedegree}
%Let $f: X \rightarrow Y$ be a surjective finite morphism of varieties
%of degree~$d$. Then we have the equality $f_*([X]) = d \cdot [Y]$ of
%cycles on $Y$. If $f$ is in addition flat, then $f_*\mathcal{O}_X$ is a locally
%free $\mathcal{O}_Y$-module of rank~$d$.
%\end{satz}
%\begin{proof} The equality $f_*([X]) = d \cdot [Y]$ of cycles holds by
%  our definition of the degree of a surjective finite morphism (see,
%  e.g., Subsection 1.4 of~\cite{Fu} for the definition
%of the push-forward of cycles).

%If $f$ is in addition flat, then $f_*\mathcal{O}_X$ is a locally
%free $\mathcal{O}_Y$-module by Proposition III.9.2 (e)
%in~\cite{Ha}. By localization at the generic points of the irreducible
%components of $Y$, we see that $f_*\mathcal{O}_X$ is locally free 
%of rank~$d = \deg f$.
%\end{proof}

%\textbf{Remark:} We could also formulate our results in the language of
%classical algebraic geometry. However, it turns out to be more
%convenient to use the language of schemes instead.

%\subsection{Algebro-geometric preliminaries}
\subsection{Galois action on subvarieties}
\label{sec:galois}
Let $X = X_{0,\,\C}$ be a variety over $K \subset \C$. Then there is
a natural action of the absolute Galois group~$G_K$ on $X_{0,\,\overline{K}}$ which
induces an action of~$G_K$ on the set of subvarieties of $X$ which are defined 
over~$\overline{K}$.
\begin{satz}\label{prop:galois}
A subvariety of $X$ which is defined over $\overline{K}$ is already 
defined over $K$ if and only if it is defined
over $K^{\sep}$ and $G_K$-stable.
\end{satz}
\begin{proof} 
This follows from Theorem AG. 14.4 in~\cite{Bo}.
\end{proof}
%\hfill $\Box$

\begin{satz}\label{prop:Firred}
Let $X = X_{0,\,\C}$ be a variety over $K \subset \C$. Then the 
irreducible components
of $X$ are defined over $K^{\sep}$. The absolute Galois group acts transitively
on the set of irreducible components of $X$ if and only
if $X$ is $K$-irreducible.
\end{satz}
\begin{proof} Corollary 5.56 (2) in~\cite{GoWe} implies that the
irreducible components of $X$ are defined over $K^{\sep}$. The second statement
is a direct consequence of Proposition~\ref{prop:galois}.
\end{proof}

\section{Drinfeld modular varieties}
\label{ch:DMV}
\subsection{Analytic description and modular interpretation}
We consider the following datum:

\begin{itemize}
\item
A global function field $F$ together with a fixed place $\infty$,
\item
a positive integer $r$, called \emph{rank}, and
\item
a compact open subgroup $\cal K$ of $\GL_r(\AZ_F^f)$, called \emph{level}.
\end{itemize}

We define \emph{Drinfeld's upper half-space} over $F$ of
dimension~$r-1$ by
\[ \Omega_F^r := \PZ^{r-1}(\CZ_{\infty}) \setminus \{
F_{\infty}\textrm{-rational hyperplanes} \}. \]
\begin{satz}  \label{satz_Omegar}
The points of Drinfeld's upper half-space $\Omega_F^r$ are in bijective
correspondence with the set of injective $F_{\infty}$-linear maps
$F_{\infty}^r \hookrightarrow \CZ_{\infty}$ up to multiplication by a
constant in $\CS$ via the assignment
\[ [\omega_1 : \cdots : \omega_r] \longmapsto [(a_1,\ldots,a_r) \mapsto
a_1\omega_1 + \cdots + a_r\omega_r] .\]
\end{satz}
\begin{proof} We have the canonical bijection
\[  \begin{array}{rcl}
   \C^r & \longrightarrow & \{\FC\textrm{-linear maps }F_{\infty}^r
\rightarrow \CZ_{\infty}\} \\
   (\omega_1,\ldots,\omega_r) & \longmapsto & (a_1,\ldots,a_r) \mapsto
a_1\omega_1 + \cdots + a_r\omega_r
   \end{array}. \]
The $\FC$-linear map $(a_1,\ldots,a_r) \mapsto a_1\omega_1 + 
\cdots + a_r\omega_r$
is injective if and only if $\omega_1,\ldots,\omega_r$ are
$\FC$-linearly independent, i.e., if and only if $(\omega_1,\ldots,\omega_r)$
does not lie in a $\FC$-rational hyperplane. Hence, factoring out the
action of $\CS$ on both sides, we get the desired bijection of
Drinfeld's upper half-space with the set of injective $F_{\infty}$-linear maps
$F_{\infty}^r \hookrightarrow \CZ_{\infty}$ up to multiplication by a
constant in $\CS$.%\hfill$\Box$
\end{proof}

In the following, we use the identification given by
Proposition~\ref{satz_Omegar} and denote the element of $\Omega_F^r$ associated
to an injective $F_{\infty}$-linear map $\omega: F_{\infty}^r \hookrightarrow
\CZ_{\infty}$ by $\overline{\omega}$.

Using this notation, one sees that $\GL_r(F)$ acts on $\Omega_F^r$ from the
left by 
\begin{equation} T \cdot \overline{\omega} := \overline{\omega \circ
    T^{-1}}
\end{equation} 
for $T \in \GL_r(F)$ considered as automorphism of $F_{\infty}^r$. 

\textbf{Remark:} This action can also be described regarding $\Omega_F^r$ as
a subset of $\PZ^{r-1}(\C)$. A short calculation shows that, for
$\omega = [\omega_1 : \cdots : \omega_r] \in \Omega_F^r \subset
\PZ^{r-1}(\C)$ and $T \in \GL_r(F)$ with $T^{-1} = (s_{ij})$, we have
\begin{equation} \label{eq:actionGLrF}
T \cdot \omega = [s_{11}\omega_1 + \cdots + s_{r1}\omega_r : \cdots :
s_{1r}\omega_1 + \cdots + s_{rr}\omega_r].   
\end{equation}
In other words, the action of a $T \in \GL_r(F)$ on $\Omega_F^r
\subset \PZ^{r-1}(\C)$ is the restriction to $\Omega_F^r$ of the 
natural action of $(T^{-1})^T \in \GL_r(\C)$ on $\PZ^{r-1}(\C)$.

\begin{theorem} \label{th:Dmodulischeme}
There is a normal affine variety $S_{F,\cal K}^r$ of dimension $r-1$ over
$F$ together with an isomorphism
\begin{equation} \label{eq:rigid_iso}
 S_{F,\cal K}^r(\C) \cong \GL_r(F) \setminus (\Omega_F^r \times
\GL_r(\AZ_F^f) / \cal K) 
\end{equation}
of rigid-analytic spaces, where $\GL_r(\AZ_F^f) / \cal K$ is viewed as
a discrete set.
\end{theorem}
%In the following proof, we construct a variety $S_{F,\cal K}^r$ 
\textbf{Remarks:}
\begin{itemize}
\item
In the proof, we define a variety $S_{F,\cal K}^r$ over $F$ together
with a rigid-analytic isomorphism of the form~\eqref{eq:rigid_iso} up
to isomorphism over $F$. This variety is
called the \emph{Drinfeld modular variety} associated to the datum
$(F,r,\cal K)$. We will
identify its $\C$-valued points with double cosets in $\GL_r(F) \setminus
(\Omega_F^r \times \GL_r(\AZ_F^f) / \cal K)$ via the 
rigid-analytic isomorphism given in the proof.

\item
Later (Corollary~\ref{cor:nonsingular}), we will show that $S_{F,\cal K}^r$
is a non-singular variety if ${\cal K}$ is sufficiently small in a certain
sense.
\end{itemize}
%Note that the rigid-analytic structure of the double quotient 
%$\GL_r(F) \setminus (\Omega_F^r \times \GL_r(\AZ_F^f) / \cal K)$ is
%defined as in Subsection 6.B) of~\cite{Dr1}

\begin{proof}
The proof consists of several steps:

\begin{properties}
\item
We use Drinfeld's construction of Drinfeld moduli schemes in~\cite{Dr1}
to define $S_{F,{\cal K}}^r$ and a rigid-analytic
isomorphism of the form~\eqref{eq:rigid_iso} for ${\cal K} = {\cal
  K}(I) \subset \GL_r(\hat{A})$ a
principal congruence subgroup modulo a proper ideal $I$ of $A$.

\item
For $g \in \GL_r(\AZ_F^f) \cap \Mat_r(\hat{A})$ and proper ideals $I,\,J$ of $A$ with 
$J\hat{A}^r \subset gI\hat{A}^r$, we define
morphisms
\[ \pi_g: S_{F,{\cal K}(J)}^r \longrightarrow S_{F,{\cal K}(I)}^r, \]
which are defined over $F$ and satisfy the compatibility relation
\[ \pi_g \circ \pi_{g'} = \pi_{gg'}. \]
In particular, these morphisms define an action of $\GL_r(\hat{A})$ on
$S_{F,{\cal K}(I)}^r$.
%We show that $\GL_r(\hat{A})$ naturally acts on each $S_{F,{\cal
%    K}(I)}^r$ by morphisms defined over $F$.

\item
We use this action to extend the definition in (i) to all compact open
subgroups ${\cal K} \subset \GL_r(\hat{A})$.

\item
We extend the definition in (ii) to get morphisms
\[ \pi_g: S_{F,{\cal K}'}^r \longrightarrow S_{F,{\cal K}}^r \]
for arbitrary ${\cal K},\,{\cal K'} \subset
\GL_r(\hat{A})$ and $g \in \GL_r(\AZ_F^f)$ 
with ${\cal K}' \subset g^{-1}{\cal K}g$.

\item
We define $S_{F,{\cal K}}^r$ and a rigid-analytic isomorphism~$\beta$ of the 
form~\eqref{eq:rigid_iso} for arbitrary
levels ${\cal K} \subset \GL_r(\AZ_F^f)$. We use the morphisms $\pi_g$
from (iv) to show the well-definedness of $(S_{F,{\cal K}}^r,\,\beta)$
up to isomorphism over $F$.
% also arbritrary $\pi_g$
\end{properties}

%\begin{properties}
%\item
\textbf{Step (i):} Recall that a \emph{Drinfeld $A$-module of rank $r$} over an
$F$-scheme $S$ is a line bundle $\cal L$ over $S$ together with a ring
homomorphism $\phi$ from $A$ to the ring~$\mathrm{End}_{\FZ_q}(\cal
L)$ of $\FZ_q$-linear endomorphisms of $\cal L$ (as a group scheme over $S$)
such that, over any trivializing affine open subset 
$\mathrm{Spec}(B) \subset S$,
the homomorphism~$\phi$ is given by
\[ \phi: \begin{array}{rcl}
 A & \longrightarrow & \mathrm{End}_{\FZ_q}(\mathbb{G}_{a,\mathrm{Spec}(B)}) =
 B\{\tau\} \\
 a & \longmapsto & \phi_a = \sum_{i=0}^{m(a)} b_i(a) \tau^i 
\end{array} \]
where $\tau$ denotes the $q$-power Frobenius and, for all $a \in A$,
\begin{propertiesabc}
\item
$q^{m(a)} = |A/(a)|^r$,

\item
$b_{m(a)}(a) \in B^*$,

\item
$b_0(a) = \gamma(a)$ where $\gamma$ is the ring homomorphism $F
\rightarrow B$ corresponding to the morphism of affine schemes 
$\mathrm{Spec}(B) \hookrightarrow S \rightarrow \mathrm{Spec}(F)$.
\end{propertiesabc}
For a proper ideal $I$ of $A$, an \emph{$I$-level structure} on a Drinfeld
module ${\cal L} / S$ of rank~$r$ is an $A$-linear isomorphism of group schemes
over $S$
\[ \alpha: \underline{(I^{-1} / A)^r} \longrightarrow {\cal L}_I :=  \bigcap_{a \in I} \ker ({\cal L}
\stackrel{a}{\rightarrow} {\cal L}), \]
where $\underline{(I^{-1} / A)^r}$ denotes
the constant group scheme over $S$ with fibers $(I^{-1} / A)^r$.

\textbf{Remark:} In general, one can also define Drinfeld $A$-modules
together with level structures over $A$-schemes 
instead of $F$-schemes. In this case, one
uses a different definition of $I$-level structure to deal smoothly
with the fibers over $\pp \in \mathrm{Spec}(A)$ dividing $I$, see,
e.g., Section I.6 in~\cite{DeHu}.

%For a proper ideal $I$ of $A$, an \emph{$I$-level structure} on a Drinfeld
%module ${\cal L} / S$ of rank~$r$ is an $A$-linear morphism of group schemes
%over $S$
%\[ \alpha: \underline{(I^{-1} / A)^r} \longrightarrow {\cal L}_I :=  \bigcap_{a \in I} \ker ({\cal L}
%\stackrel{a}{\rightarrow} {\cal L}) \]
%such that we have the equality
%\[ \sum_{i \in (I^{-1} / A)^r} \alpha(i) = {\cal L}_I \]
%of divisors on ${\cal L}$ where $\underline{(I^{-1} / A)^r}$ denotes
%the constant group scheme over $S$ with fibers $(I^{-1} / A)^r$.

By Section 5 of~\cite{Dr1}, the functor
\[ {\cal F}_{F,I}^r:\ \ \ \begin{array}{rcl}
\text{$F$-schemes} & \longrightarrow & \text{Sets}
\\
S & \longmapsto & \{\text{Isomorphism classes of Drinfeld}\
A\text{-modules}\\
&& \text{of rank}\ r\
\text{over}\ S\ \text{with}\ I\text{-level structure}\}
\end{array} \]
is representable by a nonsingular affine scheme of finite type over
$F$ of dimension~$r-1$. Note that, in \cite{Dr1}, it is actually shown that the
corresponding functor from the category of schemes over
$\mathrm{Spec}\ A$ to the category of sets is representable if $I$ is
contained in two distinct maximal ideals of $A$. The argument in the
proof shows that it is enough that $I$ is a proper ideal of $A$ if we
work with schemes over $\mathrm{Spec}\ F$.

By our conventions in Subsection~\ref{sec:conv}, the base 
extension to $\C$ of the above representing scheme is a non-singular
variety of dimension~$r-1$ defined over $F$.
% NOTE: It is automatically reduced because it is nonsingular! This
% follows because each regular local ring is an integral domain (in a
% non-reduced scheme there are local rings with nilpotent elements)
We denote it by $S_{F,{\cal K}(I)}^r$, where ${\cal K}(I)$ denotes the
principal congruence subgroup modulo $I$. By Proposition 6.6 in~\cite{Dr1}, 
there is a natural isomorphism 
\begin{equation} \label{eq:rigid_iso2}
 S_{F,{\cal K}(I)}^r(\C) \cong \GL_r(F) \setminus (\Omega_F^r \times
 \GL_r(\AZ_F^f) / {\cal K}(I)) 
\end{equation}
of rigid-analytic spaces. Under this isomorphism, the equivalence class of an element
$(\overline{\omega},h) \in \Omega_F^r \times \GL_r(\AZ_F^f)$ is mapped
to the $\C$-valued point of $S_{F,{\cal K}(I)}^r$ corresponding to
the Drinfeld module over $\C$ associated to the lattice
\[ \Lambda := \omega(F^r \cap h \hat{A}^r) \]
with $I$-level structure given by the composition of the isomorphisms
\[ (I^{-1}/A)^r \stackrel{h}{\longrightarrow} 
I^{-1} \cdot (F^r \cap h \hat{A}^r) / (F^r \cap h\hat{A}^r)
\stackrel{\omega}{\longrightarrow} I^{-1}\cdot \Lambda / \Lambda, \]
where the first isomorphism is given by the multiplication by $h$ on
$(\AZ_F^f)^r$ via the natural identifications
\begin{eqnarray*}
(I^{-1} / A)^r & \cong & I^{-1}\hat{A}^r / \hat{A}^r \\
I^{-1} \cdot (F^r \cap h \hat{A}^r) / (F^r \cap h\hat{A}^r) & \cong &
I^{-1} \cdot h\hat{A}^r / h\hat{A}^r
\end{eqnarray*}
by the inclusion maps. For a detailed survey of this modular interpretation, 
we refer to the explanations in Section II.5 in~\cite{DeHu}.

%\setcounter{ctr}{1}
%\item
\textbf{Step (ii):} Let $I,\,J$ be proper ideals of $A$ and $g \in
\GL_r(\AZ_F^f) \cap \Mat_r(\hat{A}^r)$ such that $J\hat{A}^r \subset
gI\hat{A}^r$. For such a datum, we construct a morphism of functors
\[ {\cal F}_{F,J}^r \longrightarrow {\cal F}_{F,I}^r. \]
The given $g \in \GL_r(\AZ_F^f)$ with matrix entries in $\hat{A}$
induces a surjective endomorphism of $(\AZ_F^f)^r / \hat{A}^r$ with
kernel $g^{-1}\hat{A}^r / \hat{A}^r$. Since there is a natural
isomorphism $(F / A)^r \cong (\AZ_F^f / \hat{A})^r$ induced by the
inclusion maps, we therefore get a surjective homomorphism of $A$-modules
\[ (F / A)^r  \stackrel{g}{\longrightarrow} (F / A)^r. \]
The kernel $U := \ker g$ of this homomorphism is contained in $(J^{-1}
/ A)^r$ because we have $g^{-1}\hat{A}^r \subset J^{-1}I\hat{A}^r
\subset J^{-1}\hat{A}^r$ by our assumption $J\hat{A}^r \subset
gI\hat{A}^r$.

For any Drinfeld module $\cal{L}$ over an $F$-scheme $S$ with $J$-level
structure $\alpha: \underline{(J^{-1} / A)^r}
\stackrel{\sim}{\rightarrow} {\cal L}_J$, the image
of $\underline{U} \subset \underline{(J^{-1} / A)^r}$ under $\alpha$
is a finite $A$-invariant subgroup scheme of ${\cal L}$ over
$S$. Hence, the quotient ${\cal L}' := {\cal L} /
\alpha(\underline{U})$ is also a Drinfeld $A$-module over
$S$ and contains the finite subgroup scheme ${\cal L}_J /
\alpha(\underline{U})$. Since $g(J^{-1} / A)^r \cong (J^{-1} / A)^r /
U$, there is a unique $A$-linear isomorphism~$\alpha'$ of group schemes over
$S$ such that the diagram
\[ \xymatrix{ \underline{(J^{-1} / A)^r} \ar[r]^{\ \ \ \ \sim}_{\ \ \
    \ \alpha}
  \ar[d]^g & {\cal L}_J\ar[d]^{\pi} \\
\underline{g(J^{-1} / A)^r} \ar[r]^{\ \sim}_{\ \alpha'} & {\cal L}_J
/ \alpha(\underline{U}) }
\]
commutes, where $\pi: {\cal L}_J \rightarrow {\cal L}_J /
\alpha(\underline{U})$ is the canonical
projection. By the assumption $J\hat{A}^r \subset
gI\hat{A}^r$, we have $(I^{-1} / A)^r \subset g(J^{-1} /
A)^r$. Restricting the isomorphism~$\alpha'$ to the $I$-torsion gives therefore an $I$-level structure
\[ \underline{(I^{-1} / A)^r} \stackrel{\sim}{\longrightarrow} {\cal
  L}'_I \]
of ${\cal L}'$. % Note: By degree reasons (finite morphisms to S): The
                % I-torsion of L_J / alpha(U) is equal to the I-torsion
                % of L'

%Since $\alpha$ is a $J$-level structure on~${\cal L}$, we
%%therefore have the equality of divisors
%\[ {\cal L}'_I = \pi \left(\sum_{i \in g^{-1}((I^{-1} / A)^r)}
% \alpha(i) \right) \]
%on ${\cal L}'$ and by the commutativity of the above diagram also
%\[ {\cal L}'_I = \sum_{i \in (I^{-1} / A)^r} \alpha'(i). \]
%Hence $\alpha'$ restricted to $\underline{(I^{-1} / A)^r}$ is an
%$I$-level structure on ${\cal L}'$.

The assignment $({\cal L},\,\alpha) \rightarrow
({\cal L}',\,\alpha'|_{\underline{(I^{-1} / A)^r}})$ induces a morphism of
functors ${\cal F}_{F,J}^r \rightarrow {\cal F}_{F,I}^r$ and therefore
a morphism $\pi_g: S_{F,{\cal K}(J)}^r \rightarrow S_{F,{\cal K}(I)}^r$
defined over $F$.

A simple verification shows that $\pi_g$ is given by
\begin{equation} \label{eq:pig} 
%  \begin{array}{ccc} \GL_r(F) \setminus (\Omega_F^r \times
% \GL_r(\AZ_F^f) / {\cal K}(J)) & \longrightarrow & \GL_r(F) \setminus (\Omega_F%^r \times
% \GL_r(\AZ_F^f) / {\cal K}(I)) \\{}
[(\omega, h)]  \longmapsto  [(\omega,
hg^{-1})]. \end{equation}
on $\C$-valued points identified with double cosets via the isomorphisms~\eqref{eq:rigid_iso2}.

This description implies that we have the relation
\[ \pi_g \circ \pi_{g'} = \pi_{gg'} \]
for two such morphisms
\begin{eqnarray*}
\pi_g: S_{F,{\cal K}(I')}^r & \longrightarrow & S_{F,{\cal K}(I)}^r,  \\
\pi_{g'}: S_{F,{\cal K}(I'')}^r & \longrightarrow & S_{F,{\cal K}(I')}^r.
\end{eqnarray*}
In particular, we have an action of $\GL_r(\hat{A})$ on $S_{F,{\cal
    K}(I)}^r$ by morphisms defined over~$F$ and hence also on
isomorphism classes of Drinfeld $A$-modules with $I$-level structure.

%We introduce a left action of $\GL_r(\hat{A})$ on Drinfeld
%$A$-modules with $I$-level structure. For this, note that a $g \in
%\GL_r(\hat{A})$ acts on $(I^{-1} / A)^r$ by the composition
%\[ (I^{-1} / A)^r \cong (I^{-1} \cdot \hat{A}^r) / \hat{A}^r
%\stackrel{g}{\rightarrow} (I^{-1} \cdot \hat{A}^r) / \hat{A}^r \cong (I^{-1} /
%A)^r, \]
%where the first and last isomorphism are again induced by the
%corresponding inclusion maps. So we can define for $g \in
%\GL_r(\hat{A})$ and a Drinfeld $A$-module $({\cal L}, \phi)$ over $S$ 
%together with
%$I$-level structure $\alpha: (I^{-1} / A)^r \rightarrow {\cal L}(S)$:
%\[ g \cdot ({\cal L}, \phi, \alpha) := ({\cal L}, \phi, \alpha \circ
%g^{-1}). \]
%This defines a left action of $\GL_r(\hat{A})$ on 
%Drinfeld $A$-modules with $I$-level structure and therefore also one on
%$S_{F,{\cal K}(I)}^r$, where each element $g \in \GL_r(\hat{A})$ acts by a
%morphism defined over $F$. By the above modular interpretation, this
%action is given by
%\[ [(\overline{\omega},\,h)] \mapsto
%[(\overline{\omega},\,hg^{-1})] \]
%on $\C$-valued points.

%\item
\textbf{Step (iii):} Using the action of $\GL_r(\hat{A})$ on $S_{F,{\cal
    K}(I)}^r$ by the morphisms $\pi_g$, we define,
for a compact open subgroup ${\cal K} \subset
\GL_r(\hat{A})$,
\[ S_{F,\cal K}^r := S_{F,{\cal K}(I)}^r / {\cal K}, \] 
where ${\cal K}(I)$ is a principal congruence subgroup contained in
$\cal K$. Since ${\cal K}(I)$ acts trivially on $S_{F,{\cal K}(I)}^r$,
this quotient can be viewed as a quotient under the action of the
finite group ${\cal K} / {\cal K}(I)$ by morphisms defined over $F$. 
Hence, it is an affine variety 
defined over $F$ of dimension $r-1 = \dim S_{F,{\cal K}(I)}^r$ 
which is normal because $S_{F,{\cal K}(I)}^r$ is normal (see, e.g., 
Section III.12 in~\cite{Se}). By the description~\eqref{eq:pig} of 
the above action on $\C$-valued points, the rigid-analytic
isomorphism~\eqref{eq:rigid_iso2} induces one of the form
\begin{equation} \label{eq:rigid_iso3}
\beta_I: (S_{F,{\cal K}(I)}^r / {\cal K})(\C) \cong \GL_r(F) \setminus (\Omega_F^r \times
\GL_r(\AZ_F^f) / \cal K). \end{equation}

It remains to show that, up to $F$-isomorphism, $(S_{F,{\cal K}(I)}^r
/ {\cal K},\,\beta_I)$ is independent of the choice of~$I$. For this 
note that, for two ideals $I$, $J$ with $I
\subset J$, the functors
\begin{eqnarray*}
S & \longmapsto & {\cal F}_{F,I}^r(S) / {\cal K}(J),\\
S & \longmapsto & {\cal F}_{F,J}^r(S)
\end{eqnarray*}
are isomorphic, where the quotient is taken with respect to the action
of $\GL_r(\hat{A})$ on ${\cal F}_{F,I}^r(S)$. The isomorphism is given
by restricting $I$-level structures to $(J^{-1} / A)^r$.

Therefore, we have a natural isomorphism
\[ S_{F,{\cal K}(I)}^r / {\cal K}(J) \cong S_{F,{\cal K}(J)}^r \]
defined over $F$, which is compatible with the  
isomorphisms~\eqref{eq:rigid_iso3} and~\eqref{eq:rigid_iso2}.%, i.e., the diagram
%\[ \xymatrix{ (S^r_{F,{\cal K}(I)} / {\cal K}(J))(\C) 
%\ar[dr]^{\eqref{eq:rigid_iso3}} \ar[dd]^{\sim}& \\
%           &  \GL_r(F) \setminus (\Omega_F^r \times
%\GL_r(\AZ_F^f) / {\cal K}(J)) \\
%S^r_{F,{\cal K}(J)}(\C) \ar[ur]^{\eqref{eq:rigid_iso2}} } \]
%commutes.
% is induced by the morphism
%\[ \pi_1: S_{F,{\cal K}(I)}^r \longrightarrow S_{F,{\cal K}(J)}^r \]
%defined in Step (ii). 

So for two ideals $J,\,I$ with ${\cal K}(I)
\subset {\cal K}$ and ${\cal K}(J) \subset {\cal K}$ we have
\[ S_{F,{\cal K}(I)}^r / {\cal K} \cong S_{F,{\cal K}(I \cap J)}^r /
{\cal K} \cong S_{F,{\cal K}(J)}^r / {\cal K}, \]
and these isomorphisms are compatible with the 
isomorphisms~\eqref{eq:rigid_iso3}. Therefore, we can well-define
$S_{F,\cal K}^r$ up to isomorphism over $F$ by $S_{F,{\cal K}(I)}^r$
together with the rigid-analytic isomorphism~$\beta_I$.
%\end{properties}
 
\textbf{Step (iv):} Let $g \in \GL_r(\AZ_F^f) \cap \Mat_r(\hat{A})$ and ${\cal K}, {\cal
  K}' \subset \GL_r(\hat{A})$ with ${\cal K}' \subset g^{-1}{\cal
  K}g$ be given. Choose proper ideals $I$ and $J$ of $A$ such that
${\cal K}(I) \subset {\cal K}$, ${\cal K}(J) \subset {\cal K}'$ and
$J\hat{A}^r \subset gI\hat{A}^r$. Then, by Step (iii),
\begin{eqnarray*}
S_{F,{\cal K}'}^r & := & S_{F,{\cal K}(J)}^r / {\cal K}', \\
S_{F,{\cal K}}^r & := & S_{F,{\cal K}(I)}^r / {\cal K}
\end{eqnarray*}
and, by Step (ii), there is a morphism
\[ \pi_g: S_{F,{\cal K}(J)}^r \longrightarrow S_{F,{\cal K}(I)}^r. \]
Since $g{\cal K}'g^{-1} \subset {\cal K}$,
for each $k' \in {\cal K}'$, there is a $k \in {\cal K}$ such that
$gk' = kg$ and
\[ \pi_g \circ \pi_{k'} = \pi_k \circ \pi_g \]
as morphisms $S_{F,{\cal K}(J)}^r \longrightarrow S_{F,{\cal
    K}(I)}^r$. So the composition of $\pi_g$ with the canonical
projection $S_{F,{\cal K}(I)}^r \rightarrow S_{F,{\cal K}}^r$ is
${\cal K}'$-invariant and induces therefore a morphism $\pi_g: S_{F,{\cal
    K}'}^r \rightarrow S_{F,{\cal K}}^r$
such that the diagram
\[ \xymatrix{S_{F,{\cal K}(J)}^r \ar[r]^{\pi_g} \ar[d] & S_{F,{\cal
      K}(I)}^r \ar[d] \\
S_{F,{\cal K}'}^r \ar[r]^{\pi_g} & S_{F,{\cal K}}^r } \]
commutes, where the vertical maps are the canonical
projections. By~\eqref{eq:pig}, using the identifications $S_{F,\cal
  K}^r(\C)$ and $S_{F,{\cal K}'}^r(\C)$ with double coset spaces given 
by~\eqref{eq:rigid_iso3}, this morphism $\pi_g: S_{F,{\cal
    K}'}^r \rightarrow S_{F,{\cal K}}^r$ is given by
\begin{equation} \label{eq:descrC}
 [(\omega, h)] \longmapsto [(\omega, hg^{-1})] \end{equation}
on $\C$-valued points. Therefore, we have defined $\pi_g$
independently of the choice of $I$ and $J$ if all matrix entries of
$g$ lie in $\hat{A}$.
 
If $g \in \GL_r(\AZ_F^f)$ is arbitrary, there is a $\lambda
\in A \setminus \{0\}$ such that $\lambda \cdot g \in \GL_r(\AZ_F^f)
\cap \Mat_r(\hat{A})$. We then define $\pi_g := \pi_{\lambda \cdot
  g}$. This morphism is independent of the choice of $\lambda$ because
we have
\[ [(\omega, h(\lambda g)^{-1})] = [(\omega, hg^{-1})] \]
in $S_{F,\cal K}^r(\C)$ for all $\lambda \in A \setminus \{0\}$ and
$[(\omega, h)] \in S_{F,{\cal K}'}^r(\C)$. In particular, $\pi_g$ is
still described by~\eqref{eq:descrC} on $\C$-valued points.

The latter implies the relation
\begin{equation} \label{eq:pig_action}
   \pi_g \circ \pi_{g'} = \pi_{gg'} 
\end{equation}
for two such morphisms $\pi_g: S_{F,{\cal K}'}^r \rightarrow
S_{F,{\cal K}}^r$ and $\pi_{g'}: S_{F,{\cal K}''}^r \rightarrow
S_{F,{\cal K}'}^r$.

\textbf{Step (v):} For an arbitrary compact open subgroup
${\cal K} \subset \GL_r(\AZ_F^f)$, we choose a $g \in \GL_r(\AZ_F^f)$
such that $g{\cal K}g^{-1} \subset \GL_r(\hat{A})$. 
The composition of the rigid-analytic
isomorphism~\eqref{eq:rigid_iso3}
\[ S_{F,g{\cal K}g^{-1}}^r(\C) \cong \GL_r(F) \setminus
(\Omega_F^r \times \GL_r(\AZ_F^f) / g{\cal K}g^{-1}) \]
and $[(\overline{\omega},h)] \mapsto
[(\overline{\omega},hg)]$ gives a rigid-analytic isomorphism
\[ \beta_g: S_{F,g{\cal K}g^{-1}}^r(\C) \cong \GL_r(F) \setminus
(\Omega_F^r \times \GL_r(\AZ_F^f) / {\cal K}). \] % Bezeichung mit
                                % Index g?
For another $g' \in \GL_r(\AZ_F^f)$ with 
$g'{\cal K}g'^{-1} \subset \GL_r(\hat{A})$, the diagram
\[ \xymatrix{ S^r_{F,g{\cal K}g^{-1}}(\C) \ar[dr]^{\beta_g}_{\sim} 
\ar[dd]^{\pi_{g'g^{-1}}}_{\sim}& \\
           &  \GL_r(F) \setminus (\Omega_F^r \times
\GL_r(\AZ_F^f) / {\cal K}) \\
S^r_{F,g'{\cal K}g'^{-1}}(\C) \ar[ur]^{\beta_{g'}}_{\sim} } \] 
commutes. By the relation~\eqref{eq:pig_action}, the \
vertical arrow $\pi_{g'g^{-1}}$ is  
an isomorphism over $F$ with inverse $\pi_{gg'^{-1}}$.

Therefore, we can well-define $S_{F,\cal K}^r$ up to $F$-isomorphism
 as $S_{F,g{\cal K}g^{-1}}^r$ together with the rigid-analytic
isomorphism~$\beta_g$. Since we have seen in Step (iii) that
$S_{F,g{\cal K}g^{-1}}^r$ is a normal affine variety of
dimension~$r-1$ defined over $F$, the same holds for $S_{F,\cal K}^r$.
\end{proof}

\begin{satz} \label{theorem_irrcomp}
Let $C$ be a set of representatives in $\GL_r(\AZ_F^f)$ for $\GL_r(F) \setminus \GL_r(\AZ_F^f) / {\cal K}$, and set $\Gamma_g
  := g{\cal K}g^{-1} \cap \GL_r(F)$ for $g \in C$. Then the map
\begin{eqnarray*} 
  \coprod_{g \in C} \Gamma_g \setminus \Omega_F^r & \longrightarrow & \GL_r(F)
  \setminus (\Omega_F^r \times \GL_r(\AZ_F^f) / {\cal K}) \\{}
  [\overline{\omega}]_g & \longmapsto & [(\overline{\omega},\,g)]
\end{eqnarray*}
is a rigid analytic isomorphism which maps for each $g \in C$ the quotient
space $\Gamma_g \setminus \Omega_F^r$ to the $\C$-valued points of an
irreducible component~$Y_g$ of $S_{F,\cal K}^r$ over $\C$.
\end{satz}
This theorem implies that the irreducible components of $S_{F,\cal
  K}^r$ over $\C$ are disjoint and that $C$ 
is in bijective correspondence with the
set of irreducible components of $S_{F,\cal K}^r$ over $\C$ where $g
\in C$ corresponds to the irreducible component $Y_g$ with $Y_g(\C) \cong
\Gamma_g \setminus \Omega_F^r$ via the isomorphism given in the theorem. 
% and $C$ is in
%  bijective correspondence with the irreducible components of
%  $S_{F,\cal K}^r$ over $\C$ where $g \in C$ corresponds to $Y_h(\C)$  and the map 
%\begin{eqnarray*} 
%  \coprod_{g \in C} \Gamma_g \setminus \Omega_F^r & \longrightarrow & \GL_r(F)
%  \setminus (\Omega_F^r \times \GL_r(\AZ_F^f) / {\cal K}) \\{}
%  [\overline{\omega}]_g & \longmapsto & [(\overline{\omega},\,g)]
%\end{eqnarray*}
%is a rigid analytic isomorphism which maps for $g \in C$ the quotient
%space $\Gamma_g \setminus \Omega_F^r$ to the $\C$-valued points
%$Y_h(\C)$ of the irreducible component $Y_h$ the finitely many quotient 
%spaces $\Gamma_g \setminus \Omega^r$, $g \in C$, to the $\C$-valued
%points of the irreducible components of $S_{F,\cal K}^r$ over $\C$.
%\end{theorem}

\begin{proof}A direct calculation shows that the considered map is
  well-defined and bijective. Since $\GL_r(\AZ_F^f) / {\cal K}$ is
  viewed as a discrete set, the map
is also an isomorphism of rigid analytic spaces. 

Therefore, it only remains to show that the quotient spaces $\Gamma_g \setminus
\Omega_F^r,\ g \in C,$ are irreducible as rigid-analytic spaces because the
irreducible components of the rigid analytification of $S_{F,\cal K}^r$
coincide with the rigid analytification of the irreducible components
of $S_{F,\cal K}^r$ (see, e.g., Theorem 2.3.1 in~\cite{Co}). Since 
$S_{F,\cal K}^r$ is a normal variety and therefore its rigid
analytification a normal rigid analytic space, this is equivalent to the
connectedness of the quotient spaces $\Gamma_g \setminus \Omega_F^r$. 
The latter follows because $\Omega_F^r$ is a connected rigid-analytic
space by Theorem 2.4 in~\cite{Ko}. 
\end{proof} % see Conrad for def. of
            % irreducible comp. of rigid-analytic space (via normalization)

\begin{definition} \label{def_end}
For a $\C$-valued point $p = [(\overline{\omega},h)] \in 
S(\C)$ of a Drinfeld modular variety $S = S_{F,{\cal K}}^r$ 
with $h \in \GL_r(\AZ_F^f)$ and $\overline{\omega} \in
\Omega_F^r$ associated to $\omega: F_{\infty}^r \hookrightarrow \C$,
the elements of
\[ \End(p) := \{ u \in \C: u \cdot \omega(F^r) \subset
\omega(F^r) \} \]
are called \emph{endomorphisms} of $p$.
\end{definition}

Note that $\End(p)$ is well-defined because the homothety class of 
$\omega(F^r) \subset \C$ does not depend on the chosen 
representatives $\omega$ and $h$.

\textbf{Remark:} If ${\cal K} = {\cal K}(I)$ and $p \in S_{F,{\cal
    K}(I)}^r(\C)$ is corresponding to the Drinfeld module $\phi$ over
$\C$ associated to the lattice $\Lambda \subset \C$, then $\omega(F^r)
= F \cdot \Lambda$, and therefore
\[ \End(p) = F \cdot \End(\phi) \]
for the endomorphism ring $\End(\phi) \subset \C$ of $\phi$.

\begin{lemma} \label{lemma:end1}
The set $\End(p)$ of endomorphisms of $p$ is a field extension 
of~$F$ contained in $\C$ of finite degree dividing $r$ with only place above $\infty$.
\end{lemma}
\begin{proof} This follows from the argumentation in the proof of
Proposition 4.7.17 in~\cite{Go} noting that the endomorphism ring of a Drinfeld module
in generic characteristic is commutative.
\end{proof}
%Since $V := \omega(F^r) \subset \C$ is an
%$F$-subvectorspace of $\C$, the subset $\End(p) \subset \C$
% contains $F$
%and is closed under addition and multiplication. Furthermore,
%multiplicative inverses of elements $0 \neq u \in \End(p)$ 
%also lie in $\End(p)$ because
%dimension reasons imply that $x \mapsto u \cdot x$ is an automorphism
%of $V$ with inverse $x \mapsto u^{-1} \cdot x$. Therefore $\End(p)$ 
%is a field extension of $F$ contained in $\C$.
%Now take an element $0 \neq \xi \in V$. We have $x \cdot \xi \in V$ for all $x
%\in \End(p)$, hence $F \subset \End(p) \subset \xi^{-1}\cdot V$. This
%implies
%\[ r = \dim_F (\xi^{-1}\cdot V) = \dim_{\End(p)} (\xi^{-1}\cdot V) 
%  \cdot [\End(p) / F]. \]
%Therefore the field extension $\End(p) / F$ is finite with degree
%dividing $r$. 
%\end{proof}

\begin{lemma} \label{lemma_end}
Each irreducible component~$X$ of a Drinfeld modular variety   
$S_{F,\cal K}^r$ over $\C$ contains a point $p \in
X(\C)$ with $\End(p) = F$.
\end{lemma}
\begin{proof} Choose $\overline{\omega} \in \Omega_F^r$ such that
$\omega(F^r) = F \oplus F \cdot \xi_2 \oplus \cdots \oplus F \cdot
\xi_r$ with $\xi_2,\ldots,\xi_r \in \C$ algebraically independent over
$F$. This is possible because $\C$ as uncountable field is of infinite
transcendence degree over the countable field~$F$.

Now choose $h \in \GL_r(\AZ_F^f)$ such that $p :=
[(\overline{\omega},h)] \in X(\C)$ (use the description of the
irreducible components of $S_{F,\cal K}^r$ over $\C$ given in 
Proposition~\ref{theorem_irrcomp}).
Since $1 \in \omega(F^r)$, we have on the one hand $\End(p) \subseteq
\omega(F^r)$. On the other hand, all elements of
$\End(p)$ are algebraic over $F$ because the extension
$\End(p) / F$ is finite. But by the choice of
$\xi_2,\ldots,\xi_r$, every element of $\omega(F^r)$ which is
algebraic over $F$ lies in $F$. Hence, $\End(p) = F$.
\end{proof}

%For the second statement, note that the points $p \in S_{F,\cal
%  K}^r(\C)$ lying in a special subvariety of positive codimension are
%of the form $p = \incl(p')$ for some inclusion $\incl: S' \rightarrow S$
%and $p' \in S'(\C)$ with $F' \supsetneq F$. Hence, by
%Lemma~\ref{lemma:end1} and \ref{lemma:end2}, these points satisfy 
%$\End(p) = \End(p') \supset F' \supsetneq F$.

\subsection{Rank one case} \label{sec:rank1}
In the case $r = 1$ the variety $S_{F,\cal K}^r$ is
zero-dimensional and defined over $F$ 
for any compact open subgroup ${\cal K} \subset
\GL_1(\AZ_F^f) = (\AZ_F^f)^*$. Hence, $S_{F,\cal K}^1$ consists only of finitely
many closed points and it can be set-theoretically identified with
$S_{F,\cal K}^1(\C)$. By Proposition~\ref{prop:Firred}, the closed
points are all defined over $F^{\sep}$ and the absolute Galois group
$\Gal(F^{\sep}/F)$ acts on $S_{F,\cal K}^1$.

Drinfeld's upper half-space~$\Omega_F^1$ 
just consists of one point. Therefore, we have
\[ S_{F,\cal K}^1 = F^* \setminus (\AZ_F^f)^*\, /\, \cal K \]
as a set. Since $(\AZ_F^f)^*$ is abelian, this set can be identified
with the abelian group $(\AZ_F^f)^*\, /\,(F^* \cdot {\cal K})$.

Since $F^* \cdot {\cal K}$ is a closed subgroup of finite index of
$(\AZ_F^f)^*$, by class field theory, there is a finite 
abelian extension $H / F$ totally split at $\infty$ such that 
the Artin map
\[ \psi_{H / F}: (\AZ_F^f)^* \longrightarrow \Gal(H / F) \]
induces an isomorphism $(\AZ_F^f)^*\, /\,(F^* \cdot {\cal K}) 
\cong \Gal(H / F)$. In particular we have
\[ |S_{F,\cal K}^1| = [H : F]. \] 
%$F^* \setminus (\AZ_F^f)^*\, /\, \cal K$ can be 
%identified with $\Gal(H_{\cal K} / F)$ via the Artin isomorphism
% where $H_{\cal K}$ is the class field associated to the
%closed finite index subgroup $F^* \cdot {\cal K}$
%of the group $(\AZ_F^f)^*$ of finite ideles. 
%By class field theory, the class group of finite ideles $C_f := F^* \setminus (\AZ_F^f)^*$ can be
%identified with $\Gal (F_f^{\ab} / F)$ via the Artin isomorphism where
%$F_f^{\ab}$ denotes the maximal abelian extension of $F$ which
%completely splits in $\infty$. 

\begin{theorem} \label{th:CFT}
If $\psi_{H / F}(g) = \sigma|_H$ for a $g \in (\AZ_F^f)^*$ and a
$\sigma \in \Gal(F^{\sep} / F)$, then the action of~$\sigma$ on 
$S_{F,\cal K}^1 = F^* \setminus (\AZ_F^f)^*\, /\, 
\cal K$ is given by multiplication with~$g^{-1}$.
\end{theorem}
\begin{proof}
This follows from Theorem 1 in Section 8 of Drinfeld's
article~\cite{Dr1}. Note that in this article the
action of an element $g \in (\AZ_F^f)^*$ on $S_{F,\cal K}^1 = 
F^* \setminus (\AZ_F^f)^*\, /\, \cal K$ is given by the 
morphism~$\pi_g$, which is given by multiplication with $g^{-1}$.
\end{proof}

\begin{korollar} \label{cor:transgal}
The absolute Galois group $\Gal (F^{\sep} / F)$ acts transitively 
on~$S_{F,\cal K}^1$.\hfill$\qedsymbol$
\end{korollar}

% We will only use Corollary~\ref{cor:transgal} in the following sections.

\section{Morphisms and Drinfeld modular subvarieties} 
\label{ch:morphisms}

%%%%%%%%%%%%%%%%%%%%%%%%%%%%%%%%%%%%%%%%%%%%%%%%%%%%%%%%%%%%%%%%%%%%%%%%%%%%
\subsection{Projection morphisms and Hecke correspondences} \label{subsec:Hecke}
%%%%%%%%%%%%%%%%%%%%%%%%%%%%%%%%%%%%%%%%%%%%%%%%%%%%%%%%%%%%%%%%%%%%%%%%%%%%
Let $S_{F,\cal K}^r$ be a fixed Drinfeld modular variety. For each
  $g \in \GL_r(\AZ_F^f)$ and all compact open subgroups ${\cal{K'}} \subset g^{-1}{\cal
  K}g$ of $\GL_r(\AZ_F^f)$, we have a well-defined map
\begin{equation} \label{eq:pig2}
 \begin{array}{rcl} S_{F,{\cal K'}}^r(\C) & \rightarrow & S_{F,{\cal
      K}}^r(\C) \\{}
      [(\overline{\omega},h)] & \mapsto & [(\overline{\omega}, hg^{-1})]. 
   \end{array} \end{equation}

\begin{theorem} \label{th:proj}
This map is induced by a unique finite morphism $\pi_g: 
S_{F,{\cal K'}}^r \rightarrow S_{F,{\cal K}}^r$ defined over $F$.
% of degree $[g^{-1}{\cal K}g : {\cal K'}\cdot ({\cal K} \cap F^*)]$. 
\end{theorem}
\begin{proof}In the case that ${\cal K}$ and
  ${\cal K}'$ are contained in $\GL_r(\hat{A})$, we 
already showed the existence of a morphism $\pi_g$ which is described
  by~\eqref{eq:pig2} on $\C$-valued points in step (iv) of the 
proof of Theorem~\ref{th:Dmodulischeme}. If ${\cal K}$ and
  ${\cal K}'$ are arbitrary with ${\cal{K'}} \subset g^{-1}{\cal
  K}g$, there is an $s \in \GL_r(\AZ_F^f)$ with
\[ s{\cal K}'s^{-1} \subset sg^{-1}{\cal K}gs^{-1} \subset
\GL_r(\hat{A}). \]
By our definition in the proof of Theorem~\ref{th:Dmodulischeme},
we have $S_{F,{\cal K}'}^r = S_{F,s{\cal K}'s^{-1}}^r$, where under
the identifications of $\C$-valued points introduced in step (v) of the proof 
of Theorem~\ref{th:Dmodulischeme}
\[ [(\omega,h)] \in S_{F,{\cal K}'}^r(\C) \longleftrightarrow  
 [(\omega,hs^{-1})] \in S_{F,s{\cal K}'s^{-1}}^r(\C). \]
Similarly, we have $S_{F,{\cal K}}^r = S_{F,sg^{-1}{\cal K}gs^{-1}}^r$
with
\[ [(\omega,h)] \in S_{F,\cal K}^r(\C) \longleftrightarrow 
[(\omega,hgs^{-1})] \in S_{F,sg^{-1}{\cal K}gs^{-1}}^r(\C). \]
Using these identifications, we can define the morphism 
$\pi_g: S_{F,{\cal K}'}^r \rightarrow S_{F,{\cal K}}^r$ as 
$\pi_1: S_{F,s{\cal K}'s^{-1}}^r \rightarrow S_{F,sg^{-1}{\cal
    K}gs^{-1}}^r(\C)$. Since the latter morphism~$\pi_1$ is given by
$[(\omega,h)] \mapsto [(\omega, h)]$ on $\C$-valued points, 
by the above identifications $\pi_g$ is indeed
described by~\eqref{eq:pig2} on $\C$-valued points. So we have shown
the existence of the morphism~$\pi_g$ defined over $F$. It is uniquely 
determined by~\eqref{eq:pig2} because $\C$ is algebraically closed.

It remains to show finiteness of the morphism~$\pi_g$. By the above 
definition of a general morphism~$\pi_g$, it is enough to show it
for morphisms of the form $\pi_1: S_{F,{\cal K}'}^r \rightarrow
S_{F,{\cal K}}^r$ with ${\cal K}' \subset {\cal K} \subset
\GL_r(\hat{A})$.

We first assume that ${\cal K}' = {\cal K}(I)$ is a principal
congruence subgroup. Then $\pi_1$ is the canonical projection
\[ S_{F,{\cal K}(I)}^r \rightarrow S_{F,{\cal K}(I)}^r / {\cal K} \]
by the construction in the proof of
Theorem~\ref{th:Dmodulischeme}. Since ${\cal K}(I) \subset {\cal K}$
acts trivially on $S_{F,{\cal K}(I)}^r$, the quotient $S_{F,{\cal
    K}(I)}^r / {\cal K}$ can be viewed as a quotient under the action
of the finite group ${\cal K} / {\cal K}(I)$. Therefore, $\pi_1$ is
finite.

For general subgroups ${\cal K}' \subset {\cal K} \subset
\GL_r(\hat{A})$, choose a proper ideal $I$ of $A$ with ${\cal K}(I)
\subset {\cal K}'$. Then we have the following 
commutative diagram of projection maps:
\[ \xymatrix{S_{F,{\cal K}(I)}^r \ar[dd]^{\pi_1} \ar[dr]^{\pi_1} & \\
             & S_{F,{\cal K}'}^r \ar[dl]_{\pi_1} \\
             S_{F,\cal K}^r & }\]
We have already shown that the morphisms $\pi_1: S_{F,{\cal K}(I)}^r
\rightarrow S_{F,{\cal K}'}^r$ and $\pi_1: S_{F,{\cal K}(I)}^r
\rightarrow S_{F,{\cal K}}^r$ are finite. Therefore $\pi_1:
S_{F,{\cal K}'}^r \rightarrow S_{F,{\cal K}}^r$ is also finite.
\end{proof}

In the following, we call the morphisms $\pi_g$ \emph{projection
  morphisms} of Drinfeld modular varieties. In the case $g = 1$ we
also call them \emph{canonical projections} of Drinfeld modular
varieties. For two elements $g,g' \in \GL_r(\AZ_F^f)$ and two subgroups ${\cal{K'}} \subset g^{-1}{\cal
  K}g$, ${\cal{K''}} \subset {g'}^{-1}{\cal
  K'}g'$, by the description on
$\C$-valued points, we have
\begin{equation} \label{eq:projmor}
 \pi_{gg'} = \pi_g \circ \pi_{g'}. \end{equation} 

\begin{definition} \label{def:suff_small}
A compact open subgroup ${\cal K} \subset \GL_r(\AZ_F^f)$ is called
\emph{amply small} if there is a proper ideal $I$ of $A$ and a $g \in
\GL_r(\AZ_F^f)$ such that $g{\cal K}g^{-1}$ is contained in the
principal congruence subgroup ${\cal K}(I) \subset \GL_r(\hat{A})$.
%\[ g{\cal K}g^{-1} \subset \{t = (t_{\qq}) \in \GL_r(\hat{A})\,|\,
%t_{\pp} \equiv 1 (\mathrm{mod}\ \pp) \}. \]
\end{definition}

\begin{satz} \label{prop:projetale}
Let ${\cal K} \subset \GL_r(\AZ_F^f)$ be amply small, $g \in
\GL_r(\AZ_F^f)$ and ${\cal K'} \subset g^{-1}{\cal K}g$. Then the finite 
morphism $\pi_g: S_{F,{\cal K'}}^r \rightarrow S_{F,{\cal K}}^r$ is
\'etale of degree $[g^{-1}{\cal K}g : {\cal K}']$. Furthermore, 
if ${\cal K'}$ is a normal subgroup of $g^{-1}{\cal K}g$, 
it is an \'etale Galois cover over $F$ with 
group $g^{-1}{\cal K}g / {\cal K'}$ where the automorphism of the
cover corresponding to a coset $[x] \in g^{-1}{\cal K}g / {\cal K'}$
is given by $\pi_x: S_{F,{\cal K'}}^r \rightarrow S_{F,{\cal K'}}^r$.
%finite and \'etale of degree $[g^{-1}{\cal K}g : {\cal K'}]$. If
%${\cal K'}$ is moreover a normal subgroup of $g^{-1}{\cal K}g$, then
%$\pi_g$ is an \'etale Galois cover with group $
\end{satz}
\textbf{Remark:} In fact, the condition that some conjugate 
of ${\cal K}$ is contained in a principal congruence subgroup of
$\GL_r(\hat{A})$ could be weakened. Indeed it is enough that there is
a prime~$\pp$ such that the image of some conjugate of ${\cal K}$
in $\GL_r(A/\pp)$ is unipotent (cf. Proposition 1.5 in~\cite{Pi}).
\begin{proof}
Since ${\cal K}$ is amply small, there is an $h \in \GL_r(\AZ_F^f)$
and a proper ideal $I$ of $A$ such that $h^{-1}{\cal K}h \subset 
{\cal K}(I) \subset \GL_r(\hat{A})$. By the 
relation~\eqref{eq:projmor}, we have the commutative diagram
\[ \xymatrix{S_{F,h^{-1}g{\cal K}'g^{-1}h}^r \ar[r]^{\ \ \
    \pi_{g^{-1}h}}_{\ \ \ \ \sim} \ar[d]^{\pi_1}&
  S_{F,{\cal K}'}^r \ar[d]^{\pi_g} \\ S_{F,h^{-1}{\cal K}h}^r
  \ar[r]^{\pi_h}_{\sim} & S_{F,{\cal K}}^r } \]
where the horizontal morphisms are isomorphisms with $(\pi_h)^{-1} =
\pi_{h^{-1}}$ and $(\pi_{g^{-1}h})^{-1} = \pi_{h^{-1}g}$. Therefore, 
we can assume w.l.o.g. that $g = 1$ and ${\cal K'} \subset
{\cal K} \subset {\cal K}(I) \subset \GL_r(\hat{A})$.

\textbf{Case (i):} Let ${\cal K'}$ be a principal
congruence subgroup ${\cal K}(J)$ modulo a proper ideal $J$ of $A$,
i.e., ${\cal K'} = {\cal K}(J) \lhd {\cal K} \subset {\cal K}(I)$.

Then, by our definition in the proof of
Theorem~\ref{th:Dmodulischeme}, 
$\pi_1: S_{F,{\cal K}(J)}^r \rightarrow S_{F,{\cal K}}^r$ is the
canonical projection
\[ S_{F,{\cal K}(J)}^r \longrightarrow S_{F,{\cal K}(J)}^r / {\cal
  K}. \]
We show that ${\cal K} / {\cal K}(J)$ acts freely on the closed points
of $S_{F,{\cal K}(J)}^r$. This implies that
this projection is a finite \'etale morphism of degree~$[{\cal K} :
{\cal K}(J)]$ (see, e.g., Section II.7 in~\cite{Mu}). By the modular interpretation
of $S_{F,{\cal K}(J)}^r$ given in the proof of
Theorem~\ref{th:Dmodulischeme}, it is enough to show that the action of
${\cal K} / {\cal K}(J)$ on isomorphism classes of Drinfeld
$A$-modules over $\C$ together with $J$-level structure is
free. 

Indeed, assume that a coset $[k] \in {\cal K} / {\cal K}(J)$ stabilizes
the isomorphism class of the Drinfeld module $\phi$ over $\C$ associated to
a lattice $\Lambda \subset \C$ together with  
$J$-level structure $\alpha: (J^{-1} / A)^r 
\stackrel{\sim}{\rightarrow} J^{-1} \cdot \Lambda / \Lambda$. By our
definition of the action of $\GL_r(\hat{A})$ on Drinfeld modules with
$J$-level structure in the proof of Theorem~\ref{th:Dmodulischeme}, this means
that there is an automorphism~$c$ of $\phi$ under which the $J$-level structure $\alpha$ 
passes into $\alpha \circ k^{-1}$. Note that the restrictions of $\alpha$ and 
$\alpha \circ k^{-1}$ to $(I^{-1} / A)^r$ coincide because $k \in {\cal K}(I)$. 
Rigidity of Drinfeld modules with $I$-level structure (see, e.g., p. 30 in~\cite{Le})
therefore implies that $c$ is the identity. This is only possible if $k \in {\cal K}(J)$, i.e.
if $[k]$ is trivial in ${\cal K} / {\cal K}(J)$.

So we have shown that $\pi_1: S_{F,{\cal K}(J)}^r \rightarrow
S_{F,{\cal K}}^r = S_{F,{\cal K}(J)}^r / {\cal K}$ is an \'etale
cover of degree $[{\cal K} : {\cal K}(J)]$. The group 
${\cal K} / {\cal K}(J)$ injects into the
automorphism group over $F$ of this cover via $[k] \mapsto
\pi_k$. Since the degree of the cover is equal to $[{\cal K} : {\cal
  K}(J)]$ and $S_{F,{\cal K}(J)}^r$ is $F$-irreducible, the
automorphism group (over $F$) is therefore equal to ${\cal K} / {\cal K}(J)$. 
Furthermore, the cover is Galois because this group acts simply transitively 
on the geometric fibers.

\textbf{Case (ii):} Let ${\cal K}'$ be an arbitrary normal subgroup of
${\cal K}$, i.e., ${\cal K}' \lhd {\cal K} \subset {\cal K}(I)$.

Choose a proper ideal $J$ of $A$ such that 
${\cal K}(J) \subset {\cal K}'$ and note that
the diagram
\begin{equation} 
\xymatrix{S_{F,{\cal K}(J)}^r \ar[dr]^{\pi_1} \ar[dd]^{\pi_1} \\ & S_{F,{\cal
    K}'}^r = S_{F,{\cal K}(J)}^r / {\cal K}' \ar[dl]_{\pi_1} \\ 
S_{F,{\cal K}}^r = S_{F,{\cal K}(J)}^r / {\cal K} & }
\end{equation}
commutes. Since ${\cal K}'$ is normal in ${\cal K}$, the action of
${\cal K}$ on $S_{F,{\cal K}(J)}^r$ induces an action of ${\cal K} /
{\cal K}'$ on the quotient $S_{F,{\cal K}'}^r = S_{F,{\cal K}(J)}^r /
{\cal K}'$. By the commutativity of the diagram, 
the variety $S_{F,{\cal K}}^r$ is the quotient of $S_{F,{\cal K}'}^r$ 
under this action. Furthermore, this action is free on the closed
points of $S_{F,{\cal K}'}^r$ because ${\cal K}
/ {\cal K}(J)$ acts freely on the closed points of $S_{F,{\cal
    K}(J)}^r$. Therefore, we conclude by the same arguments as above
that $\pi_1: S_{F,{\cal K}'}^r \rightarrow S_{F,{\cal K}}^r$ is an
\'etale Galois cover of degree $[{\cal K} : {\cal K'}]$
with group ${\cal K} / {\cal K}'$ where the automorphism
of the cover corresponding to a coset $[k] \in {\cal K} / {\cal K}'$
is given by $\pi_k$. 

\textbf{Case (iii):} Let ${\cal K}'$ be an arbitrary subgroup of
${\cal K}$, i.e., ${\cal K}' \subset {\cal K} \subset {\cal K}(I)$.

As in case (ii) above, choose a proper ideal $J$ of $A$ such that ${\cal
  K}(J) \subset {\cal K}'$. The diagram above then also commutes and
$\pi_1: S_{F,{\cal K}(J)}^r \rightarrow S_{F,{\cal K}'}^r$ and $\pi_1:
S_{F,{\cal K}(J)}^r \rightarrow S_{F,{\cal K}}^r$ are surjective \'etale
morphisms by case (i). Furthermore, $S_{F,{\cal K}(J)}^r$ is a
non-singular variety as explained in step (i) of the proof of
Theorem~\ref{th:Dmodulischeme}.

Proposition 17.3.3.1 in EGA IV~\cite{EGAIV1} says that if $X \rightarrow Y$ is a flat,
surjective morphism of schemes and $X$ is regular, then $Y$ is also
regular. Therefore, $S_{F,{\cal K}}^r$ and $S_{F,{\cal K}'}^r$ are
both non-singular varieties. 

By Proposition 10.4 in~\cite{Ha}, a morphism $f: X \rightarrow Y$ 
of non-singular varieties of the same dimension
over an algebraically closed field is \'etale if and only if, for every
closed point $x \in X$, the induced map $T_x \rightarrow T_{f(x)}$ 
on Zariski tangent spaces is an isomorphism. We can apply this
criterion because $S_{F,{\cal K}(J)}^r$, $S_{F,{\cal K}}^r$ and
$S_{F,{\cal K}'}^r$ are all non-singular. Since the morphisms $\pi_1: S_{F,{\cal K}(J)}^r \rightarrow S_{F,{\cal K}'}^r$ and $\pi_1:
S_{F,{\cal K}(J)}^r \rightarrow S_{F,{\cal K}}^r$ are \'etale, the
commutativity of the above diagram therefore implies that
$\pi_1: S_{F,{\cal K}'}^r \rightarrow S_{F,{\cal K}}^r$ is \'etale and
finite of degree $[{\cal K} : {\cal K}(J)] / [{\cal K}' : {\cal K}(J)]
= [{\cal K}:{\cal K}']$. 
\end{proof}

\begin{korollar} \label{cor:nonsingular}
If ${\cal K} \subset \GL_r(\AZ_F^f)$ is amply small, then the
Drinfeld modular variety~$S_{F,\cal K}^r$ is non-singular.
\end{korollar}
\begin{proof} See case (iii) of the above proof of
Proposition~\ref{prop:projetale}.
\end{proof}

\begin{definition}[(Hecke correspondence)] \label{def:hecke}
For $g \in \GL_r(\AZ_F^f)$ and ${\cal K}_g := {\cal K} \cap
g^{-1}{\cal K}g$ the diagram
\[ \xymatrix{ & S_{F,{\cal K}_g}^r \ar[dl]_{\pi_1} \ar[dr]^{\pi_g} & \\
              S_{F,{\cal K}}^r & & S_{F,\cal K}^r }
\]
is called the \emph{Hecke correspondence} $T_g$ associated to $g$. For
subvarieties $Z \subset S_{F,\cal K}^r$ we define
\[ T_g(Z) := \pi_g ( \pi_1^{-1}(Z)). \]
\end{definition}
Note that $T_g(Z)$ is a subvariety of $S_{F,\cal K}^r$ for any 
subvariety $Z \subset S_{F,\cal K}^r$ because $\pi_g$ is finite and
hence proper. The integer
\[ \deg (T_g) := [{\cal K} : {\cal K} \cap g^{-1}{\cal K}g] \]
is called the \emph{degree} of the Hecke
correspondence~$T_g$. If ${\cal K}$ is amply small, by
Proposition~\ref{prop:projetale}, it is equal to $\deg \pi_1$.

%%%%%%%%%%%%%%%%%%%%%%%%%%%%%%%%%%%%%%%%%%%%%%%%%%%%%%%%%%%%%%%%%%%%%%%%%%%%%%%%%%%%%
\subsection{Inclusions of Drinfeld modular varieties} 
%%%%%%%%%%%%%%%%%%%%%%%%%%%%%%%%%%%%%%%%%%%%%%%%%%%%%%%%%%%%%%%%%%%%%%%%%%%%%%%%%%%%%

Let $S = S^r_{F,\cal K}$ be a given Drinfeld modular variety. We
consider the following datum:
\begin{itemize}
\item
A finite extension $F' \subset \CZ_{\infty}$ of $F$ of degree $r / r'$ 
for some integer $r' \geq 1$ with only one place~$\infty'$ lying over
 $\infty$,  and
\item
an $\AZ_F^f$-linear isomorphism $b: (\AZ_F^f)^r
\stackrel{\sim}{\rightarrow} (\AZ_{F'}^f)^{r'}$.
\end{itemize}
Note that the integral closure~$A'$ of $A$ in $F'$ is equal to the ring of elements of elements
of $F'$ regular away from $\infty'$ because $\infty'$ is the only place of $F'$
lying over $\infty$.

The above datum defines a subgroup
\[ {\cal K}' = (b{\cal K}b^{-1}) \cap \GL_{r'}(\AZ_{F'}^f) \]
of $\GL_{r'}(\AZ_{F'}^f)$.

\begin{lemma}\label{lemma:kprime}
The subgroup ${\cal K}'$ is compact and open in $\GL_{r'}(\AZ_{F'}^f)$. 
If ${\cal K}$ is amply small, it is also amply small.
\end{lemma}
\begin{proof}We fix an $\AZ_F^f$-linear isomorphism $b': (\AZ_F^f)^r
\stackrel{\sim}{\rightarrow} (\AZ_{F'}^f)^{r'}$ with $b'({\widehat{A}}^r) =
  \widehat{A'}^{r'}$ and set $g := b'^{-1} \circ b \in \GL_r(\AZ_F^f)$.
Since ${\cal K}$ is compact and open in $\GL_r(\AZ_F^f)$, there is a proper ideal~$I$
of $A$ such that the principal congruence subgroup~${\cal K}(I)$ is contained in $g{\cal K}g^{-1}$
with finite index. Therefore,
${\cal K}' = (b'g{\cal K}g^{-1}b'^{-1}) \cap \GL_{r'}(\AZ_{F'}^f)$ contains
the subgroup ${\cal K}'' = (b'{\cal K}(I)b'^{-1})\cap \GL_{r'}(\AZ_{F'}^f)$ with finite index. The latter subgroup
exactly consists of the elements of $\GL_{r'}(\AZ_{F'}^f)$ which stabilize $b'({\widehat{A}}^r) = \widehat{A'}^{r'}$ and
induce the identity on the quotient $\widehat{A'}^{r'} / I\cdot\widehat{A'}^{r'} \cong (A' /\, IA')^{r'}$. Hence, ${\cal K}''$
is the principal congruence subgroup modulo $IA'$ and ${\cal K}'$ is compact and open in $\GL_{r'}(\AZ_{F'}^f)$.

If ${\cal K}$ is
amply small, there is a proper ideal $I$ of $A$ and an $h \in
\GL_r(\AZ_F^f)$ such that
$h{\cal K}h^{-1} \subset {\cal K}(I)$. Therefore ${\cal K}'$ is contained in
the subgroup
\[ (bh^{-1}{\cal K}(I)hb^{-1}) \cap \GL_{r'}(\AZ_{F'}^f) \]
of $\GL_{r'}(\AZ_{F'}^f)$. This subgroup exactly consists of all 
elements of $\GL_{r'}(\AZ_{F'}^f)$ which stabilize $\Lambda := bh^{-1}(\hat{A}^r)
\subset (\AZ_{F'}^f)^{r'}$ and induce the identity on $\Lambda /
I \cdot \Lambda$. Since these elements are $\AZ_{F'}^f$-linear,
they also stabilize the $\widehat{A'}$-lattice $\Lambda' := \widehat{A'} \cdot
\Lambda$ and induce the identity on 
$\Lambda' / I \cdot \Lambda'$. Since $\Lambda'$ is a finitely
generated $\widehat{A'}$-submodule of $(\AZ_{F'}^f)^{r'}$ with
$\AZ_{F'}^f \cdot \Lambda' = (\AZ_{F'}^f)^{r'}$ and $\widehat{A'}$ is
a direct product of principal ideal domains, $\Lambda'$ is a free
$\widehat{A'}$-module of rank~$r'$. Hence, there is a $g' 
\in \GL_{r'}(\AZ_{F'}^f)$ such that $\Lambda' = g'\widehat{A'}^{r'}$
and
\[ {\cal K}' \subset (bh^{-1}{\cal K}(I)hb^{-1}) \cap \GL_{r'}(\AZ_{F'}^f)
\subset g'{\cal K}(I')g'^{-1} \]
for $I' := I A'$. This implies that ${\cal K}'$ is
amply small.
\end{proof}  

We choose an isomorphism
\[ \phi: F^r \stackrel{\sim}{\longrightarrow} {F'}^{r'} \]
of vector spaces over $F$. By scalar extension 
to $F_{\infty}$ and $\AZ_F^f$ it induces isomorphisms
\[ \begin{array}{rcl}
 F_{\infty}^r & \stackrel{\phi}{\longrightarrow} & {F'}_{\!\!\!\infty'}^{r'} \\
 (\AZ_F^f)^r & \stackrel{\phi}{\longrightarrow} & (\AZ_{F'}^f)^{r'}
\end{array},\]
which we also denote by $\phi$. We now define a morphism from the
lower-rank Drinfeld modular variety $S' = S_{F',\level}^{r'}$ into $S$:

\begin{theorem}\label{th_inclusion}
There is a finite morphism $\incl: S' \rightarrow S$ defined over $F'$
which on $\C$-valued points is given by the injective map
\begin{equation} 
\label{eq:incl2}
\begin{array}{ccc} S'(\C) &
  \longrightarrow & S(\C) \\{} 
  [(\overline{\omega'},h')] & \longmapsto & [(\overline{\omega' \circ
  \phi}, \phi^{-1} \circ h' \circ b)], \end{array} 
\end{equation}
where $\overline{\omega'} \in \Omega_{F'}^{r'}$ and 
$h' \in \GL_{r'}(\AZ_{F'}^f)$. The morphism $\incl$ is independent of
the choice of $\phi: F^r \stackrel{\sim}{\longrightarrow} {F'}^{r'}$.
\end{theorem}
\begin{proof} %We denote by $A'$ the integral closure of $A$ in
  %$F'$. Since there is only one place~$\infty'$ of $F'$ above
  %$\infty$, this is the ring of elements of $F'$ regular away from $\infty'$.
%
\textbf{Case (i):} We first consider the case where $b({\widehat{A}}^r) =
  \widehat{A'}^{r'}$ and ${\cal K} = {\cal
    K}(I)$ is a principal congruence subgroup modulo a proper ideal
  $I$ of $A$. In this case, $\level = 
(b{\cal K}b^{-1}) \cap \GL_{r'}(\AZ_{F'}^f)$ is the principal
congruence subgroup modulo $I' := IA'$ (see the proof of Lemma~\ref{lemma:kprime})
and $b$ induces an $A$-linear isomorphism $(I^{-1} / A)^r
\stackrel{\sim}{\longrightarrow} (I'^{-1} / A')^{r'}$, which we again
denote by $b$. 
Therefore, for a Drinfeld~$A'$-module $({\cal
  L},\,\psi)$ of rank~$r'$ over an $F'$-scheme with $I'$-level structure
\[ \alpha: \underline{(I'^{-1} / A')^{r'}} \stackrel{\sim}{\longrightarrow} {\cal L}_{I'}, \]
the restriction $({\cal L},\,\psi|_A)$ to $A \subset A'$ is a Drinfeld
$A$-module of rank $r = r' \cdot [F' / F]$ over $S$ and the composition
\[ \underline{(I^{-1} / A)^r} \stackrel{b}{\longrightarrow}
\underline{(I'^{-1} / A')^{r'}} \stackrel{\alpha}{\longrightarrow} {\cal
  L}_{I} \]
is an $I$-level structure on $({\cal L},\,\psi|_A)$ (note that the
$I$-torsion subgroup scheme ${\cal L}_I$ of~${\cal L}$ 
coincides with the $I'$-torsion subgroup scheme ${\cal L}_{I'}$
because $I$ generates $I'$ as an ideal of $A'$). The assignment
\begin{equation} \label{eq:incl_modint} ({\cal L},\, \psi,\, \alpha) \longmapsto ({\cal L},\,\psi|_A,
\alpha \circ b) \end{equation}
 defines a
morphism of functors from ${\cal F}_{F',I'}^{r'}$ to the restriction
of ${\cal F}_{F,I}^r$ to the subcategory of $F'$-schemes (see step (i)
of the proof
of Theorem~\ref{th:Dmodulischeme} for the definition of these
functors). Therefore, we have a morphism
\[ \incl: S_{F',{\cal K}(I')}^{r'} \longrightarrow S_{F,{\cal K}(I)}^r \]
defined over $F'$. By Lemma 3.1 and Proposition 3.2 in~\cite{Br2}, it
is a proper morphism which is injective on $\C$-valued
points. Since $S_{F',{\cal K}(I')}^{r'}$ and $S_{F,{\cal K}(I)}^r$
are both affine schemes of finite type over $\C$, the morphism $\incl$
is therefore a proper morphism of finite presentation with finite
fibers. This implies that $\incl$ is finite by Theorem 8.11.1 
of EGA IV~\cite{EGAIV3}.

Using
\[ \omega'({F'}^{r'} \cap h'\widehat{A'}^{r'}) = (\omega'
\circ \phi)(F^r \cap (\phi^{-1}\circ h' \circ b)\widehat{A}^r) \]
one sees that $\incl$ is given by~\eqref{eq:incl2} on $\C$-valued points.
%Note that, in the considered case, the definition of the
%morphism~$\incl$ is independent of the choice of $\phi$.

\textbf{Case (ii):} For $b: (\AZ_F^f)^r
\stackrel{\sim}{\rightarrow} (\AZ_{F'}^f)^{r'}$ and ${\cal K} \subset
\GL_r(\AZ_F^f)$ arbitrary, we choose 
\begin{itemize} 
\item a $g' \in \GL_{r'}(\AZ_{F'}^f)$ with $g'{\cal K}'g'^{-1} \subset
  \GL_{r'}(\widehat{A'})$,
\item
an $\AZ_F^f$-linear isomorphism $b': (\AZ_F^f)^r
\stackrel{\sim}{\rightarrow} (\AZ_{F'}^f)^{r'}$ with
$b'(\widehat{A}^r) = \widehat{A'}^{r'}$,
\item a proper ideal $I$ of $A$ with ${\cal K}(I) \subset g^{-1}{\cal K}g$,
  where $g := b^{-1} \circ
g'^{-1} \circ b' \in \GL_r(\AZ_F^f)$.
\end{itemize}
Then $g' \circ b = b' \circ g^{-1}$, hence
\[ g'{\cal K}'g'^{-1} = (b'g^{-1}{\cal K}gb'^{-1}) \cap \GL_{r'}(\AZ_{F'}^f)
\supset (b'{\cal K}(I)b'^{-1})\cap \GL_{r'}(\AZ_{F'}^f) = {\cal K}(I A'), \]
and by case (i) and Theorem~\ref{th:proj}, the composition of morphisms
\[ S_{F',{\cal K}(I A')}^{r'} \stackrel{\iota_{F,\,b'}^{F'}}{\longrightarrow}
S_{F,{\cal K}(I)}^r \stackrel{\pi_g}{\longrightarrow} S_{F,\cal
  K}^r \]
is defined and finite. Because of
\[ (g'{\cal K'}g'^{-1})b'g^{-1} = b'g^{-1}(b^{-1}{\cal K}'b) \subset
b'g^{-1}{\cal K} \]
this composition is invariant under the action of $g'{\cal K'}g'^{-1}$
on $S_{F',{\cal K}(I A')}^{r'}$. Hence, it induces a finite
morphism $f: S_{F',g'{\cal K}'g'^{-1}}^{r'} \rightarrow S_{F,\cal K}^r$
such that the diagram
\[ \xymatrix{ S_{F',{\cal K}(I A')}^{r'}
  \ar[r]^{\iota_{F,\,b'}^{F'}} \ar[d]^{\pi_1}&  S_{F,{\cal K}(I)}^r \ar[d]^{\pi_g} \\
S_{F',g'{\cal K}'g'^{-1}} \ar[r]^{\ \ \ \ f} & S_{F,\cal K}^r  } \]
commutes. We can now define $\incl := f \circ \pi_{g'}$, where
$\pi_{g'}: S_{F',{\cal K}'}^{r'} \rightarrow S_{F',g'{\cal
    K}'g'^{-1}}^{r'}$. For $[(\overline{\omega'},h')] \in S_{F',{\cal
    K}'}^{r'}(\C)$ we indeed have
\[ \incl([(\overline{\omega'},h')]) = [(\overline{\omega' \circ \phi},
\phi^{-1} \circ h'g'^{-1} \circ b' \circ g^{-1})] =
[(\overline{\omega' \circ \phi}, \phi^{-1} \circ h' \circ b)], \]
independently of the choice of $\phi: F^r \stackrel{\sim}{\rightarrow}
{F'}^{r'}$ and the representative $(\overline{\omega'},\,h') \in
\Omega_{F'}^{r'} \times \GL_{r'}(\AZ_{F'}^{r'})$. This also shows that
our definition of $\incl$ is independent of the choice of $g'$, $b'$
and $I$.

%The definition of the map is independent of the 
%chosen representative\linebreak
%$(\overline{\omega'},h') \in \Omega_{F'}^{r'} \times
%\GL_{r'}(\AZ_{F'}^f)$.
%Indeed, we
%have for another representative\linebreak $(T' \overline{\omega'},\,T'h'k')$
%where $T' \in \GL_{r'}(F')$ and $k' \in \level$
%\[ \overline{(\omega' \circ T'^{-1})
%  \circ \phi} = T \cdot \overline{\omega' \circ
%  \phi} \]
%with $T := \phi^{-1} \circ T' \circ \phi \in \GL_r(F)$ and
%\[ \phi^{-1} \circ T'h'k' \circ b = T(\phi^{-1} \circ h'
%  \circ b)k \]
%with $k := b^{-1} \circ k' \circ b \in {\cal K}$. 
%Hence, the map is well-defined.

%The map is independent of the choice of $\phi$ because another
%$F$-linear isomorphism $F^r \stackrel{\sim}{\rightarrow} {F'}^{r'}$
%is of the form $\phi' = \phi \circ T$ with an $T \in \GL_r(F)$ and
%\[ [(\overline{\omega' \circ \phi'}, \phi'^{-1} \circ h' \circ b)] =
%[(T^{-1} \cdot \overline{\omega' \circ \phi}, T^{-1}(\phi^{-1} \circ
%h' \circ b))] = [(\overline{\omega' \circ \phi}, \phi^{-1} \circ h' \circ b)]\]
%for all $(\overline{\omega'},h') \in \Omega_{F'}^{r'} \times
%\GL_{r'}(\AZ_{F'}^f)$.

It remains only to prove that $\incl$ is injective on $\C$-valued
points, i.e., that the map~\eqref{eq:incl2} is injective. For this,
consider two 
elements $[(\overline{\omega'_1}, h'_1)],\,[(\overline{\omega'_2}, h'_2)]$ of
$S_{F',\level}^{r'}(\C)$ with $\overline{\omega'_1},\,\overline{\omega'_2} \in
\Omega_{F'}^{r'}$ associated to $\omega'_1, \omega'_2: {F'_{\infty'}}^{\!\!\!\!\!r'}
 \hookrightarrow \C$ and $h'_1,\, h'_2 \in \GL_{r'}(\AZ_{F'}^f)$ which
 are mapped to the same element of
$S(\C)$. This means that there exist $T \in \GL_r(F)$ and
$k \in {\cal K}$ such that
\begin{properties}
\item
$\overline{\omega'_1 \circ \phi \circ T^{-1}} = \overline{\omega'_2
  \circ \phi}$,
\item
 $T(\phi^{-1} \circ h'_1 \circ b)k =
  \phi^{-1} \circ h'_2 \circ b$. 
\end{properties}

By (i) there is a $\rho \in \CS$ such that the diagram
\[ \xymatrix{
F_{\infty}^r \ar[r]^{\phi}_{\sim} \ar[d]^T &
{F'_{\infty'}}^{\!\!\!\!\!r'} \ar@{^{(}->}[r]^{\omega'_1} & \C
\ar[d]^{\rho} \\
F_{\infty}^r \ar[r]^{\phi}_{\sim} & {F'_{\infty'}}^{\!\!\!\!\!r'}
\ar@{^{(}->}[r]^{\omega'_2} & \C
}
 \]
commutes. Since the maps $\omega'_1,\,\omega'_2,\,\rho$ are injective
and $F'$-linear, this implies that the $F$-linear automorphism $T' := \phi \circ T \circ
\phi^{-1}$ of ${F'}^{r'}$ is also $F'$-linear
and lies in $\GL_{r'}(F')$. Thus, we have $T'\cdot
\overline{\omega'_1} = \overline{\omega'_2}$, i.e., $\overline{\omega'_1}$ and
$\overline{\omega'_2}$ lie in the same $\GL_{r'}(F')$-orbit.

Equation (ii) implies that $T'h'_1(b \circ k \circ b^{-1}) 
= h'_2$ in $\GL_{r'}(\AZ_{F'}^f)$. Since $h'_1, h'_2$ and
$T'$ all lie in $\GL_{r'}(\AZ_{F'}^f)$, we conclude that
\[ b \circ k \circ
b^{-1} \in \level = (b{\cal K}b^{-1}) \cap \GL_{r'}(\AZ_{F'}^f), \]
i.e., $[(\overline{\omega'_1}, h'_1)]=[(\overline{\omega'_2}, h'_2)]$
in $S_{F',\level}^{r'}(\C)$.
\end{proof}

Since the morphism~$\incl: S' \rightarrow S$ is injective 
on $\C$-valued points, we call it an \emph{inclusion} of Drinfeld 
modular varieties (by a slight abuse of terminology). 
If ${\cal K} \subset \GL_r(\AZ_F^f)$ is amply small (in the
sense of Definition~\ref{def:suff_small}), we can show that it is in
fact a closed immersion:

\begin{satz} \label{prop:closedimm}
Let $\incl: S_{F',{\cal K}'}^{r'} \rightarrow S_{F,\cal K}^r$ be an
inclusion of Drinfeld modular varieties with ${\cal K} \subset
\GL_r(\AZ_F^f)$ amply small. Then $\incl$ is a closed immersion of varieties.
\end{satz}
Before giving the proof of Proposition~\ref{prop:closedimm}, we summarize the
description of the tangent spaces at the closed points 
of a Drinfeld modular variety
$S_{F,{\cal K}}^r$ with ${\cal K} = {\cal K}(I)$ for a proper ideal
$I$ of $A$ given in~\cite{Ge2}.

We use for $a \in A$ the notation
\[ \deg a := \log_q(|A / (a)|) \]
and denote by $\C\{\{\tau\}\}$ the ring of formal non-commutative power
series in the variable~$\tau$ with coefficients in $\C$ and the 
commutator rule $\tau \lambda = \lambda^q \tau$ for $\lambda \in \C$. 

\begin{definition}
Let $\phi: A \rightarrow \C \{\tau\}$ be a Drinfeld module over
$\C$ of rank~$r$. An $\FZ_q$-linear map 
$\eta: A \rightarrow \tau \C\{\tau\}$ is called a \emph{derivation} with
respect to $\phi$ if, for all $a,\,b \in A$, the derivation rule
\[ \eta_{ab} = a\eta_b + \eta_a \circ \phi_b \]
is satisfied. Such a derivation is called \emph{reduced} 
resp. \emph{strictly reduced} if
it satisfies $\deg_{\tau}\eta_a \leq r \cdot \deg a$
resp. $\deg_{\tau}\eta_a < r \cdot \deg a$ for all $a \in A$. The
space of reduced resp. strictly reduced derivations $A \rightarrow
\tau \C\{\tau\}$ with respect to $\phi$ is denoted by $D_r(\phi)$
resp. $D_{sr}(\phi)$.
\end{definition}

\begin{theorem}
Let $x$ be a $\C$-valued point of $S_{F,{\cal K}(I)}^r$
corresponding to a Drinfeld $A$-module $\phi$ with $I$-level structure
$\alpha$. Then there is a natural isomorphism
\begin{equation} \label{eq:sr}
T_x(S_{F,{\cal K}(I)}^r) \stackrel{\sim}{\longrightarrow}
D_{sr}(\phi) \end{equation}
of vector spaces over $\C$.
\end{theorem}
\begin{proof}This follows from the discussion in the proof of Theorem
  6.11 in~\cite{Ge2} and the lemmata before this proof. 
\end{proof}

The isomorphism~\eqref{eq:sr}
 is given as follows: A tangent vector $\xi \in T_x(S_{F,{\cal
    K}(I)}^r)$ is an element of $S_{F,{\cal
    K}(I)}^r(\C[\varepsilon] / (\varepsilon^2))$ which projects to $x
\in S_{F,{\cal K}(I)}^r(\C)$ under the canonical projection
$\C[\varepsilon] / (\varepsilon^2) \rightarrow \C$. It corresponds to
the isomorphism class of a Drinfeld $A$-module over $\C[\varepsilon] /
(\varepsilon^2)$ with $I$-level structure which projects to $(\phi,
\alpha)$ under the canonical projection\linebreak $\C[\varepsilon] / 
(\varepsilon^2) \rightarrow \C$. There is a unique Drinfeld
$A$-module $\tilde{\phi}$ in this isomorphism class such that, 
for all $a \in A$,
\[ \tilde{\phi}_a = \phi_a + \varepsilon \cdot \eta_a \]
where $a \mapsto \eta_a$ is a strictly reduced derivation with
respect to $\phi$. The tangent vector $\xi$ is mapped to this strictly
reduced derivation under~\eqref{eq:sr}.

\begin{theorem}
Let $\phi$ be the Drinfeld $A$-module over $\C$ associated to an
$A$-lattice $\Lambda \subset \C$. Then there is a natural isomorphism
\begin{equation} \label{eq:derham}
 D_r(\phi) \stackrel{\sim}{\longrightarrow} \Hom_A(\Lambda, \C). 
\end{equation}
The $\C$-linear subspace $D_{sr}(\phi) \subset D_r(\phi)$ is mapped to a
subspace of $\Hom_A(\Lambda, \C)$ which is a complement of $\C \cdot
\id$, where $\id: \Lambda \hookrightarrow \C$ is the canonical inclusion.
\end{theorem}
\begin{proof} See Theorem 5.14 in~\cite{Ge1} and Theorem 6.10
  in~\cite{Ge1}.
\end{proof}
The isomorphism~\eqref{eq:derham} is called \emph{de Rham
  isomorphism} and can be described as follows: Let $\eta$
be a reduced derivation with respect to $\phi$. Then, for all
non-constant $a \in A$, there is a unique solution 
$F_{\eta} \in \C\{\{\tau\}\}$ satisfying the difference equation
\begin{equation} \label{eq:diffeq}
 F_{\eta}(az) - aF_{\eta}(z) = \eta_a (e_{\Lambda}(z)) 
\end{equation}
where 
\[ e_{\Lambda}(z) = z \cdot \prod_{0 \neq \lambda \in \Lambda}(1 - z
/ \lambda) \]
denotes the exponential function associated to the lattice
$\Lambda$. This solution is independent of the choice of $a \in A$ and
defines an entire function $\C \rightarrow \C$ 
which restricts to an $A$-linear map $\Lambda \rightarrow \C$. The
reduced derivation $\eta$ is mapped to $F_{\eta}|_{\Lambda}$
under~\eqref{eq:derham}.

\begin{proof}[Proof of Proposition~\ref{prop:closedimm}] 
We use the following criterion given
in Proposition 12.94 of~\cite{GoWe}:

\emph{A proper morphism $f: X \rightarrow Y$ of varieties 
over an algebraically closed field~$K$ is a closed immersion if and
only if the map $X(K) \rightarrow Y(K)$ induced by $f$ is injective
and, for all $x \in X(K)$, the induced map on Zariski tangent spaces
$T_x(X) \rightarrow T_{f(x)}(Y)$ is injective.}

Since finite morphisms are proper, by Theorem~\ref{th_inclusion} we
already know that $\incl$ is proper and injective on $\C$-valued
points. We therefore only
have to show that, for all $x \in S_{F',{\cal K}'}^{r'}(\C)$, the
induced map on Zariski tangent spaces ${\incl}_*: T_x(S_{F',{\cal K}'}^{r'})
\rightarrow T_{\incl(x)}(S_{F,\cal K}^r)$ is injective.

\textbf{Case (i):} 
As in the proof of Theorem~\ref{th_inclusion}, we first consider the
case where $b(\hat{A}^r) = \widehat{A'}^{r'}$ and ${\cal K} = {\cal
  K}(I)$ is a principal congruence subgroup modulo a proper ideal~$I$
of $A$. In this case, we have ${\cal K}' = {\cal K}(I')$ with $I' := I A'$. 
We can therefore use the description of the tangent spaces
given above.

Let $x \in S_{F',{\cal K}(I')}^{r'}(\C)$ be a point corresponding to the
  Drinfeld $A'$-module $\phi$ associated to an $A'$-lattice $\Lambda \subset
  \C$ of rank~$r'$ with $I'$-level structure. Since we defined $\incl$
  by restricting Drinfeld $A'$-modules to Drinfeld $A$-modules, the point 
$\incl(x) \in S_{F,\cal K}^r(\C)$ corresponds to the Drinfeld
$A$-module $\phi|_A$ associated to the same $\Lambda \subset \C$ considered as
$A$-lattice of rank~$r$ with some $I$-level structure. 
We can therefore consider the following diagram
\[ \xymatrix{
T_x(S_{F',{\cal K}(I')}^{r'}) \ar[r]^{\ \ \ \ \ \eqref{eq:sr}}_{\ \ \ \ \ \sim}
\ar[d]^{{\incl}_*} & D_{sr}(\phi) \ar[d] \ar@{^{(}->}[r]^{\!\!\!\!\!\!\eqref{eq:derham}} & 
\Hom_{A'}(\Lambda, \C) \ar@{^{(}->}[d] \\
T_{\incl(x)}(S_{F,\cal K}^r) \ar[r]^{\ \ \eqref{eq:sr}}_{\ \ \sim} &
D_{sr}(\phi|_A) \ar@{^{(}->}[r]^{\!\!\!\!\!\eqref{eq:derham}} & \Hom_A(\Lambda, \C)
} \] 
where the vertical arrow in the middle denotes the restriction of
derivations from $A'$ to $A$ and the one at the right 
the canonical inclusion. The left square of the diagram commutes by
the definition of~\eqref{eq:sr} because ${\incl}_*$ has the modular 
interpretation of restricting Drinfeld $A'$-modules 
over $\C[\varepsilon] / (\varepsilon^2)$ to $A$. The right square also
commutes because the unique solution of~\eqref{eq:diffeq} is
independent of $a \in A'$ and $\Lambda$ as an $A'$-lattice
 has the same exponential function as $\Lambda$ as an $A$-lattice.

Hence, the diagram commutes and, since the right vertical arrow is an
injective map, also the other two are injective maps. In particular,
the induced map ${\incl}_*$ between tangent spaces is injective.

\textbf{Case (ii):} Let $b: (\AZ_F^f)^r \stackrel{\sim}{\rightarrow} 
(\AZ_{F'}^f)^{r'}$ be arbitrary and ${\cal K} \subset
\GL_{r}(\AZ_F^f)$ be an arbitrary amply small subgroup. Then,
by the construction in the proof of Theorem~\ref{th_inclusion}, there is 
\begin{itemize} 
\item a $g' \in \GL_{r'}(\AZ_{F'}^f)$ with $g'{\cal K}'g'^{-1} \subset
  \GL_{r'}(\widehat{A'})$,
\item
an $\AZ_F^f$-linear isomorphism $b': (\AZ_F^f)^r
\stackrel{\sim}{\rightarrow} (\AZ_{F'}^f)^{r'}$ with
$b'(\widehat{A}^r) = \widehat{A'}^{r'}$,
\item a proper ideal $I$ of $A$ with ${\cal K}(I) \subset g^{-1}{\cal K}g$,
  where $g := b^{-1} \circ
g'^{-1} \circ b' \in \GL_r(\AZ_F^f)$
\end{itemize}
such that the diagram
\[ \xymatrix{ S_{F',{\cal K}(I')}^{r'}
  \ar[r]^{\iota_{F,\,b'}^{F'}} \ar[d]^{\pi_{g'^{-1}}} & S_{F,{\cal
      K}(I)}^r \ar[d]^{\pi_g} \\
  S_{F',{\cal K}'}^{r'} \ar[r]^{\incl} & S_{F,\cal K}^r } \]
with $I':= I A'$ commutes. By Proposition~\ref{prop:projetale}
and Corollary~\ref{cor:nonsingular}, 
the projection maps $\pi_{g'^{-1}}$ and $\pi_g$ in
this diagram are \'etale morphisms between non-singular varieties
because ${\cal K}'$ and ${\cal K}$ are amply small. Hence, they
induce isomorphisms on tangent spaces of closed points (Proposition
10.4 in~\cite{Ha}). By case (i), the upper horizontal
arrow~$\iota_{F,\,b'}^{F'}$ induces injections on tangent spaces of 
closed points. Therefore, the commutativity of the diagram implies
that, for all $x \in S_{F',{\cal K}'}^{r'}(\C)$, the
induced map ${\incl}_*: T_x(S_{F',{\cal K}'}^{r'})
\rightarrow T_{\incl(x)}(S_{F,\cal K}^r)$ is injective.
\end{proof}
 
\begin{satz} \label{prop:twoincl}
Let $\incla: S_{F',\cal K'}^{r'} \rightarrow S$ and 
$\inclb: S_{F'',\cal K''}^{r''} \rightarrow S$ be two
inclusions of Drinfeld modular varieties with $F'' \subset F'$. Then
for a $\AZ_{F''}^f$-linear isomorphism
$c: (\AZ_{F''}^f)^{r''} \rightarrow (\AZ_{F'}^f)^{r'}$ with
\begin{equation} \label{eq:incl}
b_1 = c \circ b_2 \circ k  
\end{equation}
for some $k \in {\cal K}$, the diagram
\[ \xymatrix{ 
S_{F',\cal K'}^{r'} \ar[dd]_{\iota^{F'}_{F''\!,\,c}} \ar[dr]^{\incla} & \\
& S \\
S_{F'',\cal K''}^{r''} \ar[ur]_{\inclb} &
} \]
commutes.
\end{satz}
\begin{proof} Note that we have
\[ {\cal K'} = (b{\cal K}b^{-1})
\cap \GL_{r'}(\AZ_{F'}^f) = (c{\cal
    K''}c^{-1}) \cap \GL_{r'}(\AZ_{F'}^f) \]
by the definition of $\cal K'$ and $\cal K''$ and
equation~\eqref{eq:incl}. Therefore, there is an inclusion 
$\iota^{F'}_{F''\!,\,c}: S_{F',\cal K'}^{r'} \rightarrow
S_{F'',\cal K''}^{r''}$.

The commutativity of the diagram follows by a direct calculation on
$\C$-valued points\linebreak using~\eqref{eq:incl2}.
\end{proof}

%%%%%%%%%%%%%%%%%%%%%%%%%%%%%%%%%%%%%%%%%%%%%%%%%%%%%%%%%%
\subsection{Drinfeld modular subvarieties}
%%%%%%%%%%%%%%%%%%%%%%%%%%%%%%%%%%%%%%%%%%%%%%%%%%%%%%%%%

The image of an inclusion $\incl: S' \rightarrow S$ of Drinfeld 
modular varieties is a subvariety of $S$ because finite morphisms are proper.

\begin{definition}
A subvariety of $S$ of the form $X = \incl(S')$ for an
inclusion~$\incl$ is called a \emph{Drinfeld modular
  subvariety} of $S$. An irreducible component of a Drinfeld modular
subvariety over $\C$ is called a \emph{special subvariety} and a
special subvariety of dimension~$0$ a \emph{special point}.
\end{definition}

\begin{lemma} \label{lemma:changeK}
Let $\tilde{\cal K} \subset {\cal K}$ be an open subgroup and $\pi_1:
S_{F,\tilde{\cal K}}^r \rightarrow S_{F,\cal K}^r$ the corresponding
canonical projection. Then the following holds:
\begin{properties}
\item
For each Drinfeld modular subvariety $X'
\subset S_{F,\tilde{\cal K}}^r$, the image $\pi_1(X')$ is a Drinfeld modular
subvariety of $S_{F,\cal K}^r$. 

\item
For each Drinfeld modular subvariety $X = \incl(S_{F',\level}^{r'})
\subset S_{F,\cal K}^r$, the pre\-image $\pi_1^{-1}(X)$ is a finite
union of Drinfeld modular subvarieties of $S_{F,\tilde{\cal K}}^r$.
%each $F'$-irreducible component of
%$\pi_1^{-1}(X)$ is a Drinfeld modular subvariety of $S_{F,\tilde{\cal K}}^r$.
\end{properties}
\end{lemma}
\begin{proof} For (i), assume that $X'$ is the image of the
inclusion $\tilde{\iota}_{F,\,b}^{F'}: S_{F',\tilde{\cal
    K'}}^{r'} \rightarrow S_{F,\tilde{\cal
     K}}^r$ associated to the datum $(F',\,b)$ and consider the
 inclusion morphism $\incl: S_{F',\cal K'}^{r'} \rightarrow S_{F,{\cal K}}^r$
 associated to the same datum. The diagram   
\[ \xymatrix{ 
 S_{F',\tilde{\cal K'}}^{r'} \ar[d]^{\pi'_1}
 \ar[r]^{\tilde{\iota}_{F,\,b}^{F'}} & S_{F,\tilde{\cal
     K}}^r \ar[d]^{\pi_1}
 \\
 S_{F',\cal K'}^{r'} \ar[r]^{\incl} & S_{F,{\cal K}}^r
} \]
with $\pi'_1$ and $\pi_1$ the respective canonical projections 
commutes by definition of the inclusion morphisms. Hence, 
\[ \pi_1(X') = \incl(\pi_1'(S_{F',\tilde{\cal K'}}^{r'})) =
\incl(S_{F',\cal K'}^{r'}) \]
is a Drinfeld modular subvariety of $S_{F,{\cal K}}^r$.

For (ii), choose a set of representatives $k_1,\ldots,k_l \in {\cal K}$
for the left cosets ${\cal K} / \tilde{\cal K}$ and consider
the inclusion morphisms $\iota_{F,\,b \circ k_i}^{F'}:
S_{F',\tilde{{\cal K}'_i}}^{r'} \rightarrow S_{F,\tilde{\cal K}}^r$ associated
to $(F',\,b \circ k_i)$ for $i = 1,\ldots,l$. By the definition of the
inclusion morphisms we have
\[ \pi_1^{-1}(X) = \bigcup_{i=1}^l
\iota_{F,\,b \circ k_i}^{F'}(S_{F',\tilde{{\cal K}'_i}}^{r'}), \]
hence $\pi_1^{-1}(X)$ is a finite union of Drinfeld modular
subvarieties of $S_{F,\tilde{\cal K}}^r$.
%where $\iota_{F,\,b \circ k_i}^{F'}(S_{F',\tilde{{\cal K}'_i}}^{r'})$
%is $F'$-irreducible for all $i$ by Corollary~\ref{cor:irrcomp}. 
%Hence, each $F'$-irreducible
%component of $\pi_1^{-1}(X)$ is of 
%the form $\iota_{F,\,b \circ k_i}^{F'}(S_{F',\tilde{{\cal
%      K}'_i}}^{r'})$ for some $i$ and therefore a Drinfeld modular
%subvariety of $S_{F,\tilde{\cal K}}^r$.
\end{proof}

\begin{lemma} \label{lemma:end2}
For an inclusion $\incl: S' \rightarrow S$, we have
\[ \End(p') = \End(\incl(p')) \]
for all $p' \in S'(\C)$.
\end{lemma}
\textbf{Remark:} This is an equality of subfields of $\C$ and not just
an abstract isomorphism of fields.
\begin{proof} This follows from our definitions because, for $p' = [(\overline{\omega'},h')] \in S'(\C)$, 
we have $\End(p') = \{ u \in \C : u \cdot 
\omega'({F'}^{r'}) \subset \omega'({F'}^{r'})\}$ and $\End(\incl(p')) = 
\{ u \in \C : u \cdot (\omega' \circ \phi)(F^r) \subset (\omega' \circ \phi)(F^r)\}$
for a chosen $F$-isomorphism $\phi: F^r \stackrel{\sim}{\rightarrow} {F'}^{r'}$.
\end{proof}

%This follows from the definition of~$\incl$ because 
%$(\omega' \circ \phi)(F^r) = \omega'({F'}^{r'})$ for 
%$\overline{\omega'} \in \Omega_{F'}^{r'}$ and
%an $F$-linear isomorphism $\phi: F^r \stackrel{\sim}{\rightarrow}
%{F'}^{r'}$.\end{proof}

% the diagram
%\[ \xymatrix{\GL_{r'}(F')  \times  \Omega_{F'}^{r'} \ar@{^{(}->}@<-4ex>[d]^{i_{\GL_{r'}(F')}}
%  \ar@{^{(}->}@<5ex>[d]^j \ar[r] &
%  \Omega_{F'}^{r'} \ar@{^{(}->}[d]^j \\
%   \GL_r(F)  \times \Omega_F^r  \ar[r] & \Omega_F^r  }
%\]
%commutes. Therefore, we have a well defined map

Now we give a criterion under which two Drinfeld modular subvarieties
are contained in each other.

\begin{satz} \label{prop:sub_contain}
Let $X' = \incla(S_{F',\level}^{r'})$ and $X'' =
\inclb(S_{F'',{\cal K}''}^{r''})$ be two Drinfeld modular subvarieties
of $S$. The following statements are equivalent:
\begin{properties}
\item
$X'$ is contained in $X''$.
\item
There is an irreducible component of $X'$ over $\C$ which is contained in~$X''$.
\item
$F'' \subset F'$ and there exist $k \in {\cal K}$ 
and an $\AZ_{F''}^f$-linear isomorphism
$c: (\AZ_{F''}^f)^{r''} \rightarrow (\AZ_{F'}^f)^{r'}$ such that $b_1
= c \circ b_2 \circ k$.
\end{properties}

\end{satz}
\begin{proof} We write $S' = S_{F',\level}^{r'}$ and $S'' =
S_{F'',{\cal K}''}^{r''}$.

The implication (i) $\Rightarrow$ (ii) is trivial and (iii)
$\Rightarrow$ (i) follows from Proposition~\ref{prop:twoincl}.

For (ii) $\Rightarrow$ (iii) assume that $\incla(Y') \subset
\inclb(S'')$ for an irreducible component $Y'$ of
$S'$ over $\C$. By Lemma~\ref{lemma_end} 
there is a $p'=[(\overline{\omega'},h')] 
\in Y'(\C)$ with $\End(p') = F'$. Now let $\incla(p') = \inclb(p'')$
for a suitable $p''=[(\overline{\omega''},h'')] \in
S''(\C)$. Lemma~\ref{lemma:end1} and \ref{lemma:end2} yield
\[ F' = \End(p') = \End(\incla(p')) = \End(\inclb(p'')) = \End(p'')
\supset F''. \]

Because of $\incla(p') = \inclb(p'')$ we have
\[ [(\overline{\omega' \circ \phi_1},\phi_1^{-1} \circ h' \circ
b_1)] = [(\overline{\omega'' \circ \phi_2},\phi_2^{-1} \circ h'' \circ
b_2)] \]
for $F$-linear isomorphisms $\phi_1: F^r
\stackrel{\sim}{\rightarrow} {F'}^{r'}$ and $\phi_2: F^r
\stackrel{\sim}{\rightarrow} {F''}^{r''}$. Hence, 
there are $T \in \GL_r(F)$ and $k \in {\cal K}$ such that
\begin{enumerate}
\item
$\overline{\omega' \circ \phi_1} = \overline{\omega'' \circ \phi_2
  \circ T^{-1}}$,
\item
$\phi_1^{-1} \circ h' \circ b_1 = T(\phi_2^{-1} \circ h'' \circ b_2)k$. 
\end{enumerate}
Because of 1. and $F'' \subset F'$, one concludes as in the proof of 
Theorem~\ref{th_inclusion}
that the $F$-linear isomorphism $\psi := \phi_1 \circ T \circ 
\phi_2^{-1}: F''^{r''} \rightarrow F'^{r'}$ is $F''$-linear.

We set $c := b_1 \circ k^{-1} \circ b_2^{-1}: (\AZ_{F''}^f)^{r''}
\rightarrow (\AZ_{F'}^f)^{r'}$. By 2. this is equal to
\[ c = h'^{-1} \circ \phi_1 \circ T \circ \phi_2^{-1} \circ h'' 
= h'^{-1} \circ \psi \circ h''. \]
Since $\psi$ is $F''$-linear and $F'' \subset F'$ we conclude that $c$
is an $\AZ_{F''}^f$-linear isomorphism. Furthermore, we 
have $b_1 = c \circ b_2 \circ k$ by the definition of $c$,
which shows~(iii).
\end{proof}

\begin{korollar}\label{cor:sub_sub}
Let $X' = \inclz(S_{F'',\cal K''}^{r''})$ be a fixed Drinfeld modular
subvariety of~$S$. Then the assignment
\[ X \longmapsto \inclz(X) \]
is a bijection from the set of Drinfeld modular subvarieties of
$S_{F'',\cal K''}^{r''}$ to the set of Drinfeld modular subvarieties
of $S$ contained in $X'$.
\end{korollar}
\begin{proof} Since $\inclz$ is injective on $\C$-valued points, it
is enough to show that
\begin{properties}
\item
$\inclz(X)$ is a Drinfeld modular subvariety of $S$ for each Drinfeld
modular subvariety $X$ of $S_{F'',\cal K''}^{r''}$,
\item
$(\inclz)^{-1}(X)$ is a Drinfeld modular subvariety of $S_{F'',\cal
  K''}^{r''}$ for every Drinfeld modular subvariety $X \subset X'$ of
$S$.
\end{properties}
For (i), let $X = \iota^{F'}_{F''\!,\,c}(S_{F',\level}^{r'})$ be a
Drinfeld modular subvariety of $S_{F'',\cal K''}^{r''}$. The 
map
\[ b := c \circ b': (\AZ_F^f)^r \rightarrow (\AZ_{F'}^f)^{r'} \]
is an $\AZ_F^f$-linear isomorphism, 
hence we can apply Proposition~\ref{prop:twoincl} 
to conclude that
\[ \inclz(X) = \inclz(\iota^{F'}_{F''\!,\,c}(S_{F',\level}^{r'})) =
\incl(S_{F',\level}^{r'}) \]
is a Drinfeld modular subvariety of $S_{F,\cal K}^r$.

For (ii), let $X = \incl(S_{F',\level}^{r'})$ be a Drinfeld modular
subvariety of $S$ which is contained in~$X'$. By 
Proposition~\ref{prop:sub_contain}, we have $F \subset F'' \subset F'$ and there
are an $\AZ_{F''}^f$-linear isomorphism $c: (\AZ_{F''}^f)^{r''}
\stackrel{\sim}{\rightarrow} (\AZ_{F'}^f)^{r'}$ and a $k \in {\cal K}$ such that
\[ b = c \circ b' \circ k. \]
By Proposition~\ref{prop:twoincl}, we have
\[ X = \incl(S_{F',\level}^{r'}) =
\inclz(\iota^{F'}_{F''\!,\,c}(S_{F',\level}^{r'})). \]
Since $\inclz$ is injective on $\C$-valued points, this implies that
$(\inclz)^{-1}(X) = \iota^{F'}_{F''\!,\,c}(S_{F',{\cal K'}}^{r'})$ is a Drinfeld
modular subvariety of $S_{F'',\cal K''}^{r''}$. 
\end{proof}

From Proposition~\ref{prop:sub_contain}, the following criterion for
equality of Drinfeld modular subvarieties follows:

\begin{korollar}\label{cor:sub_equal}
Let $X' = \incla(S_{F',\level}^{r'})$ and $X'' =
\inclb(S_{F'',{\cal K}''}^{r''})$ be two Drinfeld 
modular subvarieties of $S$. The following statements are equivalent:
\begin{properties}
\item
$X' = X''$.
\item
$X'$ and $X''$ have a common irreducible component over $\C$.
\item
$F' = F''$ (hence $r' = r''$) and there exist $s \in
\GL_{r'}(\AZ_{F'}^f)$ and $k \in {\cal K}$ such that $b_1
= s \circ b_2 \circ k$.
\end{properties}
In particular, each special subvariety of $S$ is an irreducible
component over $\C$ of a unique Drinfeld modular subvariety of $S$.\hfill $\qedsymbol$
\end{korollar} 

%\vspace{12pt}
\begin{korollar}
For a Drinfeld modular subvariety $X' \subset S$ there is a unique 
extension $F' \subset \C$ of $F$ and a unique conjugacy class $C$ of
compact open subgroups of $\GL_{r'}(\AZ_{F'}^f)$ with 
$r' = r / [F' / F]$ such that $F'' = F'$ and ${\cal K''} \in C$ for all
inclusions $\iota_{F,\,c}^{F''}: S_{F'',{\cal K''}}^{r''}
\rightarrow S$ with image $X'$.
\end{korollar}
\begin{proof} By definition, $X'$ is the image of some inclusion
$\incl: S_{F',\level}^{r'} \rightarrow S$. For any other
inclusion $\iota_{F,\,c}^{F''}: S_{F'',{\cal K''}}^{r''} 
\rightarrow S$ with image $X'$, Corollary~\ref{cor:sub_equal} implies that
$F''=F'$ and $b = s \circ c \circ k$ for
suitable $s \in \GL_{r'}(\AZ_{F'}^f)$ and $k \in {\cal K}$. The latter
implies $\level = s{\cal K}''s^{-1}$, i.e., ${\cal K''}$ lies in the
conjugacy class of $\level$ in $\GL_{r'}(\AZ_{F'}^f)$.  
\end{proof}

The preceding corollary allows us to make the following definition:
\begin{definition} \label{def:reflex_index}
For a Drinfeld modular subvariety $X' = \incl(S_{F',\level}^{r'})$ of
$S$, the extension $F' \subset \C$ of $F$ is called the
\emph{reflex field} of $X'$ and the index of $\level$ in a maximal
compact subgroup of $\GL_{r'}(\AZ_{F'}^f)$ is called the \emph{index} of
$X'$ and is denoted by~$i(X')$. Furthermore, the product
\[ D(X') := |\mathrm{Cl}(F')| \cdot i(X'), \]
where $\mathrm{Cl}(F')$ denotes the class group of $A' \subset F'$, is
called the \emph{predegree} of $X'$.
\end{definition}
By Corollary~\ref{cor:sub_equal}, each special subvariety of $S$ is an
irreducible component of a unique Drinfeld modular subvariety of
$S$. This allows us to define the reflex field of a special
subvariety:

\begin{definition}
For a special subvariety~$V$ of $S$ which is an irreducible component
of a Drinfeld modular subvariety $X'$ of $S$, the reflex field of $V$
is defined to be the reflex field of $X'$. 
\end{definition}

If ${\cal K} = \GL_r(\hat{A})$, Corollary~\ref{cor:sub_equal}
immediately implies the following characterization of the set of
Drinfeld modular subvarieties of $S$ with a given reflex field $F'$:
\begin{korollar} \label{prop:sub_lattice}
Assume that $S = S_{F,\cal K}^r$ with ${\cal K} = \GL_r(\hat{A})$ and let
$F' \subset \C$ be an extension of $F$ of degree $r / r'$ for some
integer $r' \geq 1$ with only one place $\infty'$ lying
over~$\infty$. Then the set of Drinfeld
modular subvarieties of $S$ with reflex field~$F'$ is in bijective 
correspondence with the set of orbits of the action of
$\GL_{r'}(\AZ_{F'}^f)$ on the set of free $\hat{A}$-submodules of rank $r$ of
$(\AZ_{F'}^f)^{r'}$ via the assignment
\[ \hspace{147pt}\incl(S') \longmapsto \GL_{r'}(\AZ_{F'}^f) \cdot 
b(\hat{A}^r).\hspace{147pt} \qed  \]%\hfill $\qedsymbol$
\end{korollar}

\begin{satz} \label{prop:GalFDmsv}
The natural action of the absolute Galois group $\Gal(F^{\sep} / F)$
on the set of subvarieties of $S = S_{F,\cal K}^r$ which are defined 
over~$\overline{F}$ restricts
to an action on the set of Drinfeld modular subvarieties of $S$. For
$\sigma \in \Gal(F^{\sep} / F)$ and a Drinfeld modular subvariety 
$X = \incl(S_{F',{\cal K}'}^{r'})$, the
Galois conjugate $\sigma(X)$ is given by 
$\iota_{F,\,\sigma \circ b}^{\sigma(F')}(S_{\sigma(F'),\sigma \circ
  {\cal K}' \circ \sigma^{-1}}^{r'})$.
\end{satz}
\textbf{Remark:} In the above formula for the Galois conjugate
$\sigma(X)$, the $\AZ_F^f$-linear isomorphism 
$(\AZ_{F'}^f)^{r'} \stackrel{\sim}{\rightarrow}
(\AZ_{\sigma(F')}^f)^{r'}$ obtained by tensoring $\sigma: F' 
\stackrel{\sim}{\rightarrow} \sigma(F')$ with $(\AZ_F^f)^{r'}$ over $F$ is also
denoted by $\sigma$.

%we used that the $F$-linear isomorphism $\sigma: F'
%\stackrel{\sim}{\rightarrow} \sigma(F')$ defines an $\AZ_F^f$-linear isomorphism 
%$\sigma: (\AZ_{F'}^f)^{r'} \stackrel{\sim}{\rightarrow}
%(\AZ_{\sigma(F')}^f)^{r'}$ after tensoring with $(\AZ_F^f)^{r'}$, which we also denote by $\sigma$.
\begin{proof}
%We have to show that, for every Drinfeld modular subvariety $X$ of $S$
%with reflex field $F'$, its
%Galois conjugates $\sigma(X)$ are Drinfeld modular subvarieties with
%reflex field $\sigma(F')$ for all $\sigma \in \Gal(F^{\sep} / F)$. 
As explained in Subsection~\ref{sec:conv}, we identify $\Gal(F^{\sep} / F)$ with
$\Aut_F(\overline{F})$ via the unique extension of the elements of
$\Gal(F^{\sep} / F)$ to $\overline{F}$.

\textbf{Case (i):} We first consider the case where $S = S_{F,{\cal
    K}(I)}^r$ for a proper ideal~$I$ of $A$ and $X = \incl(S_{F',{\cal
    K}'}^{r'})$ for an inclusion morphism $\incl$ associated to a
datum $(F',\,b)$ satisfying $b(\widehat{A}^r) = \widehat{A'}^{r'}$
with $A'$ the integral closure of $A'$ in $F'$. As
explained in the proof of Theorem~\ref{th_inclusion}, in
this case we have ${\cal K}' = {\cal K}(I')$ with $I' = I A'$
and $\incl$ is defined by the morphism~\eqref{eq:incl_modint} of functors from ${\cal
  F}_{F',I'}^{r'}$ to ${\cal
  F}_{F,I}^r$ (restricted to the subcategory of $F'$-schemes) using
the modular interpretation of $S_{F',{\cal K}(I')}^{r'}$ and $S_{F,{\cal
  K}(I)}^r$.

Note that, for any Drinfeld $A'$-module $\phi: A' \rightarrow
\overline{F}\{\tau\}$ over $\overline{F}$,
\[ \phi^{\sigma}: \begin{array}{ccc} 
\sigma(A') & \longrightarrow & \overline{F}\{\tau\} \\
\sigma(a') & \longmapsto & (\phi_{a'})^{\sigma} \end{array}, \]
where $(\phi_{a'})^{\sigma}$ is obtained from $\phi_{a'}$ by applying
$\sigma$ to its coefficients, is a Drinfeld $\sigma(A')$-module over
$\overline{F}$.  Furthermore, for any $I'$-level structure
$\alpha: (I'^{-1} / A')^{r'} \stackrel{\sim}{\rightarrow} \phi_{I'} \subset
{\overline{F}}$ on $\phi$, the composition
\[ (\sigma(I')^{-1} / \sigma(A'))^{r'} \stackrel{\sigma^{-1}}{\longrightarrow}
(I'^{-1} / A')^{r'} \stackrel{\alpha}{\longrightarrow} \phi_{I'}
\stackrel{\sigma}{\longrightarrow} (\phi^{\sigma})_{\sigma(I')} \]
is an $\sigma(I')$-level structure on $\phi^{\sigma}$. Using the
modular interpretation of $S_{F',{\cal K}(I')}^{r'}$ and
$S_{\sigma(F'),{\cal K}(\sigma(I'))}^{r'}$, the assignment
\[ (\phi,\,\alpha) \longmapsto (\phi^{\sigma},\,\sigma \circ \alpha
\circ \sigma^{-1}) \]
defines a map $g_{\sigma}: S_{F',{\cal K}(I')}^{r'}(\overline{F})
\rightarrow S_{\sigma(F'),{\cal
    K}(\sigma(I'))}^{r'}(\overline{F})$. By construction, the map
$g_{\sigma}$ is bijective with inverse $g_{\sigma^{-1}}$.

%The $F$-isomorphism $\sigma: F' \rightarrow \sigma(F')$ defines after
%tensoring with $\AZ_F^f$ an isomorphism $\AZ_{F'}^f
%\stackrel{\sim}{\rightarrow} \AZ_{\sigma(F')}^f$ and therefore by
%taking the diagonal an isomorphism $(\AZ_{F'}^f)^{r'} \stackrel{\sim}{\rightarrow}
%%(\AZ_{\sigma(F')}^f)^{r'}$, which we also denote by $\sigma$.
Note that we have $(\sigma \circ b)(\hat{A}^r) =
\widehat{\sigma(A')}^{r'}$. Hence
the datum~$(\sigma(F'), \sigma \circ b)$ defines an inclusion map
\[ \iota^{\sigma(F')}_{F,\,\sigma \circ b}: S_{\sigma(F'),{\cal
      K}(\sigma(I'))}^{r'} \longrightarrow S_{F,{\cal K}(I)}^r, \]
which is defined by a morphism of functors from ${\cal
  F}_{\sigma(F'),\sigma(I')}^{r'}$ to ${\cal
  F}_{F,I}^r$ (restricted to the subcategory of $\sigma(F')$-schemes).
A straightforward verification shows that the diagram
\[ \xymatrix{S_{F',{\cal K}(I')}^{r'}(\overline{F}) \ar[r]^{\incl}
  \ar[d]^{g_{\sigma}} & S_{F,{\cal K}(I)}^r(\overline{F})
  \ar[d]^{\sigma} \\ S_{\sigma(F'),{\cal
      K}(\sigma(I'))}^{r'}(\overline{F})
  \ar[r]^{\ \ \iota^{\sigma(F')}_{F,\,\sigma \circ b}} & S_{F,{\cal
        K}(I)}^r(\overline{F}) } \]
commutes, where the right vertical map is given by the natural action of
$\sigma$ on the closed points of $S_{F,{\cal K}(I)}^r$ defined over
$\overline{F}$.

Since, for any subvariety $Y \subset S$ defined over $\overline{F}$, 
the set $Y(\overline{F})$ of $\overline{F}$-valued points (viewed as a
subset of the closed points of $Y \subset S$) is Zariski dense in~$Y$
(see, e.g., Corollary AG. 13.3 in~\cite{Bo}), the commutativity of the
above diagram implies that
\[ \sigma(X) = \iota^{\sigma(F')}_{F,\,\sigma \circ
  b}(g_{\sigma}(S_{F',{\cal K}(I')}^{r'})) = 
  \iota^{\sigma(F')}_{F,\,\sigma \circ b}(S_{\sigma(F'),{\cal
      K}(\sigma(I'))}^{r'}). \]
for $X = \incl(S_{F',{\cal K}(I')}^{r'})$. Hence, $\sigma(X)$ is a Drinfeld
modular subvariety of $S$ and it is of the desired form because of $\sigma \circ {\cal K}' \circ
\sigma^{-1} = {\cal K}(\sigma(I'))$.

\textbf{Case (ii):} For a general $X = \incl(S_{F',{\cal K}'}^{r'}) \subset
S_{F,\cal K}^r$, by the construction in the proof of
Theorem~\ref{th_inclusion}, there is
\begin{itemize} 
\item a $g' \in \GL_{r'}(\AZ_{F'}^f)$ with $g'{\cal K}'g'^{-1} \subset
  \GL_{r'}(\widehat{A'})$,
\item
an $\AZ_F^f$-linear isomorphism $b': (\AZ_F^f)^r
\stackrel{\sim}{\rightarrow} (\AZ_{F'}^f)^{r'}$ with
$b'(\widehat{A}^r) = \widehat{A'}^{r'}$,
\item a proper ideal $I$ of $A$ with ${\cal K}(I) \subset g^{-1}{\cal K}g$,
  where $g := b^{-1} \circ
g'^{-1} \circ b' \in \GL_r(\AZ_F^f)$
\end{itemize}
such that the diagram
\[ \xymatrix{ S_{F',{\cal K}(I')}^{r'}
  \ar[r]^{\iota_{F,\,b'}^{F'}} \ar[d]^{\pi_{g'^{-1}}} & S_{F,{\cal
      K}(I)}^r \ar[d]^{\pi_g} \\
  S_{F',{\cal K}'}^{r'} \ar[r]^{\incl} & S_{F,\cal K}^r } \]
with $I':= IA'$ commutes 
where $\pi_g$ and $\pi_{g'^{-1}}$ are surjective and defined over
$F$. This implies together with case (i)
\begin{eqnarray*}
 \sigma(X) & = & \sigma(\incl(\pi_{g'^{-1}}(S_{F',{\cal K}(I')}^{r'}))) 
= \sigma(\pi_g(\iota_{F,\,b'}^{F'}(S_{F',{\cal K}(I')}^{r'})))\\
& = & \pi_g(\sigma(\iota_{F,\,b'}^{F'}(S_{F',{\cal K}(I')}^{r'}))) =
\pi_g(\iota_{F,\,\sigma \circ b'}^{\sigma(F')}(S_{\sigma(F'),{\cal K}(\sigma(I'))}^{r'})).
\end{eqnarray*}
By a similar commutative diagram, this is equal to
\[ \iota_{F,\,\sigma \circ b}^{\sigma(F')}(S_{\sigma(F'),\sigma \circ
  {\cal K}' \circ \sigma^{-1}}^{r'}), \]
hence a Drinfeld modular subvariety of $S$ of the desired form.
\end{proof}

%%%%%%%%%%%%%%%%%%%%%%%%%%%%%%%%%%%%%%%%%%%%%%%%%%%%%%%%%%%%%%%%%%%%%
\subsection{Determinant map and irreducible components}
%%%%%%%%%%%%%%%%%%%%%%%%%%%%%%%%%%%%%%%%%%%%%%%%%%%%%%%%%%%%%%%%%%%%%
For a general Drinfeld modular variety $S^r_{F,\cal K}$, we denote by
$\det {\cal K} \subset (\AZ_F^f)^*$ the image of ${\cal K} \subset
\GL_r(\AZ_F^f)$ under the determinant map. Since the determinant map
is a group homomorphism and maps principal congruence subgroups 
of $\GL_r(\AZ_F^f)$ to principal congruence subgroups
of~$(\AZ_F^f)^*$, the subgroup $\det {\cal K} \subset (\AZ_F^f)^*$ is open
and compact.

\begin{definition}
The map $S_{F,\cal K}^r(\C) \rightarrow S_{F,\det {\cal K}}^1(\C)$
given by
\begin{eqnarray*} 
   \GL_r(F) \setminus (\Omega_F^r \times \GL_r(\AZ_F^f)\, /\, \cal K) &
   \longrightarrow & F^* \setminus (\AZ_F^f)^* \, / \,\det {\cal K} \\ 
   \left[(\overline{\omega},h)\right] & \longmapsto & [\det h]. %\left[ noetig da nach \\
\end{eqnarray*}
is called \emph{determinant map} and is denoted by $\det$.
\end{definition}
\textbf{Remark:} The determinant map can be described in terms of the
modular interpretation, using the construction of exterior powers of
Drinfeld modules in~\cite[Theorem 3.3]{vaHe}. We refrain from doing so
because we do not need that.

%Therefore in fact it is
%induced by a morphism $S_{F,\cal K}^r \rightarrow S_{F,\det {\cal K}}^1$ defined
%over $F$. %
%The determinant gives us a morphism from a general Drinfeld modular
%variety $S^r_{F,\cal K}$ to the rank one Drinfeld modular variety
%$S^1_{F,\det \cal K}$ where $\det \cal K$ is the image of $\cal K$ under
%$\GL_r(\AZ_F^f) \stackrel{\det}{\rightarrow} (\AZ_F^f)^*$. On the $\C$-valued
%points the \emph{determinant map} is given by
%\begin{eqnarray*} 
%   \GL_r(F) \setminus (\Omega_F^r \times \GL_r(\AZ_F^f)\, /\, \cal K) &
%   \longrightarrow & F^* \setminus (\AZ_F^f)^* \, / \,\det {\cal K} \\ 
%   \left[(\overline{\omega},h)\right] & \longmapsto & [\det h]. %\left[ noetig da nach \\
%\end{eqnarray*}

\begin{satz} \label{prop:detsurjective}
The determinant map is surjective and its fibers are exactly the
irreducible components of $S_{F,\cal K}^r(\C)$. 
\end{satz}
\begin{proof}
The surjectivity is immediate because $\det: \GL_r(\AZ_F^f)
\rightarrow (\AZ_F^f)^*$ is surjective.

We know by Proposition~\ref{theorem_irrcomp} that the irreducible
components of $S_{F,\cal K}^r(\C)$ are in bijective correspondence
with the double coset space $\GL_r(F) \setminus \GL_r(\AZ_F^f) / {\cal
  K}$. A point $[(\overline{\omega},h)] \in S_{F,\cal K}^r(\C)$ lies
in the irreducible component corresponding to a double coset $[g] \in 
\GL_r(F) \setminus \GL_r(\AZ_F^f) / {\cal K}$ if and only if $h \in
[g]$. 

We show that, for every $g \in \GL_r(\AZ_F^f)$, the
fiber of $[\det g] \in S_{F,\det {\cal K}}^1(\C)$ is equal to the
irreducible component corresponding to $[g] \in 
\GL_r(F) \setminus \GL_r(\AZ_F^f) / {\cal K}$. By the above remarks,
this is equivalent to
\[ h \in \GL_r(F) \cdot g \cdot {\cal K}  \Longleftrightarrow 
\det h \in F^* \cdot (\det g) \cdot (\det {\cal K})  \]
for all $h \in \GL_r(\AZ_F^f)$.

If $h \in \GL_r(F) \cdot g \cdot {\cal K}$, then we have $\det h \in
F^* \cdot (\det g) \cdot (\det {\cal K})$ by the multiplicativity of
the determinant. Conversely, assume that $\det h \in
F^* \cdot (\det g) \cdot (\det {\cal K})$. Then there are $T \in
\GL_r(F)$ and $k \in {\cal K}$ such that
\[ \det h = \det (T \cdot g \cdot k), \]
hence $Tgkh^{-1} \in \SL_r(\AZ_F^f)$. By the strong approximation
theorem~\cite{Pr} for semi-simple simply connected groups over function fields,
$\SL_r(F)$ is dense in $\SL_r(\AZ_F^f)$. Since $h{\cal K}h^{-1}$ is an
open subgroup of $\GL_r(\AZ_F^f)$, we therefore have 
\[ \SL_r(\AZ_F^f) = \SL_r(F) \cdot ((h{\cal K}h^{-1}) \cap
\SL_r(\AZ_F^f)). \]
So there are $T' \in \SL_r(F)$ and $k' \in {\cal
  K} \cap \SL_r(\AZ_F^f)$ such that $Tgkh^{-1} = T'hk'h^{-1}$. This implies
\[ h = T'^{-1}Tgkk'^{-1} \in \GL_r(F) \cdot g \cdot {\cal
  K}. \qedhere \]
\end{proof}

By Proposition~\ref{prop:detsurjective}, the determinant map induces a
bijection
\[ 
{\det}_*: \pi_0(S_{F,\cal K}^r) \stackrel{\sim}{\longrightarrow}
S_{F,\det{\cal K}}^1
\]
between the set $\pi_0(S_{F,\cal K}^r)$ of irreducible components 
of $S_{F,\cal K}^r$ over $\C$ and the set $S_{F,\det{\cal K}}^1$ 
(we identify the latter set with $S_{F,\det {\cal
    K}}^1(\C)$ as explained in Subsection~\ref{sec:rank1}). We now
consider the natural action 
of the absolute Galois group $G_F :=
\Gal(F^{\sep} / F)$ on these two sets.

\begin{satz} \label{prop:detequivariant}
The bijection $\det_*$ is $G_F$-equivariant.
\end{satz}
\begin{proof}
We consider separable extensions $F' \subset \C$ of $F$ of
degree~$r$ with only one place above $\infty$. 
%For example with Eisenstein's Criterion (Proposition III.1.14
%in~\cite{St}) for function fields,
%one sees that there are infinitely many extensions $F' \subset \C$ of this
%form. 
The intersection~$F''$ of all these extensions is equal to $F$. This follows by
induction over $r$:

Assume by contradiction that $F'' \supsetneq F$ with $[F'' / F] = r' >
1$. By Eisenstein's Criterion (Proposition III.1.14
in~\cite{St}) we find a second extension $F''_2 \neq F''$ of $F$ of
degree $r'$ with only one place $\infty''_2$ above $\infty$. By
induction hypothesis, the intersection of all separable extensions of
$F''_2$ of degree $r / r'$ with only one place above $\infty''_2$ is
equal to $F''_2$. These extensions of $F''_2$ are all separable
extensions of $F$ of degree $r$ with only one place above $\infty$,
hence its intersection~$F''_2$ contains $F''$. This is not possible 
because $F''_2 \neq F''$ and $[F''_2 / F] = [F'' / F] = r'$.

The equality $F'' = F$ implies that the subgroups $\Gal(F^{\sep} / F')
\subset G_F$ where $F'$ runs over all separable extensions of $F$ of
degree $r$ with only one place above~$\infty$ generate the whole
absolute Galois group~$G_F$. Therefore it is enough to show
that $\det_*$ is $\Gal(F^{\sep} / F')$-equivariant for all
these extensions~$F'$.

Let now $F' / F$ be a fixed extension of the above form, $Y$ an
irreducible component of $S_{F,\cal K}^r$ and $\sigma \in
\Gal(F^{\sep} / F')$. We have to show that $\det_*(\sigma(Y)) =
\sigma(\det_*(Y))$. We assume that $Y$
corresponds to the class of $g \in \GL_r(\AZ_F^f)$ in $\GL_r(F) \setminus
\GL_r(\AZ_F^f) / {\cal K}$ via the bijective correspondence from
Proposition~\ref{theorem_irrcomp}. We choose an $F$-linear isomorphism
\[ \phi: F^r \stackrel{\sim}{\longrightarrow} F', \]
and define 
\[ b := \phi \circ g : (\AZ_F^f)^r \stackrel{\sim}{\longrightarrow}
\AZ_{F'}^f. \]
The datum $(F',\,b)$ defines an inclusion morphism $\incl: S_{F',\cal
  K'}^1 \rightarrow S_{F,\cal K}^r$. By its definition, the
point $p' := [1] \in S_{F',\cal
  K'}^1 = F'^* \setminus (\AZ_{F'}^f)^* / {\cal K}'$
is mapped to the closed point
\[ p := \incl([1]) = [(\overline{i \circ \phi},\,\phi^{-1} \circ 1 \circ
b)] = [(\overline{i \circ \phi},\,g)] \]
of $S_{F,\cal K}^r$, where $i$ denotes the canonical 
inclusion $F'_{\infty'} \hookrightarrow \C$. This point 
lies in the irreducible component $Y$, which corresponds to the class
of $g$ in $\GL_r(F) \setminus \GL_r(\AZ_F^f)\, /\, {\cal K}$.

By Proposition~\ref{prop:Firred}, the point $p' \in S_{F',\cal K'}^1$ 
is defined over $F'^{\sep} = F^{\sep}$. Since $\incl$ is defined over
$F'$, the closed point $p = \incl(p') \in S_{F,\cal K}^r(\C)$ is also 
defined over $F^{\sep}$ and we have
\[ \incl(\sigma(p')) = \sigma(p) \in \sigma(Y), \]
 i.e., $\sigma(Y)$ is the unique irreducible component of $S_{F,\cal
   K}^r$ containing $\incl(\sigma(p'))$. The equality
 $\det_*(\sigma(Y)) = \sigma(\det_*(Y))$ is therefore equivalent to
 \begin{equation} \label{eq:Gal}
\det(\incl(\sigma(p'))) = \sigma(\det p). \end{equation}
We use the description of the Galois action on $S_{F',{\cal K}'}^1$
and $S_{F,\det {\cal K}}^1$ given by Theorem~\ref{th:CFT} to calculate both sides
of~\eqref{eq:Gal}. For this, let $H / F$
resp. $H' / F'$ be the finite abelian extensions corresponding 
to the closed finite index subgroups 
$F^* \cdot \det {\cal K} \subset (\AZ_F^f)^*$ resp. 
${F'}^* \cdot {\cal K}' \subset (\AZ_{F'}^f)^*$ in class field
theory, and let $E$ be the compositum of $H$ and
$H'$. Then the diagram of Artin maps
\[ \xymatrix{
(\AZ_F^f)^* \ar[r]^{\!\!\!\!\psi_{H / F}} & \Gal(H / F) \\
(\AZ_{F'}^f)^* \ar[u]_{N_{F'/F}} \ar[r]^{\!\!\!\psi_{E / F'}}
\ar[dr]^{\psi_{H' / F'}}  & \Gal(E / F') \ar[u]_{r_{E /
    H}}\ar[d]^{r_{E / H'}} \\
& \Gal(H' / F')
}  \]
commutes with $N_{F'/F}$ the norm map and $r_{E / H}$, $r_{E / H'}$ the
restriction maps. Therefore, if $h' \in (\AZ_{F'}^f)^*$ is chosen such
that $\psi_{E / F'}(h') = \sigma|_E$, then we have 
\begin{eqnarray*} \psi_{H'/F'}(h') & = & \sigma|_{H'}, \\
\psi_{H / F}(N_{F' / F}(h')) & = & \sigma|_H.
\end{eqnarray*}
With Theorem~\ref{th:CFT} this implies 
\begin{eqnarray*} 
 \det(\incl(\sigma(p'))) & = & \det(\incl([h'^{-1}])) = \det(\phi^{-1}
\circ h'^{-1} \circ b) \\
& = & \det(\phi^{-1} \circ h'^{-1} \circ \phi) \cdot \det g 
= [N_{F' / F}(h')^{-1} \cdot \det g] \\
& = & \sigma([\det g]) = \sigma(\det p).\end{eqnarray*}
So we have shown~\eqref{eq:Gal}, which is equivalent 
to $\det_*(\sigma(Y)) = \sigma(\det_*(Y))$.
\end{proof}
\begin{korollar}
The determinant map is induced by a unique morphism $S_{F,\cal K}^r
\rightarrow S_{F,\det {\cal K}}^1$ defined over $F$.
\end{korollar}
\begin{proof}
By Proposition~\ref{prop:detsurjective}, the determinant map is constant on the
irreducible components of $S_{F,\cal K}^r(\C)$. Since these
irreducible components and all closed points of $S_{F,\det {\cal
    K}}^1$ are defined over~$F^{\sep}$, the determinant map is therefore induced by a
unique morphism defined over $F^{\sep}$.

By Proposition~\ref{prop:detequivariant}, this morphism over
$F^{\sep}$ is $G_F$-equivariant. Hence, by~\cite[AG 14.3]{Bo} it is defined over $F$. 
\end{proof}

\begin{korollar} \label{cor:irrcomp}
$S_{F,\cal K}^r$ is $F$-irreducible and has exactly
\[ |S_{F,\det \cal K}^1| = |F^* \setminus (\AZ_F^f)^*\, /\,
{\det \cal K}| = |\mathrm{Cl}(F)| \cdot |\hat{A}^* / (\FZ_q^* \cdot \det {\cal
  K})| \]
% = |\Gal (H_{\det \cal K} / F)| = [H_{\det \cal K} : F] \]
irreducible components over $\C$.
\end{korollar}
\begin{proof} By Corollary~\ref{cor:transgal} and
  Proposition~\ref{prop:detequivariant}, it follows that 
the absolute Galois group $G_F$ 
acts transitively on the set of irreducible components of
$S_{F,\cal K}^r$ over $\C$. Hence, $S_{F,\cal K}^r$ is $F$-irreducible 
by Proposition~\ref{prop:Firred}.

It only remains to show the second equality. Note that
\[ (\AZ_F^f)^* / (F^* \cdot \hat{A}^*) \cong \mathrm{Cl}(F) \]
by the direct adaptation of Proposition VI.1.3 in~\cite{Ne} to
the function field case. Therefore we have
\[ |F^* \setminus (\AZ_F^f)^*\, /\, {\det \cal K}| = 
|\mathrm{Cl}(F)| \cdot |(F^* \cdot \hat{A}^*) / (F^* \cdot
   {\det \cal K})|. \] 

%Note that
%\[ \begin{array}{ccc} (\AZ_F^f)^* & \longrightarrow & I(A) \\
%                      (x_{\pp}) & \longmapsto & \prod_{\pp}
%                      \pp^{v_{\pp}(x_{\pp})}  \end{array}, \]
%with $v_{\pp}$ the normalized discrete valuation associated to $\pp$
%and $I(A)$ the group of fractional ideals of $A$, is a surjective homomorphism with kernel $\hat{A}^*$. Therefore there are
%isomorphisms of abelian groups
%\[ (\AZ_F^f)^* / (F^* \cdot \hat{A}^*) \cong F^* \setminus ((\AZ_F^f)^*\, / \,
%                      \hat{A}^*) \cong F^* \setminus I(A) =
%                      \mathrm{Cl}(F). \]
%The compact open subgroup $\det{\cal K}$ of $(\AZ_F^f)^*$ lies in the 
%unique maximal compact open subgroup~$\hat{A}^*$. Hence
The claim now follows from
\[ (F^* \cdot \hat{A}^*) / (F^* \cdot {\det \cal K}) \cong \hat{A}^*\, /\,
((F^* \cdot \det {\cal K}) \cap \hat{A}^*) \]
and
\[ (F^* \cdot \det {\cal K}) \cap \hat{A}^* = (F^* \cap \hat{A}^*)
\cdot \det {\cal K} = \FZ_q^* \cdot \det \cal K. \qedhere \] 
\end{proof}

\begin{korollar} \label{cor:subirreducible}
Each Drinfeld modular subvariety of $S_{F,\cal K}^r$ with reflex
field~$F'$ is $F'$-irreducible.
\end{korollar}
\begin{proof}
A Drinfeld modular subvariety $X$ of $S_{F,\cal K}^r$ with reflex
field~$F'$ is the image of an inclusion morphism $\incl: S_{F',{\cal
    K}'}^{r'} \rightarrow S_{F,\cal K}^r$. Since $\incl$ is defined
over $F'$ by Theorem~\ref{th_inclusion}, Corollary~\ref{cor:irrcomp}
immediately implies the $F'$-irreducibility of~$X$.
\end{proof}

%%%%%%%%%%%%%%%%%%%%%%%%%%%%%%%%%%%%%%%%%%%%%%%%%%%%%%%%%%%%%%%%%%%%%%%
\section{Degree of subvarieties} \label{ch:degree}
%%%%%%%%%%%%%%%%%%%%%%%%%%%%%%%%%%%%%%%%%%%%%%%%%%%%%%%%%%%%%%%%%%%%%%%
\subsection{Compactification
of Drinfeld modular varieties}
%We assume that all Drinfeld modular varieties in this subsection are
%fine moduli schemes.
In~\cite{Pi} Pink constructs the 
\emph{Satake compactification}~$\overline{S}_{F,\cal K}^r$ of a
Drinfeld modular variety~$S_{F,\cal K}^r$ with ${\cal K} \subset
\GL_r(\hat{A})$. It is a normal projective variety which contains
$S_{F,\cal K}^r$ as an open dense subvariety and it is characterized up
to unique isomorphism by a certain universal property.

If $\cal K$ is amply small, $\overline{S}_{F,\cal K}^r$ is 
endowed with a natural ample invertible 
sheaf~${\cal L}_{F,\cal K}^r$. In~\cite{Pi}, the space of global sections of
its $k$-th power is defined to be the space of algebraic modular forms
of weight~$k$ on $S_{F,\cal K}^r$.

If ${\cal K} \subset \GL_r(\AZ_F^f)$ is arbitrary (not necessarily
contained in $\GL_r(\hat{A})$) and $g \in \GL_r(\AZ_F^f)$ is chosen 
such that $g{\cal K}g^{-1} \subset \GL_r(\hat{A})$, we define
\begin{equation} \label{eq:satake1}
  \overline{S}_{F,\cal K}^r := \overline{S}_{F,g{\cal
    K}g^{-1}}^r \end{equation}
and, if $\cal K$ is amply small,
\begin{equation} \label{eq:satake2}
 {\cal L}_{F,\cal K}^r := {\cal L}_{F,g{\cal K}g^{-1}}^r. 
\end{equation}
As in Step (v) of the proof of
Theorem~\ref{th:Dmodulischeme}, one can show, using part (i) of the following
proposition for ${\cal K} \subset \GL_r(\hat{A})$, that this defines
$\overline{S}_{F,\cal K}^r$ and ${\cal L}_{F,\cal K}^r$ up to isomorphism.

%\begin{definition}
%For a Drinfeld modular variety $S_{F,\cal K}^r$ with ${\cal K} \subset 
%\GL_r(\AZ_F^f)$ amply small, its \emph{Satake 
%compactification} in the sense of Defintion 4.1 in~\cite{Pi} is denoted 
%by $\overline{S}_{F,\cal K}^r$. The natural ample invertible sheaf on
%$\overline{S}_{F,\cal K}^r$ defined in Section 5 of~\cite{Pi} is
%denoted by ${\cal L}_{F,\cal K}^r$. 
%\end{definition}
%\textbf{Remark:} Note that 

\begin{satz} \label{prop:satake}
% properties Satake compactification
%Let $S_{F,\cal K}^r$ be a Drinfeld modular variety with ${\cal K} \subset 
%\GL_r(\AZ_F^f)$ amply small.
\begin{properties}
\item
For $g \in \GL_r(\AZ_F^f)$ and a compact open subgroup ${\cal K'}
\subset g^{-1}{\cal K}g$ the morphism $\pi_g: S_{F,\cal K'}^r \rightarrow
S_{F,\cal K}^{r}$ defined in Subsection~\ref{subsec:Hecke} extends uniquely
to a finite morphism $\overline{\pi_g}: \overline{S}_{F,\cal
  K'}^r \rightarrow \overline{S}_{F,\cal K}^{r}$ defined over $F$. If
$\cal K$ is amply small, then there is a canonical isomorphism
\[ {\cal L}_{F,\cal K'}^r \cong \overline{\pi_g}^* {\cal L}_{F,\cal K}^r. \]

\item 
Any inclusion $\incl: S_{F',\cal K'}^{r'} \rightarrow S_{F,\cal K}^r$
of Drinfeld modular varieties extends uniquely to a finite morphism 
$\overline{\incl}: \overline{S}_{F',\cal K'}^{r'} \rightarrow 
\overline{S}_{F,\cal K}^r$ defined over $F'$. If $\cal K$ is
amply small, then there is a canonical isomorphism
\[ {\cal L}_{F',\cal K'}^{r'} \cong \overline{\incl}^*{\cal L}_{F,\cal
  K}^r. \]
\end{properties}

\end{satz}
\begin{proof} This follows from Proposition 4.11 and 4.12 and Lemma 5.1 
in~\cite{Pi}. Note that these statements automatically hold for 
arbitrary levels ${\cal K}$ and ${\cal K'}$ (not necessarily 
contained in $\GL_r(\hat{A})$ respectively $\GL_{r'}(\widehat{A'})$) 
because the equations~\eqref{eq:satake1} and
\eqref{eq:satake2} define the Satake compactification of a general Drinfeld
modular variety as the Satake compactification of a Drinfeld modular variety
with level contained in $\GL_r(\hat{A})$
resp. $\GL_{r'}(\widehat{A'})$.
\end{proof}

\subsection{Degree of subvarieties}
In this subsection, $S_{F,\cal K}^r$ always denotes a Drinfeld modular
variety with $\cal K$ amply small.

\begin{definition} \label{def:degree}
The \emph{degree} of an irreducible subvariety $X \subset S_{F,\cal
  K}^r$ is defined to be the degree of its Zariski
closure~$\overline{X}$ in $\overline{S}_{F,\cal K}^r$ with respect 
to ${\cal L}_{F,\cal K}^r$, i.e., the integer
\[ \deg X := \deg_{{\cal L}_{F,\cal K}^r} \overline{X} = 
\int_{\overline{S}_{F,\cal K}^r} c_1({\cal L}_{F,\cal K}^r)^{\dim X}
\cap [\overline{X}],
 \]
where $c_1({\cal L}_{F,\cal K}^r) \in A^1\overline{S}_{F,\cal K}^r$ 
denotes the first Chern class of ${\cal L}_{F,\cal K}^r$, the cycle class
of~$\overline{X}$ in $A_{\dim X}\overline{S}_{F,\cal K}^r$ is denoted
by $[\overline{X}]$ and $\cap$ is the cap-product between $A^{\dim
  X}\overline{S}_{F,\cal K}^r$ and $A_{\dim X}\overline{S}_{F,\cal K}^r$. 

The degree of a reducible subvariety $X \subset S_{F,\cal K}^r$ is 
the sum of the degrees of all irreducible components of $X$.
\end{definition}

\textbf{Remark:} Note that our definition of degree for reducible
subvarieties differs from the one used in many textbooks where only
the sum over the irreducible components of maximal dimension is taken.

\begin{lemma} \label{lemma:degirr}
The degree of a subvariety $X \subset S_{F,\cal K}^r$ is at least the
number of irreducible components of $X$.
\end{lemma}
\begin{proof}
This follows by our definition of degree
because ${\cal L}_{F,\cal K}^r$ is ample and the degree
of an irreducible subvariety of a projective variety with respect to
an ample invertible sheaf is a positive integer 
(see, e.g., Lemma 12.1 in~\cite{Fu}).
\end{proof}

\begin{satz} \label{prop:degree}
\begin{properties}
\item
Let $\pi_g: S_{F,\cal K'}^r \rightarrow S_{F,\cal K}^{r}$ be the
morphism defined in Subsection~\ref{subsec:Hecke} for 
$g \in \GL_r(\AZ_F^f)$ and ${\cal K'} \subset g^{-1}{\cal K}g$. Then
\begin{equation} \label{eq:deg1}
 \deg \pi_g^{-1}(X) = [g^{-1}{\cal K}g : {\cal K'}] \cdot \deg X \end{equation}
for subvarieties $X \subset S_{F,\cal K}^r$ and
\begin{equation} \label{eq:deg2} \deg \pi_g(X') \leq \deg X' 
\end{equation}
for subvarieties $X' \subset S_{F,\cal K'}^r$. In particular, we have
\begin{equation} \label{eq:deg3} \deg T_g(X) \leq [{\cal K} : 
{\cal K} \cap g^{-1}{\cal K}g] \cdot \deg X \end{equation}
for subvarieties $X \subset S_{F,\cal K}^r$.

\item
For any inclusion $\incl: S_{F',\cal K'}^{r'} \rightarrow S_{F,\cal K}^r$
of Drinfeld modular varieties and for any
 subvariety $X \subset S_{F',\cal K'}^{r'}$, we have
\begin{equation} \label{eq:deg4}
 \deg X = \deg \incl(X). 
\end{equation}
\end{properties}
\end{satz}
\begin{proof} We use the
  \emph{projection formula} for Chern classes (see,
  e.g., Proposition 2.5~(c) in~\cite{Fu}):

\emph{If $f: X \rightarrow Y$ is a proper morphism of varieties and
  ${\cal L}$ is an invertible sheaf on $Y$, then, for all $k$-cycles $\alpha
  \in A_k(X)$, we have the equality
\begin{equation} \label{eq:projformula}
f_*(c_1(f^*{\cal L}) \cap \alpha) = c_1({\cal L}) \cap f_*(\alpha) \end{equation}
of $(k-1)$-cycles in $A_{k-1}(Y)$.}

For the proof of~\eqref{eq:deg1} and \eqref{eq:deg2}, 
we first assume that $X \subset S_{F,\cal K}^r$
  and $X' \subset S_{F,\cal K'}^r$ are irreducible. For this, 
note that $\pi_g: S_{F,\cal K'}^r \rightarrow S_{F,\cal K}^r$ is
finite of degree $[g^{-1}{\cal K}g : {\cal K}']$  
and \'etale by Proposition~\ref{prop:projetale} because ${\cal K}$ is
amply small. The latter implies that the restriction 
of $\pi_g$ to the subvariety 
$\pi_g^{-1}(X)$ is also finite of degree
$[g^{-1}{\cal K}g : {\cal K}']$ and we have the equality
\[ \overline{\pi_g}_*[\overline{\pi_g^{-1}(X)}] = [g^{-1}{\cal K}g : {\cal K}']
\cdot [\overline{X}] \]
of cycles on $\overline{S}_{F,\cal K}^r$. For $d := \dim X$, 
with Proposition~\ref{prop:satake} (i) and the above projection formula we get
\begin{eqnarray*} \deg \pi_g^{-1}(X) &=& \deg_{\overline{\pi_g}^*{\cal L}_{F,\cal K}^r}
\overline{\pi_g^{-1}(X)} = \int_{\overline{S}_{F,\cal K'}^r} 
c_1(\overline{\pi_g}^*{\cal L}_{F,\cal K}^r)^d
\cap [\overline{\pi_g^{-1}(X)}] \\
&=& \int_{\overline{S}_{F,\cal K}^r} \overline{\pi_g}_*\left(c_1(\overline{\pi_g}^*{\cal L}_{F,\cal K}^r)^d 
\cap [\overline{\pi_g^{-1}(X)}]\right) \\
&=& \int_{\overline{S}_{F,\cal K}^r} c_1({\cal L}_{F,\cal K}^r)^d \cap
\overline{\pi_g}_*[\overline{\pi_g^{-1}(X)}] \\
&=& [g^{-1}{\cal K}g : {\cal K}'] \cdot \int_{\overline{S}_{F,\cal K}^r}
c_1({\cal L}_{F,\cal K}^r)^d \cap [\overline{X}] = [g^{-1}{\cal K}g : {\cal K}']
\cdot \deg X.
\end{eqnarray*}
For the proof of~\eqref{eq:deg2}, we note that
\[ \overline{\pi_g}_*[\overline{X'}] = \deg (\pi_g|_{X'}) \cdot
[\overline{\pi_g(X')}] \]
as cycles on $\overline{S}_{F,\cal K}^r$. The same calculation as
above gives
\[ \deg X' = \deg (\pi_g|_{X'}) \cdot \deg \pi_g(X') \geq \deg
\pi_g(X'.) \]
If $X \subset S_{F,\cal K}^r$ is reducible with irreducible components
$X_1,\ldots,X_n$, we have
\[ \deg \pi_g^{-1}(X) = \sum_{i=1}^n \deg \pi_g^{-1}(X_i) \]
because the set of irreducible components of $\pi_g^{-1}(X)$ is the
disjoint union of the sets of irreducible components of the $\pi_g^{-1}(X_i)$.
Therefore, the formula~\eqref{eq:deg1} follows from the irreducible case.

If $X' \subset S_{F,\cal K'}^r$ is reducible with irreducible
components $X'_1,\ldots,X'_k$, then the set of irreducible components
of $\pi_g(X')$ is a subset of $\{\pi_g(X'_1),\ldots,\pi_g(X'_k)\}$,
hence we have
\[ \deg \pi_g(X') \leq \sum_{i=1}^k \deg \pi_g(X'_i) \]
and the inequality~\eqref{eq:deg2} follows from the irreducible case.

The inequality~\eqref{eq:deg3} immediately follows from
\eqref{eq:deg1} and \eqref{eq:deg2} because
\[ T_g(X) = \pi_g(\pi_1^{-1}(X)) \]
where $\pi_1$ and $\pi_g$ are projection morphisms 
$S_{F,{\cal K}_g}^r \rightarrow S_{F,\cal K}^r$ with ${\cal K}_g :=
 {\cal K} \cap g^{-1}{\cal K}g$ and
\[ \deg \pi_1 = [{\cal K} : {\cal K}_g] = [{\cal K} : {\cal K} \cap g^{-1}{\cal
  K}g]. \]

Finally, for the proof of~\eqref{eq:deg4} we use that $\incl:
S_{F',\cal K'}^{r'} \rightarrow S_{F,\cal K}^r$ is a closed immersion
by Proposition~\ref{prop:closedimm} because ${\cal K}$ is
amply small. For an irreducible subvariety $X \subset
S_{F',\cal K'}^{r'}$, we therefore have the equality
\[ \overline{\incl}_*[\overline{X}] = [\overline{\incl(X)}] \]
of cycles on $\overline{S}_{F,\cal K}^r$. The same calculation as in
the proof of~\eqref{eq:deg1} therefore gives
\[ \deg \incl(X) = \deg X \]
because $\overline{\incl}^*{\cal L}_{F,\cal K}^r \cong {\cal
  L}_{F',{\cal K}'}^{r'}$ by Proposition~\ref{prop:satake} (ii). 

If $X \subset S_{F',\cal K'}^{r'}$ is reducible with irreducible
components $X_1,\ldots,X_l$, then $\incl(X)$ has exactly the
irreducible components $\incl(X_1),\ldots,\incl(X_l)$ because $\incl$
is a closed immersion. Therefore, the
formula~\eqref{eq:deg4} for $X$ reducible follows from the irreducible case.
\end{proof}

We will use the following two consequences of B\'ezout's theorem to get
an upper bound for the degree of the intersection of two subvarieties
of $S_{F,\cal K}^r$:

\begin{lemma}
For subvarieties $V$, $W$ of a projective variety $U$ and an ample 
invertible sheaf ${\cal L}$ on $U$, we have
\[ \deg V \cap W \leq \deg V \cdot \deg W, \]
where $\deg$ denotes the degree with respect to ${\cal L}$.
\end{lemma}
%\textbf{Remark:} Note that we defined the degree of a subvariety $Y
%\subset U$ as the sum of the degrees of \emph{all} 
%irreducible components of $U$ (see Definition~\ref{def:degree}).

\begin{proof} See Example 8.4.6 in~\cite{Fu} in the case
that $V$ and $W$ are irreducible.

If $V = V_1 \cup \cdots \cup V_k$ and $W = W_1 \cup \cdots \cup W_l$ are
decompositions into irreducible components, then
\[ V \cap W = \bigcup_{i,j} V_i \cap W_j. \]
Therefore, each irreducible component of $V \cap W$ is an irreducible
component of some $V_i \cap W_j$. By our definition of
degree for reducible varieties this implies
\[ \deg V \cap W \leq \sum_{i,j} \deg (V_i \cap W_j). \]
Hence by the case that $V$ and $W$ are irredubible, we get
\[ \deg V \cap W \leq \sum_{i,j} \deg V_i \cdot \deg W_j =
\left(\sum_i \deg V_i \right) \cdot \left( \sum_j \deg W_j \right) =
\deg V \cdot \deg W. \]
\end{proof}

\begin{lemma} \label{lemma:bezout}
For subvarieties $V$, $W$ of $S_{F,\cal K}^r$ we have
\[ \deg V \cap W \leq \deg V \cdot \deg W. \]
\end{lemma}
\begin{proof} %Recall that we defined the degree of a subvariety
%of~$S_{F,\cal K}^r$ as the degree of its Zariski closure in the
%compactification~$\overline{S}_{F,\cal K}^r$ with respect to the line
%bundle~${\cal L}_{F,\cal K}^r$. 
In view of the previous lemma, it is
enough to show the following inequality of degrees of
Zariski closures in $\overline{S}_{F,\cal K}^r$ with respect 
to~${\cal L}_{F,\cal K}^r$:
\[ \deg \overline{V \cap W} \leq \deg \overline{V} \cap
\overline{W}. \]
For this, it suffices to show that each irreducible component of
$\overline{V \cap W}$ is an irreducible component of $\overline{V}
\cap \overline{W}$. Note that $\overline{V \cap W} \subset \overline{V} \cap \overline{W}$ and
\[ \overline{V \cap W} \cap S_{F,\cal
  K}^r = V \cap W = (\overline{V} \cap S_{F,\cal K}^r) \cap
(\overline{W} \cap S_{F,\cal K}^r) = (\overline{V} \cap
\overline{W}) \cap S_{F,\cal K}^r \] 
because $S_{F,\cal K}^r$ is Zariski open in $\overline{S}_{F,\cal K}^r$. Therefore
\begin{equation} \label{eq:VW}
 \overline{V} \cap \overline{W} = \overline{V \cap W}\ \cup\ (Y \cap (\overline{V} \cap
\overline{W})) 
\end{equation}
where $Y := \overline{S}_{F,\cal K}^r \setminus S_{F,\cal K}^r$
denotes the boundary of the compactification. Since the irreducible
components of $\overline{V \cap W}$ are the Zariski closures
of the irreducible components of $V \cap W$, they are all not
contained in $Y$ and therefore, by~\eqref{eq:VW}, irreducible
components of $\overline{V} \cap \overline{W}$.
\end{proof}

%\newpage
\subsection{Degree of Drinfeld modular subvarieties}
We let $S = S_{F,\cal K}^r$ be a Drinfeld modular
variety.

\begin{satz} \label{prop:degD}
If ${\cal K}$ is amply small, there is a constant $C > 0$  
only depending on $F$, ${\cal K}$ and $r$ such that 
\[ \deg(X) \geq C \cdot D(X) \]
for all Drinfeld modular subvarieties $X \subset S_{F,\cal K}^r$ with 
$D(X)$ the predegree of $X$ from Definition~\ref{def:reflex_index}.
\end{satz}
\textbf{Remark:} We expect that one could also prove an upper bound
for $\deg(X)$ of the form $\deg(X) \leq C' \cdot D(X)$ with a constant
$C'$ depending on $F$, $\cal K$ and $r$. Because of this expectation we call
$D(X)$ the predegree of $X$. We refrain from proving an upper 
bound because we only need a lower bound in the following.  
\begin{proof} 
Since ${\cal K}$ is amply small, there is a proper ideal $I$ of $A$
and a $g \in \GL_r(\AZ_F^f)$ such that $g{\cal K}g^{-1} \subset {\cal
  K}(I)$. As explained in the beginning of the proof of
Proposition~\ref{prop:closedimm}, for each Drinfeld modular subvariety
$X = \incl(S_{F',{\cal K}'}^{r'})$ of $S$, there is a $g' \in
\GL_{r'}(\AZ_{F'}^f)$ such that ${\cal K}' \subset g'{\cal
  K}(I')g'^{-1}$ where $I' := IA'$. Therefore, by
Proposition~\ref{prop:degree}, we have
\[ \deg(X) = \deg(S_{F',\level}^{r'}) = [g'{\cal K}(I')g'^{-1} : {\cal
  K}'] \cdot \deg(S_{F',{\cal K}(I')}^{r'}). \]
Since $S_{F',{\cal K}(I')}^{r'}$ has at least $|\mathrm{Cl}(F')|$
irreducible components over $\C$ by Corollary~\ref{cor:irrcomp}, we
have $\deg(S_{F',{\cal K}(I')}^{r'}) \geq |\mathrm{Cl}(F')|$. Using
$i(X) = [g'\GL_{r'}(\widehat{A'})g'^{-1} : {\cal K}']$ we therefore get
\[ \deg(X) \geq \frac{i(X)}{[\GL_{r'}(\widehat{A'}) : {\cal K}(I')]}
\cdot |\mathrm{Cl}(F')| = \frac{1}{[\GL_{r'}(\widehat{A'}) : {\cal
    K}(I')]} \cdot D(X). \]
Because of $[\GL_{r'}(\widehat{A'}) : {\cal
    K}(I')] \leq [\GL_r(A) : {\cal K}(I)]$ we conclude that $\deg(X)
\geq C \cdot D(X)$ for $C := \frac1{[\GL_r(A) : {\cal K}(I)]}$ only
depending on ${\cal K}$, $F$ and $r$.
\end{proof}

%\begin{satz}
%Let $X \subset S$ be a Drinfeld modular subvariety with
%reflex field~$F'$ and let ${\cal K'} \subset \GL_{r'}(\hat{A'})$ be
%amply small. Then we have
%\[ \deg(X) \geq h(F') \cdot \mathrm{max}\left\{1, \frac{i(X)}{[\GL_{r'}(\hat{A'}) : {\cal
%    K'}]}\right\}. \]
%\end{satz}

\begin{theorem} \label{th:Dunbounded}
For each sequence $(X_n)$ of pairwise distinct Drinfeld modular 
subvarieties of $S$, the sequence of predegrees $(D(X_n))$ is unbounded. In
particular, if ${\cal K}$ is amply small, the degrees $\deg(X_n)$ are
unbounded.
\end{theorem}
\begin{proof} By Proposition~\ref{prop:degD}, it is enough to show
that the sequence 
\[ D(X_n) = i(X_n) \cdot |\mathrm{Cl}(F_n)| \] 
where $F_n$ is the reflex field of $X_n$ is unbounded.

The following two propositions imply that there are only finitely many
extensions $F'$ of $F$ of degree dividing $r$ and bounded class
number:

\begin{satz}
There are only finitely many finite extensions $F' \subset \C$ of $F$ of
fixed genus $g'$ and bounded degree.
\end{satz}
\begin{proof} By Hurwitz genus formula (see e.g. Theorem III.4.12
  in~\cite{St}) the degree of the different of $F' / F$ and therefore
  also the degree of the discriminant divisor of $F' / F$ is bounded
  as $F'$ runs over all finite separable extensions of $F$ of fixed
  genus and bounded degree. Hence, Theorem 8.23.5 in~\cite{Go} implies
  that there are only finitely many separable extensions of $F$ with
  fixed genus and bounded degree. Since each finite extension of $F$
  can be decomposed into a separable and a totally inseparable
  extension and each global function field has at most one totally
  inseparable extension of a given degree, the proposition follows.
\end{proof}

\begin{satz} \label{prop:Clg}
Let $F'$ be a function field of genus $g'$ with field of constants
$\FZ_{q'}$. Then
\[ |\mathrm{Cl}(F')| \geq \frac{(q' - 1)({q'}^{2g'} - 2g'{q'}^{g'} +
  1)}{2g'({q'}^{g' + 1} - 1)}.\]
\end{satz}
\begin{proof} Proposition 3.1 in \cite{Br1}.\end{proof}

Therefore, the sequence $D(X_n)$ is unbounded if the set of
reflex fields~$F_n$ is infinite. So it suffices to show unboundedness
of the predegree~$D(X_n)$ in a sequence of pairwise distinct Drinfeld modular 
subvarieties of $S$ with fixed reflex field. This follows from the 
next theorem.
\end{proof}

\begin{theorem}
For each sequence $(X_n)$ of pairwise distinct Drinfeld modular 
subvarieties of $S$ with fixed reflex field $F'$, the indices~$i(X_n)$
are unbounded.
\end{theorem} 

\begin{proof}  We first note that we can assume w.l.o.g. that
the given compact subgroup ${\cal K}$ equals $\GL_r(\hat{A})$. Indeed, 
if ${\cal K}$ is
replaced by a compact open subgroup ${\cal L} \supset {\cal K}$ and
the $X_n$ by their images under the canonical projection
$\pi_1: S_{F,\cal K}^r \rightarrow S_{F,\cal L}^r$, the indices~$i(X_n)$
decrease by Definition~\ref{def:reflex_index}. 
Hence, we can assume that $\cal K$ is a
maximal compact open subgroup and therefore some conjugate
$h\GL_r(\hat{A})h^{-1}$ of $\GL_r(\hat{A})$. If we further replace the~$X_n$
by their images under the isomorphism $\pi_{h^{-1}}:
S_{F,\,h\GL_r(\hat{A})h^{-1}}^r \rightarrow S_{F,\,\GL_r(\hat{A})}^r$,
then the $i(X_n)$ do obviously not change because the $X_n$ are
the image of an inclusion from the
same $S_{F',\level}^{r'}(\C)$. Therefore, we can w.l.o.g. assume
${\cal K} = \GL_r(\hat{A})$.

For the following considerations, we
assume that $X_n = \iota_{F,\,b_n}^{F'}(S_{F'\!,{{\cal
      K}_{n}'}}^{r'}(\C))$ 
with $\AZ_F^f$-linear
isomorphisms $b_n: (\AZ_F^f)^r \rightarrow (\AZ_{F'}^f)^{r'}$. 
We denote by $\Lambda_n$
the $\hat{A}$-lattices~$b_n(\hat{A}^r)$ in~$(\AZ_{F'}^f)^{r'}$. 
By Corollary~\ref{prop:sub_lattice}, they are
determined up to and only up to the action of $\GL_{r'}(\AZ_{F'}^f)$ and
their orbits under the action of $\GL_{r'}(\AZ_{F'}^f)$ are pairwise
distinct.

We have the product decomposition 
$\Lambda_n = \prod_{\pp \neq \infty} \Lambda_{n,\pp} :=
\prod_{\pp \neq \infty} b_{n,\pp}(A_{\pp}^r)$, where
$\Lambda_{n,\pp} \subset {F'_{\pp}}^{r'}$ are free $A_{\pp}$-submodules
of rank $r$.  
The $A'_{\pp}$-modules $A'_{\pp} \cdot \Lambda_{n,\pp}$ are
finitely generated submodules of ${F'_{\pp}}^{r'}$ with ${F'_{\pp}}
\cdot \Lambda_{n,\pp} = {F'_{\pp}}^{r'}$, hence free of rank~$r'$ because
$A'_{\pp}$ is a direct product of principal ideal domains.
% direct product of discrete valuation rings
This implies 
that $\hat{A'} \cdot \Lambda_n$ is a free
$\hat{A'}$-submodule of $(\AZ_{F'}^f)^{r'}$ of rank $r'$. Since the
$\Lambda_n$ are determined up to and only up to the action of
$\GL_{r'}(\AZ_{F'}^f)$, we may therefore assume w.l.o.g. that $\hat{A'} \cdot
\Lambda_n = \hat{A'}^{r'}$ for all $n$.

Note that we have
\[ {\cal K}_{n}' = (b_n^{-1}\GL_r(\hat{A})b_n) \cap \GL_{r'}(\AZ_{F'}^f) =
 \mathrm{Stab}_{\GL_{r'}(\AZ_{F'}^f)} \Lambda_n. \]
Since $\hat{A'} \cdot \Lambda_n = \hat{A'}^{r'}$, these compact open
  subgroups of $\GL_{r'}(\AZ_{F'}^f)$ are all contained in the maximal
  compact subgroup $\GL_{r'}(\hat{A'}) =
  \mathrm{Stab}_{\GL_{r'}(\AZ_{F'}^f)}\hat{A'}^{r'}$. Hence, we
  can write the indices $i(X_n)$ as
\[ i(X_n) = [\GL_{r'}(\hat{A'}) :
  \mathrm{Stab}_{\GL_{r'}(\AZ_{F'}^f)}\Lambda_n] \]
and, using the above product decompositions, as $i(X_n) = \prod_{\pp \neq
  \infty}i_{n,\pp}$, where
\[ i_{n,\pp} = [\GL_{r'}(A'_{\pp}) : \mathrm{Stab}_{\GL_{r'}(F'_{\pp})} \Lambda_{n,\pp}].\]
For each $n$, almost all factors of this product are $1$ because
$\Lambda_{n,\pp} = {A'_{\pp}}^{r'}$ for almost all~$\pp$.

Since we assumed that $A'_{\pp} \cdot \Lambda_{n,\pp} =
{A'_{\pp}}^{r'}$, by the Proposition~\ref{prop_gitter} below, 
we get the estimates $i_{n,\pp} \geq C \cdot [{A'_{\pp}}^{r'}
 : \Lambda_{n,\pp}]^{1/r}$, where the constant $C$ is independent of
 $n$ and~$\pp$.

We now finish the proof by
assuming (by contradiction) that the sequence $(i(X_n))$ is bounded.
This implies by the above product decomposition of $i(X_n)$
and estimates of $i_{n,\pp}$ that \linebreak$[{A'_{\pp}}^{r'} :
  \Lambda_{n,\pp}] \leq D$ for all $n$ and $\pp$ for some uniform constant $D$.  

But, note that, as finite
$A_{\pp}$-module, ${A_{\pp}'}^{r'} / \Lambda_{n,\pp}$ is isomorphic to some product
\[ A_{\pp} / \pp^{m_1}A_{\pp} \times \cdots \times A_{\pp} /
\pp^{m_l}A_{\pp}.\]
If $\Lambda_{n,\pp} \not \supseteq \pp^N \cdot {A_{\pp}'}^{r'}$, we
have $m_i \geq N+1$ for some $i$ and therefore 
\[ |k(\pp)|^{N+1} \leq [{A_{\pp}'}^{r'} : \Lambda_{n,\pp}] \leq D.\]
In particular, we have $|k(\pp)| \leq D$ whenever $\Lambda_{n,\pp} \neq {A_{\pp}'}^{r'}$. Since
there are only finitely many primes~$\pp$ with $|k(\pp)| \leq D$, we conclude:
\begin{itemize} 
\item
There are finitely many primes $\pp_1,\ldots,\pp_k$ such that $\Lambda_{n,\pp} = {A_{\pp}'}^{r'}$
for all $n$ and $\pp \neq \pp_1,\ldots,\pp_k$.

\item
There is a $N \in \NZ$ such that, for all $\pp$ and $n$, the $A_{\pp}$-lattice
$\Lambda_{n,\pp}$ contains $\pp^N{A_{\pp}'}^{r'}$.
\end{itemize}
Since the quotients ${A_{\pp}'}^{r'} / \pp^N{A_{\pp}'}^{r'}$ are
finite, the second statement implies that for all $1 \leq i \leq k$ there are
only finitely many possibilities for $\Lambda_{n,\pp_i}$. As for $\pp
\neq \pp_1,\ldots,\pp_k$ the lattices $\Lambda_{n,\pp}$ are
independent of $n$, this implies that only finitely many
$\hat{A}$-lattices $\Lambda_n \subset (\AZ_{F'}^f)^{r'}$ occur, a
contradiction to our assumptions.
\end{proof}

\begin{satz}\label{prop_gitter}
Let $K$ be a complete field with respect to a discrete valuation $v$
with finite residue field containing $\FZ_q$ and let $R$ be the
corresponding discrete valuation ring with maximal ideal $\mm$. 
Let $K' := L_1 \times \cdots \times L_m$ with $L_i$ finite field extensions 
of~$K$ and $R' := S_1 \times \cdots \times S_m$ with $S_i \subset L_i$
the discrete valuation ring associated
to the unique extension of $v$ to $L_i$. Suppose that $r' \geq 1$ and set 
$r := r' \cdot \sum_{i=1}^m [L_i : K]$.

There is a constant $C > 0$ only depending on $q$ and $r$ such that, for
any free $R$-submodule $\Lambda \subset K'^{r'}$ of rank $r$ with $R'
\cdot \Lambda = {R'}^{r'}$, we have
\[ [\GL_{r'}(R') : 
    \mathrm{Stab}_{\GL_{r'}(K')}(\Lambda)] \geq C \cdot  [{R'}^{r'}
 : \Lambda]^{1/r}. \]

%There is a constant $C > 0$ only depending on $q$ and $r$ such that, for
%any free $R$-submodule $\Lambda \subset R'^{r'}$ of rank $r$ with $R'
%\cdot \Lambda = {R'}^{r'}$, we have
%\[ [\GL_{r'}(R') : \mathrm{Stab}_{\GL_{r'}(K')}(\Lambda)] \geq C \cdot
%   [{R'}^{r'} : \Lambda]^{1/r}. \]
\end{satz}

\begin{proof}
We introduce the notation
\[ H := \{ T \in \mathrm{Mat}_{r'}(R') : T \cdot \Lambda \subseteq
\Lambda \}. \]
This set of matrices is an $R$-subalgebra of
$\mathrm{Mat}_{r'}(R')$ with $H^{*} =
\mathrm{Stab}_{\GL_{r'}(R')}(\Lambda)$.

Note that, if $g_1,\ldots,g_r$ is a $R$-basis of $\Lambda$, then $\Lambda = \xi(R^r) \subset {K'}^{r'}$ for
\[ \xi: \begin{array}{ccc} K^r & \longrightarrow & {K'}^{r'} \\
   (x_1,\ldots,x_r) & \longmapsto & x_1g_1 + \cdots + x_rg_r\end{array}. \]
Since $K$ is complete, $\xi$ is a homeomorphism (cf. Proposition 4.9 in
\cite{Ne}). This implies that $\Lambda \subset {R'}^{r'}$ is open.

Hence, there is a $k \in \NZ$ such that $\mm^k{R'}^{r'} \subseteq
\Lambda$. Therefore $\mathrm{Mat}_{r'}(\mm^kR') \subseteq H$
and 
\[ H / \mathrm{Mat}_{r'}(\mm^kR') = \{ \overline{T} \in \Mat_{r'}(R' / \mm^kR') :
\overline{T} \cdot (\Lambda / \mm^k {R'}^{r'}) \subseteq \Lambda / \mm^k {R'}^{r'} \}\]
if we identify $\Mat_{r'}(R' / \mm^kR')$ with $\Mat_{r'}(R') /
\Mat_{r'}(\mm^kR')$. For the stabilizer of $\Lambda / \mm^k {R'}^{r'}$
under the action of $\GL_{r'}(R' / \mm^kR')$, this means that
\[ (H / \mathrm{Mat}_{r'}(\mm^kR'))^* = \mathrm{Stab}_{\GL_{r'}(R' /
  \mm^kR')}(\Lambda / \mm^k{R'}^{r'}).\]
The orbit of $\Lambda$ under $\GL_{r'}(R')$ is in bijective
correspondence with the orbit of $\Lambda / \mm^k{R'}^{r'}$ under
$\GL_{r'}(R' / \mm^kR')$ via
\[ T \cdot \Lambda \longmapsto (T \cdot \Lambda) / \mm^k{R'}^{r'}. \]
Therefore the above formulas for the corresponding stabilizers
give us the following estimate:
\[ [\GL_{r'}(R') : H^*] = [\GL_{r'}(R' / \mm^k R'):(H / \mathrm{Mat}_{r'}(\mm^k
R'))^*] \geq \frac{|\GL_{r'}(R' / \mm^k R')|}{[H : \Mat_{r'}(\mm^k
    R')]} \]
\begin{lemma}\label{lemma_mat} There is a constant $C$ only depending on $q$ and $r$ (namely $C = (1-\frac{1}{q})^r$)
  such that 
 \[ |\GL_{r'}(R' / \mm^k R')| \geq C \cdot |\Mat_{r'}(R' /
  \mm^k R')|.\]
\end{lemma}
\begin{proof} By definition of $R'$ the quotient $R' / \mm^k R'$ is
isomorphic to $S_1 / \mm^k S_1 \times \cdots \times S_m / \mm^k S_m$,
hence
\begin{eqnarray*} \GL_{r'}(R' / \mm^k R') & \cong & \GL_{r'}(S_1 / \mm^k S_1) \times
\cdots \times \GL_{r'}(S_m / \mm^k S_m) \\
\Mat_{r'}(R' / \mm^k R') & \cong & \Mat_{r'}(S_1 / \mm^k S_1) \times
\cdots \times \Mat_{r'}(S_m / \mm^k S_m). 
\end{eqnarray*}
Now note that, for any $l \geq 1$ and any discrete valuation ring $U$ 
with maximal ideal~$\nn$ and residue field $\FZ_{q'}$ containing $\FZ_q$, a matrix $T \in \Mat_{r'}(U)$ is invertible if and only if its
reduction modulo $\nn^l$ is invertible in $\Mat_{r'}(U / \nn^l)$. In
particular, $\GL_{r'}(U / \nn^l)$ exactly consists of the matrices
with reduction modulo $\nn$ lying in $\GL_{r'}(U / \nn)$. As the
fibers of the projection $\Mat_{r'}(U / \nn^l) \rightarrow \Mat_{r'}(U
/ \nn)$ have all cardinality $|\nn / \nn^l|^{{r'}^2} = {q'}^{(l-1){r'}^2}$, we
get
\begin{eqnarray*} 
|\GL_{r'}(U / \nn^l)| & = & {q'}^{(l-1){r'}^2}|\GL_{r'}(\FZ_{q'})| \\
                      & = &
                      {q'}^{(l-1){r'}^2}({q'}^{r'}-1)({q'}^{r'}-q')
                      \cdots ({q'}^{r'} - {q'}^{r'-1}) \\
                      & \geq & {q'}^{l{r'}^2}\left(1 - \frac{1}{q}\right)^{r'} \\
                      & = & \left(1 - \frac{1}{q}\right)^{r'}|\Mat_{r'}(U / \nn^l)|.
\end{eqnarray*}
Since $m \leq r / r'$, we have in total
\[ |\GL_{r'}(R' / \mm^k R')| \geq \left(1 - \frac{1}{q}\right)^{mr'}|\Mat_{r'}(R'
/ \mm^k R')| \geq C \cdot |\Mat_{r'}(R' / \mm^k R')| \]
with $C = \left(1 - \frac{1}{q}\right)^{r}$. 
\end{proof}

By Lemma~\ref{lemma_mat} and the estimate before, we have
\[ [\GL_{r'}(R') : H^*] \geq C \cdot \frac{[\Mat_{r'}(R') :
  \Mat_{r'}(\mm^k R')]}{[H : \Mat_{r'}(\mm^k R')]} = C \cdot
  [\Mat_{r'}(R') : H]. \]
To finish the proof of Proposition~\ref{prop_gitter}, we consider
  a $R$-basis $g_1,\ldots,g_r$ of $\Lambda$ and the $R$-module homomorphism
\[ \begin{array}{ccc} \Mat_{r'}(R')^r & \longrightarrow & {R'}^{r'} / \Lambda \\
  (T_1,\ldots,T_r) & \longmapsto & (T_1 \cdot g_1 + \cdots + T_r \cdot
  g_r) \mod \Lambda. \end{array} \]
It is surjective and its kernel contains $H^r$. Therefore, we have
\[ [\Mat_{r'}(R') : H]^r = [\Mat_{r'}(R')^r : H^r] \geq [{R'}^{r'} : \Lambda] \]
and in total
\[ [\GL_{r'}(R') : \mathrm{Stab}_{\GL_{r'}(K')}(\Lambda)] 
   \geq C \cdot [\Mat_{r'}(R') : H] \geq C \cdot
   [{R'}^{r'} : \Lambda]^{1/r}. \qedhere \] 
\end{proof}

%%%%%%%%%%%%%%%%%%%%%%%%%%%%%%%%%%%%%%%%%%%%%%%%%%%%%%%%%%%%%%%%%%%%%%
\section{Zariski density of Hecke orbits}\label{ch:density}
%%%%%%%%%%%%%%%%%%%%%%%%%%%%%%%%%%%%%%%%%%%%%%%%%%%%%%%%%%%%%%%%%%%%%%
In the whole section, $S = S_{F,\cal K}^r$ denotes a Drinfeld modular
variety and $C$ a set of representatives in $\GL_r(\AZ_F^f)$ for
$\GL_r(F) \setminus \GL_r(\AZ_F^f) / {\cal K}$. We use the description
of the irreducible components of $S$ over $\C$ given in
Proposition~\ref{theorem_irrcomp}: We let $Y_h$ be the
irreducible component of $S$ over $\C$ corresponding to $h \in C$ and
identify its $\C$-valued points $Y_h(\C) \subset \GL_r(F) \setminus
(\Omega_F^r \times \GL_r(\AZ_F^f) / {\cal K})$ with $\Gamma_h \setminus
\Omega_F^r$ where $\Gamma_h := h{\cal K}h^{-1} \cap \GL_r(F)$
via the isomorphism from Proposition~\ref{theorem_irrcomp}. 
\subsection{$T_g + T_{g^{-1}}$-orbits}
For $g \in \GL_r(\AZ_F^f)$ and closed subvarieties $Z \subset
S$ we define
\[ (T_g + T_{g^{-1}})(Z) := T_g(Z) \cup T_{g^{-1}}(Z), \]
and recursively
\begin{eqnarray*}
 (T_g + T_{g^{-1}})^0(Z) & := & Z \\
 (T_g + T_{g^{-1}})^n(Z) & := & (T_g + T_{g^{-1}})\big((T_g +
    T_{g^{-1}})^{n-1}(Z)\big),\ n \geq 1. 
\end{eqnarray*}

\begin{definition}
For a geometric point $x \in S(\C)$ and $g \in \GL_r(\AZ_F^f)$, the union
\[ T_g^\infty(x) := \bigcup_{n \geq 0}(T_g + T_{g^{-1}})^n(x)
 \subset S(\C) \]
is called the $T_g + T_{g^{-1}}$-\emph{orbit} of $x$.
\end{definition}
Note that $T_g^\infty(x)$ is the smallest subset of $S(\C)$
containing $x$ which is mapped into itself under $T_g$ and
$T_{g^{-1}}$.

We now give an explicit description of~the intersection of
$T_g^\infty(x)$ with the irreducible components of $S$ over $\C$ for
$x \in S(\C)$ and $g \in \GL_r(\AZ_F^f)$.

\begin{satz} \label{satz:hecke_explizit}
Let $h_1,h_2 \in C$ and assume that $x \in Y_{h_1}(\C)$
with $x = [\omega] \in \Gamma_{h_1} \setminus \Omega_F^r$. Then the
intersection of $T_g^\infty(x)$ with $Y_{h_2}(\C)$ is given by
\[ T_g^\infty(x) \cap Y_{h_2}(\C) = \{ [T\omega] \in \Gamma_{h_2} \setminus \Omega_F^r \,:\,T \in
h_2\langle{\cal K}g{\cal K}\rangle h_1^{-1} \cap \GL_r(F) \}, \]
where $\langle{\cal K}g{\cal K}\rangle$ denotes the subgroup of
$\GL_r(\AZ_F^f)$ generated by the double coset ${\cal K}g{\cal K}$.
\end{satz}

\begin{proof} By assumption, we have $x = [(\omega, h_1)] \in
\GL_r(F) \setminus \big( \Omega_F^r \times \GL_r(\AZ_F^f) / {\cal K}
\big)$. Hence, by Definition~\ref{def:hecke} and the recursive definition of
$(T_g + T_{g^{-1}})^n(x)$, the elements of
  $T_g^\infty(x)$ are exactly those of the form $[(\omega, 
h_1k_1g_1k_2g_2\cdots k_ng_n)]$
with $n \geq 0$, $k_i \in {\cal K}$ and $g_i \in \{g,g^{-1}\}$.
Hence, an element $y \in T_g^\infty(x) \cap Y_{h_2}(\C)$ can be written
as $y=[(\omega,h_1s)]$ with $s \in \langle {\cal K}g{\cal K} \rangle $.
Since $y$ lies in $Y_{h_2}$, there exist $T \in \GL_r(F)$ and $k \in 
\cal K$ with $Th_1sk = h_2$. Therefore 
\[ y = [(\omega, h_1s)] = [(T\omega, Th_1sk)] = [(T\omega, h_2)] \] 
is equal to $[T\omega] \in \Gamma_{h_2} \setminus
\Omega_F^r$, where $T \in h_2\langle{\cal K}g{\cal K}\rangle
h_1^{-1} \cap \GL_r(F)$.

Conversely, an element $[T\omega] \in \Gamma_{h_2} \setminus \Omega_F^r$ with
$T=h_2sh_1^{-1} \in h_2\langle {\cal K}g{\cal K} \rangle h_1^{-1} \cap
\GL_r(F)$ is equal to
\[ [(T\omega, h_2)] = [(\omega,T^{-1}h_2)] = [(\omega,
h_1s^{-1}h_2^{-1}h_2)] = [(\omega, h_1s^{-1})] \]
with $s^{-1} \in \langle {\cal K}g{\cal K} \rangle$, hence lies in
$T_g^\infty(x) \cap Y_{h_2}(\C)$.
\end{proof}

\subsection{Zariski density}
We give a sufficient condition for a subset $M \subset S(\C)$ to be
Zariski dense in one irreducible component $Y_h$ of $S$ over $\C$. 
Recall that, for a place
$\pp \neq \infty$ of~$F$, by $\AZ_F^{f,\,\pp}$ we denote the adeles
 outside $\infty$ and $\pp$.

\begin{satz} \label{satz:zardensity_subsets}
Let $M$ be a subset of $S(\C)$ contained in an irreducible component
$Y_h$ of $S$ over $\C$ for $h \in C$ and suppose that $M$ contains an element 
$x=[\omega] \in Y_h(\C) = \Gamma_h \setminus \Omega_F^r$ such that
there exists a place~$\pp \neq \infty$ of $F$ and an open subgroup
${\cal K'} \subset \GL_r(\AZ_F^{f,\,\pp})$ with
\[ M' := \{[T\omega] \in  \Gamma_h \setminus \Omega_F^r\ :\ T \in (\mathrm{SL}_r(F_{\pp}) \times {\cal K'})
\cap \mathrm{GL}_r(F) \} \subset M. \]
Then $M$ is Zariski dense in $Y_h$.
\end{satz}

\begin{proof} We denote the Zariski closure of $M'$ by $Y$.
It is enough to show that $Y(\C) = Y_h(\C)$. As the nonsingular locus 
$Y^{\mathrm{ns}}$ of $Y$ over $\C$ is Zariski open and dense in
$Y$ (Theorem I.5.3 in~\cite{Ha}), the intersection 
$Y^{\mathrm{ns}}(\C) \cap M'$ is non-empty. Since 
$(\mathrm{SL}_r(F_{\pp}) \times {\cal K'}) \cap \mathrm{GL}_r(F)$ is 
a subgroup of $\mathrm{GL}_r(F)$,
%, after replacing $[\omega]$ by 
%$[\omega']$, we still have 
%\[ M'=\{[T\omega]
% \in Y_h(\C)\ :\ T \in (\mathrm{SL}_r(F_{\pp}) \times {\cal K'})
%\cap \mathrm{GL}_r(F) \}.\] 
we can therefore assume that $x = [\omega]$
lies in $Y^{\mathrm{ns}}(\C)$. Hence it is enough to show that the tangent
space $T_xY$ of $Y$ at $x$ is of dimension $r-1 = \dim S$.

Since ${\cal K'}$ is open in $\GL_r(\AZ_F^{f,\,\pp})$, there is an $N \in
A$ with $N \not\in \pp$ such that $K'(N) \subset {\cal K'}$, where
$K'(N)$ denotes the principal congruence subgroup modulo~$N$ of
$\GL_r(\AZ_F^{f,\,\pp})$. Now let $l \geq 1$ such that 
$\pp^l = (\pi)$ is a principal
ideal of $A$ and consider for $1 \leq i \leq r-1$ and $k \geq 1$ the matrices
\[ A_{ik} := \left( \begin{array}{ccccc} 1 &&&& \\ &\ddots &&& \\ &&1&&\\
   &&&\ddots& \\ &&\frac{N}{\pi^k}&&1 \end{array} \right) \ \ \in \mathrm{SL}_r(F), \]
with the entry $\frac{N}{\pi^k}$ in the ith column.
As elements of 
$\GL_r(\AZ_F^f)$ (diagonally embedded) they lie in
$\mathrm{SL}_r(F_{\pp}) \times K'(N) \subset \mathrm{SL}_r(F_{\pp})
\times {\cal K'}$. Hence, for all $1 \leq i \leq r-1$ and $k \geq 1$, 
$[A_{ik}\omega]$ lies in $M' \subset Y(\C)$.

We now view $\Omega_F^r$ as a subset of $\AZ^{r-1}(\C)$ by identifying
$[\omega_1:\cdots:\omega_{r-1}:1]$ with $(\omega_1,\ldots,\omega_{r-1})$ 
(note that the $r$-th projective coordinate $\omega_r$ of an 
arbitrary element of $\Omega_F^r$ can be assumed to be $1$ because 
the $F_{\infty}$-rational hyperplane $\omega_r = 0$ does not belong to
$\Omega_F^r$). Assume that we have $\omega =
(\omega_1,\ldots,\omega_{r-1})$
in this identification. Then, using~\eqref{eq:actionGLrF}, we see that
\[ A_{ik}\omega = (\omega_1, \ldots,\omega_i -
\frac{N}{\pi^k},\ldots,\omega_{r-1}) \]
for all $1 \leq i \leq r-1$ and $k \geq 1$. Note that 
$\omega_i - \frac{N}{\pi^k}$ converges to $\omega_i$ in $\C$ for $k
\rightarrow \infty$ and that $\{[A_{ik}\omega]\}_{k \geq 1} \subset Y(\C)$  
for all $1 \leq i \leq r-1$. Since $Y(\C) \subset 
Y_h(\C) = \Gamma_h \setminus \Omega_F^r$ is closed
in the rigid analytic topology, it follows that
there is an $\varepsilon > 0$ such that for all $1 \leq i \leq r-1$
and $c \in \C$ with $|c|_{\infty} < \varepsilon$
\[ [(\omega_1, \ldots,\omega_i + c, \ldots, \omega_{r-1})] \in Y(\C). \] 
This implies $\dim T_xY = r-1$ and $Y(\C) = Y_h(\C)$.
\end{proof}

Now let $\pp \neq
\infty$ be a place of $F$ and $g \in \GL_r(\AZ_F^f)$ trivial outside
$\pp$, i.e., $g := (1,\ldots,g_{\pp},\ldots,1)$ for some 
$g_{\pp} \in \GL_r(F_{\pp})$. Using
Proposition~\ref{satz:zardensity_subsets},
we prove a sufficient condition
for the $T_g + T_{g^{-1}}$-orbit $T_g^\infty(x)$ to be Zariski dense in the
irreducible component of $S$ over $\C$ containing $x$. This result is a
generalization of Theorem~4.11 in~$\cite{Br2}$.

\begin{theorem} \label{th:heckedensity}
Assume that the image of the cyclic subgroup $\langle g_{\pp} \rangle
\subset \GL_r(F_{\pp})$
in $\mathrm{PGL}_r(F_{\pp})$ is unbounded and, for $x \in S(\C)$, let
$Y_x$ be the irreducible component of $S$ over $\C$ containing $x$.
Then, for all $x \in S(\C)$, the intersection of the
$T_g + T_{g^{-1}}$-orbit $T_g^\infty(x)$ with 
$Y_x(\C)$ is Zariski dense in $Y_x$.
\end{theorem}
\begin{proof} We assume that $Y_x = Y_h$ for some $h \in C$. Then, by
Proposition~\ref{satz:hecke_explizit}, we have
\[ T_g^\infty(x) \cap Y_x(\C) = \{[T\omega] \in \Gamma_h \setminus
\Omega_F^r\ :\ T \in h\langle {\cal K}g{\cal K} \rangle h^{-1} \cap
\GL_r(F) \}. \]
We can find compact open subgroups ${\cal K}_{\pp} \subset \GL_r(F_{\pp})$ and 
${\cal K}' \subset \GL_r(\AZ_F^{f,\,\pp})$ such that
\[ \langle {\cal K}_{\pp}h_{\pp}g_{\pp}h_{\pp}^{-1}{\cal K}_{\pp} \rangle \times 
{\cal K}' \subset h\langle {\cal K}g{\cal K} \rangle h^{-1}. \]
%Since ${\cal K}$ is an open subgroup of $\GL_r(\AZ_F^f)$, there is an $N
%\in A$ such that the principal congruence subgroup $K(N) \subset \GL_r(\AZ_F^f)$
%modulo $N$ is contained in $h{\cal K}h^{-1}$. If the principal ideal
%$(N) \subset A$ is equal to $\prod_{\qq}\qq^{\nu_{\qq}}$, then we can write $K(N) =
%\prod_{\qq}K_{\qq}(N)$ with
%\[ K_{\qq}(N) = \left\{\begin{array}{ll} 
%\{t_{\qq} \in \GL_r(A_{\qq})\ :\ t_{\qq} \equiv
%1\ (\mathrm{mod}\ \qq^{v_{\qq}})\} &,\ v_{\qq} > 0\\
%\GL_r(A_{\qq}) &,\ v_{\qq} = 0  \end{array} \right. . \]
%Hence, we have
%\[ \langle K_{\pp}(N)h_{\pp}g_{\pp}h_{\pp}^{-1}K_{\pp}(N) \rangle \times
%\prod_{\qq \neq \pp} K_{\qq}(N) = \langle K(N)hgh^{-1}K(N) \rangle
%\subset h\langle {\cal K}g{\cal K} \rangle h^{-1}. \]
We now consider the open subgroup $U_{\pp} := \langle K_{\pp}h_{\pp}g_{\pp}h_{\pp}^{-1}K_{\pp} \rangle
\cap \SL_r(F_{\pp})$ of $\SL_r(F_{\pp})$. It is normalized by the image of $\langle g_{\pp} 
\rangle$ in $\mathrm{PGL}_r(F_{\pp})$, which is unbounded by assumption. Since
$\mathrm{PGL}_r$ is a connected adjoint absolutely simple linear algebraic group 
over the local field~$F_{\pp}$ and 
$\mathrm{SL}_r \hookrightarrow \GL_r \rightarrow
\mathrm{PGL}_r$ is its universal covering, we conclude by Theorem~2.2
 of~\cite{Pi9} that $U_\pp$ is equal to 
$\mathrm{SL}_r(F_{\pp})$.

Hence, $\mathrm{SL}_r(F_{\pp})$ is contained in $\langle K_{\pp}h_{\pp}g_{\pp}h_{\pp}^{-1}K_{\pp} \rangle$ and we have
\[ \{[T\omega] \in \Gamma_h \setminus
\Omega_F^r\ :\ T \in (\mathrm{SL}_r(F_{\pp}) \times {\cal K'}) \cap
\GL_r(F) \} \subset T_g^\infty(x) \cap Y_x(\C). \]
Therefore, we can apply Proposition~\ref{satz:zardensity_subsets} to
the subset $T_g^\infty(x) \cap Y_x(\C)$ of $S(\C)$ and conclude
that $T_g^\infty(x) \cap Y_x(\C)$ is Zariski dense in $Y_x$.
\end{proof}

%%%%%%%%%%%%%%%%%%%%%%%%%%%%%%%%%%%%%%%%%%%%%%%%%%%%%%%%%%%%%%%%%%%%%
\section{Geometric criterion for being a Drinfeld modular
  subvariety} \label{ch:geom}    
%%%%%%%%%%%%%%%%%%%%%%%%%%%%%%%%%%%%%%%%%%%%%%%%%%%%%%%%%%%%%%%%%%%%%
\setcounter{subsection}{1}
\setcounter{theorem}{0}
\begin{satz} \label{prop:geomcrit_irred}
Let $S = S_{F,\cal K}^r$ be a Drinfeld modular variety
% where 
%${\cal K} = {\cal K}_{\pp} \times {\cal K}^{(\pp)}$ with 
%${\cal K}_{\pp} \subset \GL_r(F_{\pp})$ and ${\cal K}^{(\pp)}
%\subset \GL_r(\AZ_F^{f,\,\pp})$. 
and $Z \subset S$ an irreducible subvariety over $\C$ 
such that $Z = T_gZ = T_{g^{-1}}Z$ for some 
$g = (1,\ldots,g_{\pp},\ldots,1)$ with $g_{\pp} \in
\GL_r(F_{\pp})$. If the cyclic subgroup of $\mathrm{PGL}_r(F_{\pp})$ generated
by the image of $g_{\pp}$ is unbounded, then $Z$ is an irreducible
component of $S$ over $\C$. 
\end{satz}
\begin{proof} Let $x \in Z(\C)$ be a geometric point of $Z$. By
assumption we have\linebreak $T_g(x) \subset T_gZ = Z$ and $T_{g^{-1}}(x)
\subset T_{g^{-1}}Z = Z$, hence
\[ (T_g + T_{g^{-1}})(x) \subset Z. \]
Iterating we get for all $n \geq 1$
\[ (T_g + T_{g^{-1}})^n(x) \subset Z, \]
so the $(T_g + T_{g^{-1}})$-orbit $T_g^\infty(x)$ of $x$ is
contained in $Z$. Since $Z$ is irreducible over~$\C$, the orbit
$T_g^\infty(x)$ is contained in one irreducible component $Y$ of
$S$ over~$\C$. So $T_g^\infty(x)$ is Zariski dense in $Y$
 by Theorem~\ref{th:heckedensity}. Since $Z$ is Zariski
closed in~$S$, it follows that $Z = Y$ is an irreducible component of $S$
over $\C$. 
\end{proof}

\begin{definition} \label{def:Hodge_generic}
A subvariety $X$ defined over $\overline{F}$ of a Drinfeld modular subvariety
$S_{F,\cal K}^r$ is called \emph{Hodge-generic} if none of its
irreducible components over $\C$ is contained
in a proper Drinfeld modular subvariety of $S_{F,\cal K}^r$.
\end{definition}

\begin{theorem} \label{th:geomcrit}
Let $S = S_{F,\cal K}^r$ be a Drinfeld modular variety with 
${\cal K} = {\cal K}_{\pp} \times {\cal K}^{(\pp)}$ amply small
where ${\cal K}_{\pp} \subset \GL_r(F_{\pp})$ and ${\cal K}^{(\pp)}
\subset \GL_r(\AZ_F^{f,\,\pp})$. Suppose that $Z \subset S$ is an
$F$-irreducible Hodge-generic subvariety with $\dim Z \geq 1$ such
that $Z \subset T_gZ$ for some 
$g = (1,\ldots,g_{\pp},\ldots,1)$ with
$g_{\pp} \in \GL_r(F_{\pp})$. If, for all $k_1,k_2 \in {\cal
  K}_{\pp}$, the cyclic subgroup of $\mathrm{PGL}_r(F_{\pp})$ 
generated by the image of $k_1\cdot g_{\pp} \cdot k_2$ is unbounded,
then $Z = S$.
\end{theorem}

\textbf{Remark:} Note that the unboundedness condition in this theorem is
stronger than the one in Proposition~\ref{prop:geomcrit_irred}. For example, 
for $r = 2$, ${\cal K}_{\pp} = \GL_2(A_{\pp})$ and a uniformizer $\pi_{\pp} \in F_{\pp}$, the image of 
$g_{\pp} = \begin{pmatrix}\pi_{\pp} & 0 \\ 0 & 1\end{pmatrix}$ 
generates an unbounded subgroup of $\mathrm{PGL}_2(F_{\pp})$, but for $k_1 = \begin{pmatrix}0 & 1 \\ 1 & 0\end{pmatrix}
\in {\cal K}_{\pp}$, 
the image of $k_1g_{\pp}$ generates a bounded subgroup of $\mathrm{PGL}_2(F_{\pp})$
because $(k_1g_{\pp})^2$ is a scalar matrix.
\begin{proof} In this proof, for simplicity of notation, 
we identify $\GL_r(F_{\pp})$ as a subgroup of $\GL_r(\AZ_F^f)$ via the inclusion
\[ h_{\pp} \in \GL_r(F_{\pp}) \longmapsto (1,\ldots,h_{\pp},\ldots,1)
\in \GL_r(\AZ_F^f). \]

Let $Z = Z_1 \cup \cdots \cup Z_s$ be a decomposition of $Z$ into
irreducible components over $\C$. Since $Z$ is
defined over $F$, the irreducible component $Z_1$ is defined over some 
finite, separable extension $E$ of $F$. By the $F$-irreducibility
of $S$ and~$Z$, it is enough to show that
$Z_1$ is an irreducible component of $S$ over $\C$.
We divide the proof into two steps:

\begin{properties}
\item
We show that there is an open subgroup ${\cal K'} \subset {\cal
  K}$ with associated canonical projection $\pi: S_{F,\cal
  K'}^r \rightarrow S_{F,\cal K}^r$ and an $E$-irreducible component 
$Z'_1$ of $\pi^{-1}(Z_1)$ which is also irreducible over $\C$ 
such that $T_{h_{\pp}}Z'_1$ is $E$-irreducible for all $h_{\pp} \in \GL_r(F_{\pp})$.

\item
Using Proposition~\ref{prop:geomcrit_irred}, we prove that $Z'_1$ is
an irreducible component of $S_{F,\cal K'}^r$ over $\C$.
\end{properties}

Steps (i) and (ii) imply that $Z_1 = \pi(Z'_1)$ is an
irreducible component of $S = S_{F,\cal K}^r$ over $\C$. 

\textbf{Step (i):} Note that, by Proposition~\ref{prop:projetale},
the canonical projections
\[ \pi_{U_{\pp}}: S^r_{F,\,U_{\pp} \times {\cal
  K}^{(\pp)}} \longrightarrow S \]
where $U_{\pp}$ runs over all open normal subgroups of ${\cal
  K}_{\pp}$ form a projective system of finite \'etale Galois covers
defined over $F$ with Galois groups ${\cal K}_{\pp} / U_{\pp}$. Hence,
by Proposition~\ref{prop:projetale}
\[ \pi_{\pp}: S^{(\pp)} := \lim_{\stackrel{\longleftarrow}{U_{\pp}}} 
   S^r_{F,U_{\pp} \times {\cal K}^{(\pp)}} \longrightarrow S \]
is a pro-\'etale Galois cover with group
$\lim_{\stackrel{\longleftarrow}{U_{\pp}}} {\cal K}_{\pp} / U_{\pp}$. 
Since ${\cal K}_{\pp}$ is a profinite group, this group is isomorphic
to ${\cal K}_{\pp}$ and we have the following isomorphisms of rigid-analytic spaces:
\begin{eqnarray*} S^{(\pp)}(\C) & \cong & 
\lim_{\stackrel{\longleftarrow}{U_{\pp}}} \GL_r(F)
\setminus (\Omega_F^r \times \GL_r(\AZ_F^f) / (U_{\pp} \times {\cal
  K}^{(\pp)})) \\ 
& \cong & 
\GL_r(F) \setminus (\Omega_F^r \times \GL_r(\AZ_F^f) / {\cal
  K}^{(\pp)}). 
\end{eqnarray*}
By Proposition~\ref{prop:projetale} and these identifications, 
the automorphism of the ${\cal K}_{\pp}$-cover $\pi_{\pp}$
corresponding 
to a $k_{\pp} \in {\cal K}_{\pp}$ is given by
\[ \lim_{\stackrel{\longleftarrow}{U_{\pp}}} \pi_{k_{\pp}}: [(\overline{\omega},h)] \mapsto
[(\overline{\omega},hk_{\pp}^{-1})] \]
on $\C$-valued points of $S^{(\pp)}$. 

We now denote by $Y$ the nonsingular locus of the variety $Z_1$ over
$\C$. By Theorem I.5.3. in~\cite{Ha}, $Y$ is a non-empty open subset
of $Z_1$ and $Y$ is also defined over $E$. 

Let $y \in Y(\C) \subset S(\C)$ be a geometric point of $Y$. We
denote by $\pi_1^{\mathrm{arithm}}(Y,y)$ the arithmetic fundamental group of the
variety $Y$ over $E$, i.e., $\pi_1^{\mathrm{arithm}}(Y,y) :=
\pi_1(Y_0,y)$ if $Y = (Y_0)_{\C}$ for a scheme $Y_0$ over $E$.
Furthermore we fix a geometric point $x = [(\overline{\omega},h)] \in
S^{(\pp)}(\C)$ with $\pi_{\pp}(x) = y$ and consider the monodromy
representation
\[ \rho: \pi_1^{\mathrm{arithm}}(Y,y) \longrightarrow {\cal
  K}_{\pp} \]
associated to $x \in S^{(\pp)}(\C)$ and the ${\cal K}_{\pp}$-cover~$\pi_{\pp}$.

By Theorem 4 in~\cite{BrPi} the image of $\rho$ is open in
$\GL_r(F_{\pp})$ under the assumptions
\begin{itemize}
\item
 ${\cal K}$ is amply small,

\item
$Y$ is a smooth irreducible locally closed subvariety of $S$ with
$\dim Y \geq 1$,

\item
The Zariski closure of $Y$ in $S$ is Hodge-generic.
\end{itemize}
These assumptions are satisfied in our case, hence ${\cal K}'_{\pp} :=
\rho(\pi_1^{\mathrm{arithm}}(Y,y))$ is open in~${\cal K}_{\pp}$.

Now we set ${\cal K'} := {\cal K}'_{\pp} \times {\cal K}^{(\pp)}$ and
consider the canonical projection 
\[ \pi: S_{F,\cal
  K'}^r \rightarrow S_{F,\cal K}^r.\]
The orbit of the point $x' := [(\overline{\omega},h)] \in S_{F,\cal
  K'}^r(\C)$ lying between our base points $x \in S^{(\pp)}(\C)$ and
$y \in S_{F,\cal K}^r(\C)$ under the action of $\pi_1^{\mathrm{arithm}}(Y,y)$ on
the fiber $\pi^{-1}(y)$ equals
\[ \{[(\overline{\omega},h{k'_{\pp}}^{-1})] \in S_{F,\cal
  K'}^r(\C)\ :\ k'_{\pp} \in \rho(\pi_1^{\mathrm{arithm}}(Y,y)) = {\cal
  K}'_{\pp}\} \]
and is therefore of cardinality~$1$. Hence, the $E$-irreducible
component~$Y'$ of $\pi^{-1}(Y)$ containing $x'$ is
mapped isomorphically onto $Y$ by $\pi$. Since $Y$ is
irreducible over~$\C$, it follows that $Y'$ is also irreducible over
$\C$.

Note furthermore, for any open subgroup $\tilde{{\cal K}'_{\pp}} \subset
{\cal K}'_{\pp}$ and $\tilde{\cal K'} := \tilde{\cal K'_{\pp}} \times
{\cal K}^{(\pp)}$ with canonical projection $\pi': 
S_{F,\tilde{\cal K'}}^r \rightarrow S_{F,{\cal K}'}^r$ that
\[ {\pi'}^{-1}(x') =
\{[(\overline{\omega},hk'_{\pp})] \in S_{F,\tilde{\cal K'}}^r(\C) \ :\
k_{\pp}' \in {\cal K}'_{\pp} \} \]
is exactly one orbit under the action of
$\pi_1^{\mathrm{arithm}}(Y,y)$ on $\pi'^{-1}(\pi^{-1}(y))$. 
Therefore, ${\pi'}^{-1}(Y')$ is $E$-irreducible. Since this holds for
every open subgroup $\tilde{{\cal K}'_{\pp}} \subset {\cal K}_{\pp}'$, this
implies that $T_{h_{\pp}}Y'$ is $E$-irreducible for all $h_{\pp} \in
\GL_r(F_{\pp})$.

We now define $Z_1'$ to be the Zariski closure of $Y'$ in $S_{F,\cal
  K'}^r$. Since $Y'$ is irreducible over~$\C$, its Zariski closure
$Z_1'$ is also irreducible over $\C$ and moreover, by dimension
reasons, an irreducible component of $\pi^{-1}(Z_1)$ over $\C$. Since
$Y'$ is also $E$-irreducible, we similarly conclude that $Z'_1$ is an
$E$-irreducible component of $\pi^{-1}(Z_1)$.
%The finite morphism $\pi$ is closed, hence we have 
%\[ \pi(Z_1') = \pi(\overline{Y'}) =
%\overline{\pi(Y')} = \overline{Y} = Z_1. \]

Note that, for all $h_{\pp} \in \GL_r(F_{\pp})$, the projections
$\pi_1$ and $\pi_{h_{\pp}}$ in the definition of the Hecke
correspondence $T_{h_{\pp}}$ on $S_{F,{\cal K'}}^r$ are open and
closed because they are finite and \'etale. By the $E$-irreducibility
of $T_{h_{\pp}}Y'$ this implies that
\[ T_{h_{\pp}}Z'_1 = \pi_{h_{\pp}}(\pi_1^{-1}(\overline{Y'})) 
= \overline{\pi_{h_{\pp}}(\pi_1^{-1}(Y'))} =
\overline{T_{h_{\pp}}Y'} \]
is $E$-irreducible and concludes step (i). 

\textbf{Step (ii):} By the assumption $Z \subset T_gZ$, the
irreducible component $Z_1$ of $Z$ is contained in $T_gZ_i$ for some
$i$. Since $Z$ is $F$-irreducible, there is an element $\sigma \in
\Gal (F^{\sep} / F)$ with $Z_i = \sigma(Z_1)$. This gives for $Z'_1
\subset S_{F,\cal K'}^r$
\begin{equation} \label{eq:containHecke}
 Z'_1 \subset \pi^{-1}(Z_1) \subset \pi^{-1}(T_g
\sigma(Z_1)) = \sigma(\pi^{-1}(T_gZ_1)), 
\end{equation}
where the last equality holds because all our projection morphisms are
defined over~$F$.

A direct computation shows that
\begin{equation} \label{lemma:Hecke} \pi^{-1}(T_gZ_1) = \bigcup_{i,j=1}^l
T_{k_i^{-1}g_{\pp}k_j}Z'_1 \end{equation}
where $\{k_1,\ldots,k_l\}$ is a set of representatives for the left cosets
in ${\cal K}_{\pp} / {\cal K}'_{\pp}$. By (i), all $T_{k_i^{-1}g_{\pp}k_j}Z'_1$ 
are $E$-irreducible.

%\begin{lemma} \label{lemma:Hecke}
%Let $\{k_1,\ldots,k_l\}$ be a set of representatives for the left cosets
%in ${\cal K}_{\pp} / {\cal K}'_{\pp}$. Then we have 
%\[ \pi^{-1}(T_gZ_1) = \bigcup_{i,j=1}^l
%T_{k_i^{-1}g_{\pp}k_j}Z'_1. \]
%\end{lemma}
%\begin{proof}[Proof of Lemma~\ref{lemma:Hecke}]
%We show both inclusions on $\C$-valued points.

%First consider a $\C$-valued point $x$ of $\pi^{-1}(T_gZ_1)$. 
%Because of $Z_1 = \pi(Z'_1)$ there 
%is a $p' = [(\overline{\omega},h)] \in Z'_1(\C)$ with $x \in 
%\pi^{-1}(T_g\pi(p'))$. Therefore there are $l_1,l_2 \in
%{\cal K}$ such that
%\[ x = [(\overline{\omega},hl_1g_{\pp}^{-1}l_2)] \in S_{F,\cal
%  K'}^r(\C). \]
%Since $\{k_1,\cdots,k_l\}$ is a set of representatives for the left
%cosets in ${\cal K}_{\pp} / {\cal K}'_{\pp} \cong {\cal K} / {\cal
%  K}'$, there are $i$ and $j$ and $k'_1,\,k'_2 \in {\cal K}'$ 
%with $l_1^{-1} = k_ik'_1$ and $l_2 = k_jk'_2$. Hence we have
%\[ x =
%[(\overline{\omega},h{k'_1}^{-1}(k_j^{-1}g_{\pp}k_i)^{-1}k'_2)]
%\in T_{k_j^{-1}g_{\pp}k_i}p' \subset T_{k_j^{-1}g_{\pp}k_i}Z'_1(\C). \]

%For the other inclusion, let $x$ be a $\C$-valued point of 
%$T_{k_i^{-1}g_{\pp}k_j}Z'_1$ for
%some $i$ and $j$. Then
%there is a $p' = [(\overline{\omega},h)] \in Z'_1(\C)$ and a $k' \in {\cal
%  K'}$ with
%\[ x = [(\overline{\omega},hk'(k_i^{-1}g_{\pp}k_j)^{-1})]. \]
%It follows that
%\[ \pi(x) = [(\overline{\omega},hk'k_j^{-1}g_{\pp}^{-1})]
%\in T_g(\pi(p')) \subset T_gZ_1(\C), \]
%hence indeed $x \in \pi^{-1}(T_gZ_1)(\C).$ 
%\end{proof} 

Since $Z'_1$ is $E$-irreducible, the relations \eqref{eq:containHecke} and \eqref{lemma:Hecke} imply 
the existence of indices $i$ and~$j$ such that for $h_{\pp} := k_i^{-1}g_{\pp}k_j$
\[ Z'_1 = \sigma(T_{h_{\pp}}Z'_1). \]
Iterating this gives the inclusion
\[ Z'_1 =
\sigma(T_{h_{\pp}}\sigma(T_{h_{\pp}}Z'_1)) =
\sigma^2(T_{h_{\pp}}(T_{h_{\pp}}Z'_1)) \supset
\sigma^2(T_{h_{\pp}^2}Z'_1), \]
which must be an equality because both sides are of the same
dimension and $Z'_1$ is $E$-irreducible. Repeating the same argument
gives
\[ Z'_1 = \sigma^i(T_{h_{\pp}^i}Z'_1) \]
for all $i \geq 1$. There is an $n \geq 1$ with $\sigma^n \in
\Gal(F^{\sep}/E)$. Since $T_{h_{\pp}^n}Z'_1$ is defined over~$E$, we
conclude the relations
\begin{eqnarray*} Z'_1 & = & \sigma^n(T_{h_{\pp}^n}Z'_1) = T_{h_{\pp}^n}Z'_1, \\ 
 T_{h_{\pp}^{-n}}Z'_1 & = & T_{h_{\pp}^{-n}}(T_{h_{\pp}^n}Z'_1) \supset Z'_1.
\end{eqnarray*}
Again, the latter relation must be an equality because
$T_{h_{\pp}^{-n}}Z'_1$ is $E$-irreducible and of the same dimension as
$Z'_1$. %Hence we have
%\[ Z'_1 = T_{h_{\pp}^n}Z'_1 = T_{h_{\pp}^{-n}}Z'_1. \]
Note that the cyclic subgroup of $\mathrm{PGL}_r(F_{\pp})$ generated
by the image of $h_{\pp}^n = (k_i^{-1}g_{\pp}k_j)^n$ is unbounded by
our assumption. So we can apply Proposition~\ref{prop:geomcrit_irred}
and conclude that $Z'_1$ is an irreducible component of $S_{F,\cal
  K'}^r$ over $\C$. 
\end{proof}

\section{Existence of good primes and suitable Hecke operators} \label{ch:choice}
%%%%%%%%%%%%%%%%%%%%%%%%%%%%%%%%%%%%%%%%%%%%%%%%%%%%%%%%%%%%%%%%%%%%%
\subsection{Good primes}
In this subsection, $X = \incl(S_{F',\cal K'}^{r'})$ denotes a
Drinfeld modular subvariety of a Drinfeld modular variety $S_{F,\cal
  K}^r$ associated to the datum $(F',\,b)$.

\begin{definition}
For a prime~$\pp$ of $F$, a free $A_{\pp}$-submodule $\Lambda_{\pp} \subset 
F_{\pp}^r$ of rank $r$ is called an \emph{$A_{\pp}$-lattice}.
\end{definition}

\begin{definition} \label{def:goodprime}
A prime~$\pp$ is called \emph{good} for $X \subset S_{F,\cal K}^r$ if
there exists an $A_{\pp}$-lattice $\Lambda_{\pp} \subset F_{\pp}^r$
such that
\begin{properties}
%\item $\pp$ is unramified in $F'$,
\item ${\cal K} = {\cal K_{\pp}} \times {\cal K^{(\pp)}}$ with ${\cal
    K}_{\pp}$ the kernel of the natural map
  \[\mathrm{Stab}_{\GL_r(F_{\pp})}(\Lambda_{\pp}) \rightarrow
  \mathrm{Aut}_{k(\pp)}(\Lambda_{\pp} / \pp \cdot
  \Lambda_{\pp})\] for a ${\cal K}^{(\pp)} \subset \GL_r(\AZ_F^{f,\,\pp})$,
% with ${\cal
%    K}_{\pp} = \{t_{\pp} \in
%  \mathrm{Stab}_{\GL_r(F_{\pp})}(\Lambda_{\pp})\,|\,t_{\pp}$ induces
%  the identity on $\Lambda_{\pp} / \pi_{\pp} \cdot \Lambda_{\pp} \}$

\item there is a prime $\pp'$ of $F'$ above $\pp$ with local degree
  $[F'_{\pp'} / F_{\pp}] = 1$,

\item $b_{\pp}(\Lambda_{\pp})$ is an
  $A'_{\pp}$-submodule of ${F'_{\pp}}^{r'}$.
\end{properties}
\end{definition}
\newpage
\textbf{Remarks:} 
\begin{itemize}
\item
The definition is independent of the datum $(F',\,b)$ describing $X$
because $F'$ is uniquely determined by $X$ and $b'_{\pp} = s_{\pp}
\circ b_{\pp} \circ k_{\pp}$ with $s_{\pp} \in \GL_{r'}(F'_{\pp})$ and
$k_{\pp} \in {\cal K}_{\pp} \subset
\mathrm{Stab}_{\GL_r(F_{\pp})}(\Lambda_{\pp})$ for a second datum
$(F',\,b')$ describing $X$ by Corollary~\ref{cor:sub_equal}.

\item
The existence of a good prime $\pp$ for $X$ implies
that the reflex field $F'$ of $X$ is separable over~$F$ because there
exists a prime $\pp'$ of $F'$ which is unramified over $F$.

\item
If $\Lambda_{\pp} = s_{\pp}A_{\pp}^r$ for an $s_{\pp} \in
\GL_r(F_{\pp})$, then condition (i) is equivalent to
\[ {\cal K} = s_{\pp}{\cal K}(\pp)s_{\pp}^{-1} \times {\cal
  K}^{(\pp)}, \]
where ${\cal K}(\pp) \subset \GL_r(A_{\pp})$ is the principal
congruence subgroup modulo $\pp$.

\item
Condition (i) implies that ${\cal K}' = (b{\cal K}b^{-1}) \cap \GL_{r'}(\AZ_{F'}^f)= 
{\cal K}'_{\pp} \times {\cal K}'^{(\pp)}$ with
${\cal K}'_{\pp}$ the kernel of the natural map
\[\mathrm{Stab}_{\GL_{r'}(F'_{\pp})}(b_{\pp}(\Lambda_{\pp})) \rightarrow
  \mathrm{Aut}_{k(\pp)}(b_{\pp}(\Lambda_{\pp}) / \pp \cdot
  b_{\pp}(\Lambda_{\pp})).\]
Since $b_{\pp}(\Lambda_{\pp})$ is an $A'_{\pp}$-submodule of ${F'_{\pp'}}^{r'}$ by (iii), this
means that ${\cal K}'_{\pp}$ is conjugate to the principal congruence subgroup
modulo~$\pp$ of $\GL_{r'}(A'_{\pp})$.
\end{itemize}

%\begin{definition}
%A compact open subgroup ${\cal K} \subset \GL_r(\AZ_F^f)$ is called a
%$\pp$-\emph{con\-gruence subgroup} for an
%$A_{\pp}$-lattice~$\Lambda_{\pp} \subset F_{\pp}^r$ if
% ${\cal K} = {\cal K_{\pp}} \times
% {\cal K^{(\pp)}}$ with ${\cal K_{\pp}}$ the kernel of the natural map
%\[ \mathrm{Stab}_{\GL_r(F_{\pp})}(\Lambda_{\pp}) \rightarrow
%  \mathrm{Aut}_{k(\pp)}(\Lambda_{\pp} / \pi_{\pp} \cdot
%  \Lambda_{\pp})\] 
%for a uniformizer $\pi_{\pp} \in A_{\pp}$.
%\end{definition}

%\begin{satz}
%Let $\pp$ be a good prime for $X$ and an $A_{\pp}$-lattice
%$\Lambda_{\pp} \subset F_{\pp}^r$. Then ${\cal K} \subset
%\GL_r(\AZ_F^f)$ and ${\cal K'} \subset \GL_r(\AZ_{F'}^f)$ are amply small. 
%\end{satz}

\begin{satz} \label{prop:reduceHodge}
Let $\pp$ be a good prime for $X$. Suppose that $X$
is contained in a Drinfeld modular subvariety $X' = 
\inclz(S_{F'',\cal K''}^{r''}) \subset S_{F,\cal K}^r$. 

Then $X'' := (\inclz)^{-1}(X)$ is a Drinfeld modular subvariety of $S_{F'',\cal
  K''}^{r''}$ and there is a prime $\pp''$ of $F''$ above $\pp$ with
$k(\pp) = k(\pp'')$ such
that $\pp''$ is good for $X'' \subset S_{F'',\cal K''}^{r''}$. 
%and the $A''_{\pp''}$-lattice~$\Lambda''_{\pp''}$ where 
%\[ \phi'({m'_{\pp}}^{-1}(\Lambda_{\pp})) = \Lambda''_{\pp''} \times
%{\Lambda''}^{(\pp'')} \]
% with $\Lambda''_{\pp''} \subset {F''_{\pp''}}^{r''}$
%and ${\Lambda''}^{(\pp'')} \subset
%({F''_{\pp}}^{(\pp'')})^{r''}$. Furthermore, ${\cal K''}$ is a
%$\pp''$-congruence subgroup for $\Lambda''_{\pp''}$.
\end{satz}

\begin{proof} By Corollary~\ref{cor:sub_sub}, $X'' = (\inclz)^{-1}(X)$
  is a Drinfeld modular subvariety of $S_{F'',\cal
  K''}^{r''}$. In the proof of Corollary~\ref{cor:sub_sub} we
saw that $F \subset F'' \subset F'$ and 
there are an $\AZ_{F''}^f$-linear isomorphism $c: (\AZ_{F''}^f)^{r''}
\stackrel{\sim}{\rightarrow} (\AZ_{F'}^f)^{r'}$ and a $k \in {\cal K}$ such that
\begin{equation} \label{eq:phipsi}
 b = c \circ b' \circ k 
\end{equation}
and $X'' = \iota^{F'}_{F''\!,\,c}(S_{F',{\cal
    K'}}^{r'})$. The situation is summarized in the following
commutative diagram
where all arrows are bijections on $\C$-valued points:
\[ \xymatrix{&X & \subset & X' & \subset & S_{F,\cal K}^r \\
             S_{F',\cal K'}^{r'} \ar[ur]^{\incl} \ar[dr]_
{\iota^{F'}_{F''\!,\,c}} &&&&& \\
&X'' \ar[uu]_{\left.\inclz\right|_{X''}} & \subset & S_{F'',\cal K''}^{r''}
\ar[uu]_{\inclz} &&  }\]

%Since $X$ is contained in $X'$, by
%Proposition~\ref{prop:sub_contain}, we have $F \subset F'' \subset F'$ and there
%are $s \in \GL_{r''}(\AZ_{F''}^f)$, an $F''$-linear isomorphism $\psi:
%{F''}^{r''} \rightarrow {F'}^{r'}$ and a $k \in {\cal K}$ such that
%\[ \phi \circ m^{-1} = \psi \circ s^{-1} \circ \phi' \circ
%m'^{-1}k. \]
%By Proposition~\ref{prop:twoincl} the diagram
%\[ \xymatrix{ 
%S_{F',\cal K'}^{r'} \ar[dd]_{\iota^{F',\psi}_{F''\!\!,s}} \ar[dr]^{\incl} & \\
%& S \\
%S_{F'',\cal K''}^{r''} \ar[ur]_{\inclz} &
%} \]
%commutes. Therefore, $(\inclz)^{-1}(X) =
%\iota^{F',\psi}_{F''\!\!,s}(S_{F',{\cal K'}}^{r'})$ is a Drinfeld
%modular subvariety of $S_{F'',\cal K''}^{r''}$.
Let $\Lambda_{\pp}$ be an $A_{\pp}$-lattice and $\pp'$ a prime of $F'$ above $\pp$ for
which the conditions (i)-(iii) of Definition~\ref{def:goodprime} 
are satisfied. We define $\pp''$ to be the prime of $F''$ lying
between $\pp$ and $\pp'$. Since $\pp'$ is of local degree~$1$ over $F$, we have $k(\pp) = k(\pp') = k(\pp'')$.
We now show that $\pp''$ is a good prime for
$X'' = \iota^{F'}_{F''\!,\,c}(S_{F',{\cal K'}}^{r'}) \subset S_{F'',\cal K''}^{r''}$.
%We now show that these conditions are also satisfied for
%$X'' = \iota^{F'}_{F''\!,\,c}(S_{F',{\cal K'}}^{r'}) \subset S_{F'',\cal K''}^{r''}$, 
%some prime~$\pp''$ of $F''$ above $\pp$ 
%and $A''_{\pp''}$-lattice~$\Lambda''_{\pp''}$.

By construction, $\pp'$ is also of local degree $1$ over $F''$, i.e., 
(ii) in Definition~\ref{def:goodprime} is satisfied for~$\pp''$.

By (iii), $b_{\pp}(\Lambda_{\pp})$ is an
$A_{\pp}'$-submodule of ${F_{\pp}'}^{r'}$. So we can write
\[ b_{\pp}(\Lambda_{\pp}) = \Lambda'_{\pp''} \times
\Lambda'^{(\pp'')} \]
with $\Lambda'_{\pp''} \subset {F'_{\pp''}}^{r'}$ an
$A'_{\pp''}$-submodule (recall that $A'_{\pp''} = A' 
\otimes_{A''} A''_{\pp''}$ by our conventions). Since 
$c$ is $\AZ_{F''}^f$-linear and $A''
\subset A'$, it follows that $\Lambda''_{\pp''}:= c_{\pp''}^{-1}(\Lambda'_{\pp''})$
is an $A''_{\pp''}$-lattice in ${F''_{\pp''}}^{r''}$. By construction, condition (iii) in
Definition~\ref{def:goodprime} holds for $\pp''$ and $\Lambda''_{\pp''}$.

We note that (i) implies ${\cal K}'' = (b'{\cal K}{b'}^{-1}) 
\cap \GL_{r''}(\AZ_{F''}^f) = {\cal K}''_{\pp} \times {\cal
  K''}^{(\pp)}$ with ${\cal K}''_{\pp}$ the kernel of the natural map
$\mathrm{Stab}_{\GL_{r''}(F''_{\pp})}(b'_{\pp}(\Lambda_{\pp})) \rightarrow
  \mathrm{Aut}_{k(\pp)}(b'_{\pp}(\Lambda_{\pp}) / \pp \cdot
  b'_{\pp}(\Lambda_{\pp})). $
Note that
\[ b'_{\pp}(\Lambda_{\pp}) = b'_{\pp}(k_{\pp}\Lambda_{\pp}) 
  = c_{\pp}^{-1}(b_{\pp}(\Lambda_{\pp})) = c_{\pp}^{-1}(\Lambda'_{\pp''} \times
\Lambda'^{(\pp'')}) = \Lambda''_{\pp''} \times
{\Lambda''}^{(\pp'')}. \]
Since $k(\pp) = k(\pp'')$ and $\pp A''_{\pp''} = \pp''A''_{\pp''}$,
we therefore see that (i) is also satisfied for $\pp''$ and $\Lambda''_{\pp''}$.
\end{proof}

\subsection{Suitable Hecke correspondences} \label{sec:exHecke}
\begin{satz} \label{prop:Heckecontain}
Let $X = \incl(S_{F',\cal K'}^{r'}) \subset S_{F,{\cal K}}^r$ be a
Drinfeld modular subvariety and $g' \in \GL_{r'}(\AZ_{F'}^f)$. Then,
we have 
\[ X \subset T_gX \]
for $g := b^{-1} \circ g' \circ b \in \GL_r(\AZ_F^f)$.
\end{satz}
\begin{proof} Let $p = \incl([(\overline{\omega'},h')]) \in X(\C)$
for some $\overline{\omega'} \in \Omega_{F'}^{r'}$ and $h' \in
\GL_{r'}(\AZ_{F'}^f)$. Then we have
\[ p = [(\overline{\omega' \circ \phi}, \phi^{-1} \circ h' \circ
b)] = [(\overline{\omega' \circ \phi}, (\phi^{-1} \circ h'g' \circ
b \circ g^{-1})] \]
for an $F$-linear isomorphism $\phi: F^r \stackrel{\sim}{\rightarrow}
{F'}^{r'}$,  
therefore $p$ lies in
$T_g\left(\incl([(\overline{\omega'},h'g')])\right)$ and therefore in 
$T_gX(\C)$. Since $p \in X(\C)$ was arbitrary, we conclude $X \subset
T_gX$. 
\end{proof}

\begin{theorem} \label{th:exHecke}
Let $\pp$ be a good prime for a Drinfeld modular 
subvariety~$X = \incl(S_{F',\cal K'}^{r'}) \subset
S_{F,\cal K}^r$ and let $\pp'$ be a prime of $F'$ above $\pp$
with local degree~$1$ over $F$.
%Write $\phi(m_{\pp}^{-1}(\Lambda_{\pp})) =
%\Lambda'_{\pp'} \times {\Lambda'_{\pp}}^{(\pp')}$ with
%$\Lambda'_{\pp'} \subset {F'_{\pp'}}^{r'}$ and
%${\Lambda'_{\pp}}^{(\pp')} \subset ({F'_{\pp}}^{(\pp')})^{r'}$ for a
%prime $\pp' | \pp$ of local degree~$1$ over $F$.
Then there is a
\[ g' = (1,\ldots,1,g'_{\pp'},1,\ldots,1) \in \GL_{r'}(\AZ_{F'}^f) \]
with $g'_{\pp'} \in \GL_{r'}(F'_{\pp'})$ such that the following holds for
$g := b^{-1} \circ g' \circ b \in \GL_r(\AZ_F^f)$:
% is
%given by $\mathrm{diag}(\pi_{\pp'}^{-1},1,\ldots,1)$ with respect to some
%$A_{\pp'}'$-basis of $\Lambda_{\pp'}'$ for a uniformizer
%$\pi_{\pp'} \in A_{\pp'}'$. 
\begin{properties}
\item $X \subset T_gX$,
\item $\deg T_g = [{\cal K} : {\cal K} \cap g^{-1}{\cal
    K}g] = |k(\pp)|^{r-1}$,
\item For all $k_1,k_2 \in {\cal K}_{\pp}$, the cyclic subgroup 
  of $\mathrm{PGL}_r(F_{\pp})$ generated by
  the image of $k_1\cdot g_{\pp} \cdot k_2$ is unbounded.
\end{properties}
\end{theorem}
\begin{proof} Suppose that the conditions (i)-(iii) in 
Definition~\ref{def:goodprime} are satisfied for the $A_{\pp}$-lattice
$\Lambda_{\pp} \subset F_{\pp}^r$. 

By (iii) in Definition~\ref{def:goodprime}, $b_{\pp}(\Lambda_{\pp})$ is an
$A_{\pp}'$-submodule of ${F_{\pp}'}^{r'}$. Hence we can write
\[ b_{\pp}(\Lambda_{\pp}) = \Lambda'_{\pp'} \times {\Lambda'_{\pp}}^{(\pp')} \]
with $\Lambda'_{\pp'} \subset {F'_{\pp'}}^{r'}$ a 
free $A'_{\pp'}$-submodule of rank~$r'$.
Let $g'_{\pp'}: {F'_{\pp'}}^{r'} \rightarrow {F'_{\pp'}}^{r'}$ be given by
\[ \mathrm{diag}(\pi_{\pp'},1,\ldots,1) \]
for a uniformizer $\pi_{\pp'} \in A'_{\pp'}$ 
with respect to an $A'_{\pp'}$-basis of $\Lambda'_{\pp'}$.

We now check the conditions (i)-(iii) for $g'_{\pp'}$. 
Statement (i) follows by Proposition~\ref{prop:Heckecontain}.

For (ii) and (iii), note that each
$A'_{\pp'}$-basis of $\Lambda'_{\pp'}$ is also an $A_{\pp}$-basis 
of $\Lambda'_{\pp'}$ and
can be extended to an $A_{\pp}$-basis of
$b_{\pp}(\Lambda_{\pp})$ because the local degree $[F'_{\pp'} / F_{\pp}]$
is equal to~$1$. In particular, 
$g'_{\pp} = b_{\pp} \circ g_{\pp} \circ b_{\pp}^{-1}:
{F'_{\pp}}^{r'} \rightarrow {F'_{\pp}}^{r'}$ is given by the diagonal matrix
\[ D_{\pp} := \mathrm{diag}(\pi_{\pp},1,\ldots,1) \]
with respect to some $A_{\pp}$-basis~$B'$ of 
$b_{\pp}(\Lambda_{\pp})$ for a uniformizer $\pi_{\pp} \in
A_{\pp}$. It follows that $g_{\pp}: F_{\pp}^r \rightarrow F_{\pp}^r$
is also given by $D_{\pp}$ with respect to the 
$A_{\pp}$-basis~$b_{\pp}^{-1}(B')$ of~$\Lambda_{\pp}$. Hence,
there is a $s_{\pp} \in \GL_r(F_{\pp})$ such that
\begin{eqnarray*}
g_{\pp} & = & s_{\pp}D_{\pp}s_{\pp}^{-1}, \\
\Lambda_{\pp} & = & s_{\pp}A_{\pp}^r.
\end{eqnarray*}
By the remark after Definition~\ref{def:goodprime}, we therefore have
\[ {\cal K}_{\pp} = s_{\pp}{\cal K}(\pp)s_{\pp}^{-1} \]
with ${\cal K}(\pp)$ the principal congruence subgroup of $\GL_r(A_{\pp})$
modulo~$\pp$.

Hence, we can and do assume ${\cal K}_{\pp} = 
{\cal K}(\pp)$ and $g_{\pp} =
D_{\pp}$ because (ii) and (iii) are invariant under conjugation.

For the proof of (ii) consider the map
\[ \alpha: \begin{array}{ccc}
{\cal K}(\pp) & \longrightarrow & (A_{\pp} / (\pi_{\pp}))^{r-1} \\
h & \longmapsto & ([\pi_{\pp}^{-1} \cdot h_{21}],\ldots,
[\pi_{\pp}^{-1} \cdot h_{r1}])
\end{array}. \]
For $h, h' \in {\cal K}(\pp)$, we have for $2 \leq i \leq r$
\begin{eqnarray*} \pi_{\pp}^{-1} \cdot (hh')_{i1} & = &
(\pi_{\pp}^{-1}h_{i1})h'_{11} +
h_{ii}(\pi_{\pp}^{-1}h'_{i1}) + \sum_{j \neq
  i,1}(\pi_{\pp}^{-1}h_{ij})h'_{j1}  \\
& \equiv & \pi_{\pp}^{-1}h_{i1} + \pi_{\pp}^{-1}h'_{i1} + 0\
(\textrm{mod}\ \pp),
\end{eqnarray*}
therefore $\alpha$ is a homomorphism of groups. It is furthermore
surjective and its kernel is exactly equal to ${\cal K}(\pp) \cap D_{\pp}{\cal
    K}(\pp)D_{\pp}^{-1}$. Hence, we have
\[ [{\cal K} : {\cal K} \cap g^{-1}{\cal
    K}g] = [{\cal K}_{\pp} : {\cal K}_{\pp} \cap g_{\pp}^{-1}{\cal
    K}_{\pp}g_{\pp}] = |k(\pp)|^{r-1}. \]

For (iii), let $k_1,k_2 \in {\cal K}_{\pp} = {\cal K}(\pp)$ be
arbitrary. We prove that the eigenvalues of $(k_1g_{\pp}k_2)^{-1} =
k_2^{-1}D_{\pp}^{-1}k_1^{-1}$ do not all have the same $\pp$-valuation
by showing that the Newton polygon of the characteristic polynomial 
\[ \chi(\lambda) = \lambda^r + a_{r-1}\lambda^{r-1} + \cdots + a_1\lambda + a_0 \]
of $k_2^{-1}D_{\pp}^{-1}k_1^{-1}$ 
consists at least of two line segments. This implies
that the cyclic subgroup of $\mathrm{PGL}_r(F_{\pp})$ generated by the
image of $k_1g_{\pp}k_2$ is unbounded.

Since $k_1,k_2$ are elements of $\GL_r(A_{\pp})$, we have
$\det(k_1),\det(k_2) \in A_{\pp}^*$ and hence
\[ v_{\pp}(a_0) = v_{\pp}(\det(k_2^{-1}D_{\pp}^{-1}k_1^{-1})) = 0 - v_{\pp}(\det(D_{\pp}))
+ 0 = -1. \]
The coefficient $a_{r-1}$ can be expressed as
\[ a_{r-1} = -\mathrm{tr}(k_2^{-1}D_{\pp}^{-1}k_1^{-1}) = -\sum_i
(k_2^{-1})_{i1}\pi_{\pp}^{-1}(k_1^{-1})_{1i} - \sum_i \sum_{j \neq 1} 
(k_2^{-1})_{ij}(k_1^{-1})_{ji}.  \]
Because of $k_1,k_2 \in {\cal K}(\pp)$ we have 
$v_{\pp}((k_1^{-1})_{ij}), v_{\pp}((k_2^{-1})_{ij}) \geq 0$ with equality exactly
for $i = j$. Therefore, in the above expression for $a_{r-1}$, the
summand for $i = 1$ in the first sum has $\pp$-valuation $-1$ and all
the other summands have $\pp$-valuation at least $0$. We conclude
\[ v_{\pp}(a_{r-1}) = -1. \]
Hence, the point $(r-1,v_{\pp}(a_{r-1}))$ lies below the line through
$(0,v_{\pp}(a_0))$ and $(r,0)$. This implies that the Newton polygon
of $\chi$ consists at least of two line segments. 
\end{proof}

\subsection{Existence of good primes} \label{sec:exgoodprimes}
\begin{satz} \label{prop:goodprime}
Let $X = \incl(S_{F',\cal K'}^{r'}) \subset S_{F,\cal K}^r$ be a
Drinfeld modular subvariety and $\pp$ a prime of $F$ such that the
following holds:

\begin{properties}
\item
there is a prime $\pp'$ of $F'$ above $\pp$ with local degree
  $[F'_{\pp'} / F_{\pp}] = 1$,

\item
${\cal K} = {\cal K_{\pp}} \times {\cal K}^{(\pp)}$ with ${\cal
  K}_{\pp} \subset \GL_r(F_{\pp})$ a maximal compact subgroup and ${\cal
  K}^{(\pp)} \subset \GL_r(\AZ_F^{f,\,\pp})$,

\item
${\cal K}_{\pp}' := (b_{\pp}{\cal K}_{\pp}b_{\pp}^{-1}) 
\cap \GL_{r'}(F'_{\pp})$ is a maximal compact subgroup of
$\GL_{r'}(F'_{\pp})$.

\end{properties}

Then there is
a subgroup $\tilde{\cal K} \subset {\cal K}$ and a Drinfeld modular
subvariety $\tilde{X} \subset S_{F,\tilde{\cal K}}^r$ such that 

\begin{properties}
\item
$\pi_1(\tilde{X}) = X$ for the canonical projection $\pi_1:
S_{F,\tilde{\cal K}}^r \rightarrow S_{F,\cal K}^r$,

\item
$\pp$ is good for $\tilde{X} \subset S_{F,\tilde{\cal K}}^r$,

\item
$[{\cal K} : \tilde{\cal K}] < |k(\pp)|^{r^2}$.
\end{properties}

\end{satz}
\begin{proof} As ${\cal K}_{\pp}$ is a maximal compact subgroup of
$\GL_r(F_{\pp})$, there is an $s_{\pp} \in \GL_r(F_{\pp})$ with ${\cal
    K}_{\pp} = s_{\pp}\GL_r(A_{\pp})s_{\pp}^{-1}$. We define
  $\Lambda_{\pp}$ to be the lattice $s_{\pp} \cdot A_{\pp}^r$, for
  which we have
\[ {\cal K}_{\pp} = \mathrm{Stab}_{\GL_r(F_{\pp})}(\Lambda_{\pp}). \]
Now, we let $\tilde{{\cal K}_{\pp}}$ be the kernel of the natural map
\[ \mathrm{Stab}_{\GL_r(F_{\pp})}(\Lambda_{\pp}) \rightarrow
  \mathrm{Aut}_{k(\pp)}(\Lambda_{\pp} / \pp \cdot
  \Lambda_{\pp})\] 
and define $\tilde{\cal K} :=
\tilde{{\cal K}_{\pp}} \times {\cal K}^{(\pp)}$.

By construction, we get the upper bound (iii) for the index 
of $\tilde{\cal K}$ in ${\cal K}$:
\[ [{\cal K} : \tilde{\cal K}] = [{\cal K}_{\pp} : \tilde{{\cal
    K}_{\pp}}] = \left|\mathrm{Aut}_{k(\pp)}(\Lambda_{\pp} / 
    \pp \cdot \Lambda_{\pp})\right| = \left| \GL_r(k(\pp)) \right| 
  < |k(\pp)|^{r^2}. \]

We denote by $\tilde{\iota}_{F,\,b}^{F'}$ the inclusion
$S_{F',\tilde{\cal K'}}^{r'} \rightarrow S_{F,\tilde{\cal K}}^r$ 
associated to the same datum $(F',\,b)$ as $\incl$ and set 
$\tilde{X} := \incl(S_{F',\tilde{\cal K'}}^{r'})$. The proof
of Lemma~\ref{lemma:changeK} (i) shows that $\tilde{X}$ is a Drinfeld modular
subvariety of $S_{F,\tilde{\cal K}}^r$ with $\pi_1(\tilde{X}) = X$.

It remains to show that $\pp$ is good for $\tilde{X} \subset
S_{F,\tilde{\cal K}}^r$. Condition (i) in Definition~\ref{def:goodprime}
is satisfied by construction of $\tilde{K}$ and (ii) by assumption.
So we only have
to check that $\Lambda'_{\pp} := b_{\pp}(\Lambda_{\pp})$ is
an $A'_{\pp}$-submodule of ${F'_{\pp}}^{r'}$. Since ${\cal K}_{\pp}$ is the 
stabilizer of $\Lambda_{\pp}$ in $\GL_r(F_{\pp})$, the stabilizer of 
$\Lambda'_{\pp}$ in $\GL_{r'}(F'_{\pp})$ is exactly
\[ {\cal K}_{\pp}' := (b_{\pp}{\cal K}_{\pp}b_{\pp}^{-1}) 
\cap \GL_{r'}(F'_{\pp}), \]
which is a maximal compact subgroup of $\GL_{r'}(F'_{\pp})$ by
assumption. Therefore we have
\[ \mathrm{Stab}_{{F'_{\pp}}^*}(\Lambda'_{\pp}) = {\cal K}_{\pp}'
\cap {F'_{\pp}}^* = {A'_{\pp}}^* \]
because ${A'_{\pp}}^*$ is the unique maximal compact subgroup of
${F'_{\pp}}^*$. Since ${A'_{\pp}}^*$ generates $A'_{\pp}$ as a ring, we
conclude that $\Lambda'_{\pp}$ is an $A'_{\pp}$-submodule of 
of ${F'_{\pp}}^{r'}$. 
\end{proof}

\begin{theorem} \label{th:primeexists}
Let $S = S_{F,\cal K}^r$ be a Drinfeld modular variety and $N >
0$. For every prime~$\qq$ of $F$, denote by ${\cal K}_{\qq}$ the
projection of ${\cal K}$ to $\GL_r(F_{\qq})$. Then, for 
almost all Drinfeld modular
subvarieties $X = \incl(S_{F',\cal K'}^{r'})$ with separable reflex
field over $F$, there is a prime $\pp$
with the following properties:
\begin{properties}
\item
there is a prime $\pp'$ of $F'$ above $\pp$ with local degree
  $[F'_{\pp'} / F_{\pp}] = 1$,

\item
${\cal K}_{\pp} \subset \GL_r(F_{\pp})$ is a maximal compact subgroup and 
${\cal K} = {\cal K}_{\pp} \times {\cal K}^{(\pp)}$ with 
${\cal K}^{(\pp)} \subset \GL_r(\AZ_F^{f,\,\pp})$,

\item
${\cal K}'_{\pp} := (b_{\pp}{\cal K}_{\pp}b_{\pp}^{-1}) 
\cap \GL_{r'}(F'_{\pp})$ is a maximal compact
subgroup of $\GL_{r'}(F'_{\pp})$,

\item
$|k(\pp)|^N < D(X)$ where $D(X)$ denotes the predegree of $X$ from Definition~\ref{def:reflex_index}.
\end{properties}
\end{theorem}

Before giving the proof of this theorem, we show two lemmas.
\begin{lemma} \label{lemma:gCl}
There are absolute constants $C_1,\,C_2 > 0$ such that for all global function
fields~$F'$ with field of constants containing $\FZ_q$
\[ g(F') \leq C_1 + C_2 \cdot \log_q(|\mathrm{Cl}(F')|) \]
where $g(F')$ denotes the genus of $F'$ and $|\mathrm{Cl}(F')|$ the
class number of $F'$.
\end{lemma}
\begin{proof} Let $F'$ be a global function field with field of
constants $\FZ_{q'} \supset \FZ_q$. Then, with
Proposition~\ref{prop:Clg} we get the estimate
\[
 |\mathrm{Cl}(F')| \geq \frac{(q'-1)(q'^{2g(F')} - 2g(F')q'^{g(F')}
  + 1)}{2g(F')(q'^{g(F') + 1} - 1)} \geq (q - 1) \cdot \left( 
  \frac{q^{g(F')-1}}{2g(F')} - \frac{1}{q} \right), 
\]
which implies
\[ \frac{q^{g(F')-1}}{g(F')} \leq \frac{2|\mathrm{Cl}(F')|}{q-1} +
\frac{2}{q} \leq 4|\mathrm{Cl}(F')|, \]
and because of $x / 2 \geq \log_qx - 1$
\[ \frac{g(F')}2 - 2 \leq g(F')-1-\log_qg(F') \leq \log_q(4|\mathrm{Cl}(F')|). \]
So the desired estimate holds for the absolute constants $C_1:=8$ and $C_2:=2$.
\end{proof}

\begin{lemma} \label{lemma:gnormalclosure}
There are constants $C_3,\,C_4 > 0$ only depending on $r$ such that
for all finite separable extensions $F' / F$ of global function fields
with $[F' / F] \leq r$
\[ g(E') \leq C_3 + C_4 \cdot g(F')\]
where $E'$ denotes the normal closure of the extension $F' / F$.
\end{lemma}
\begin{proof} Let $F' / F$ be a finite separable extension of global
function fields of degree $r' \leq r$. 
Its normal closure $E'$ is the compositum of all
Galois conjugates $F'_1,\ldots, F'_{r'}$ of $F'$ over $F$. We
use Castelnuovo's inequality (Theorem III.10.3 in~\cite{St}) to bound its genus:

\emph{If a global function field~$K$ is the compositum of two
  subfields $K_1$ and $K_2$ with\linebreak $n_i := [K / K_i] < \infty$ for $i =
  1,2$, then}
\[ g(K) \leq n_1 \cdot g(K_1) + n_2 \cdot g(K_2) + (n_1 - 1)(n_2 -
1). \]

For $K_1 = F'_1$ and $K_2 = F'_2$ this gives
\[ g(F'_1 F'_2) \leq r'\cdot g(F') + r'\cdot
g(F') + (r'-1)^2 \leq 2r' \cdot g(F') + r'^2 \]
because all Galois conjugates of $F'$ over $F$ have the same genus and
$[F'_1F'_2 / F'_1] \leq [F'_2 / F] = r'$ and $[F'_1F'_2 / F'_2] \leq [F'_1 / F] =
r'$. With induction over $k$ we get
\[ g(F'_1 \cdots F'_k) \leq kr'^{k-1} \cdot
g(F') + (k-1)r'^{k} \]
and with $k = r'$
\[ g(E') \leq r'^{r'} \cdot g(F') + (r'-1) \cdot r'^{r'} \leq (r-1)r^r
+ r^r \cdot g(F'). \qedhere \]
\end{proof}

\begin{proof}[Proof of Theorem~\ref{th:primeexists}] For a Drinfeld modular
subvariety $X = \incl(S_{F',\cal K'}^{r'})$ with separable reflex
field over $F$, we denote by $n(X)$ the
number of primes of $F$ for which (ii) and (iii) do not both hold
and by $m(X,N)$ the number of primes of $F$ with (i) and (iv). We
show the following statements for Drinfeld modular
subvarieties~$X$ of $S$ with separable reflex field:

\begin{propertiesabc}
\item
$n(X) \leq C_5 + C_6 \cdot \log_q(i(X))$ for constants $C_5,\,C_6$ 
independent of $X$ where $i(X)$ denotes the index of $X$ as defined in
Definition~\ref{def:reflex_index},

\item
there is an $M > 0$ such that $m(X,N) > n(X)$ for all $X$ with $D(X) > M$.
\end{propertiesabc}

Statement (b) implies the theorem because $D(X) > M$ for almost all
Drinfeld modular subvarieties $X$ of $S$ by
Theorem~\ref{th:Dunbounded}. 

\textbf{Proof of (a):} For a Drinfeld modular subvariety $X =
\incl(S_{F',\cal K'}^{r'})$ of $S$ we have 
\[ {\cal K'} = (b{\cal K}b^{-1}) \cap
\GL_{r'}(\AZ_{F'}^f) \]
and the index $i(X)$ is the index of ${\cal K'}$ in a maximal compact
subgroup of $\GL_{r'}(\AZ_{F'}^f)$. 

For a prime $\pp$ for which (ii) holds we can write ${\cal K}_{\pp} =
\mathrm{Stab}_{\GL_r(F_{\pp})}(\Lambda_{\pp})$ for some
$A_{\pp}$-lattice $\Lambda_{\pp} \subset F_{\pp}^r$ and
\[ {\cal K}'_{\pp} = (b_{\pp}{\cal K}_{\pp}b_{\pp}^{-1}) \cap \GL_{r'}(F'_{\pp}) =
\mathrm{Stab}_{\GL_{r'}(F'_{\pp})}(\Lambda'_{\pp}) \]
with $\Lambda'_{\pp} := b_{\pp}(\Lambda_{\pp})$. Note that
$A'_{\pp} \cdot \Lambda'_{\pp}$ is a free $A'_{\pp}$-submodule of
rank~$r'$ because $A'_{\pp}$ is a direct product of principal ideal domains.
Therefore with Proposition~\ref{prop_gitter} we get the estimate 
\[ [\mathrm{Stab}_{\GL_{r'}(F'_{\pp})}(A'_{\pp} \cdot
\Lambda'_{\pp}) : {\cal K}'_{\pp}] \geq C \cdot
[A'_{\pp} \cdot \Lambda'_{\pp} : \Lambda'_{\pp}]^{1/r} \]
for some constant $C > 0$ only depending on $q$ and $r$. %Note that 
%$\mathrm{Stab}_{\GL_{r'}(F'_{\pp})}(A'_{\pp} \cdot \Lambda'_{\pp})$ is
%a maximal compact subgroup of $\GL_{r'}(F'_{\pp})$ (it is conjugated
%to $\GL_{r'}(A'_{\pp})$, see e.g. the proof of
%Proposition~\ref{prop_gitter}). 
If ${\cal K}'_{\pp}$ is not a
maximal compact subgroup of $\GL_{r'}(F'_{\pp})$ (i.e., (iii) does not
hold for $\pp$), then $\Lambda'_{\pp}$ cannot be an $A'_{\pp}$-submodule of
${F'_\pp}^{r'}$, i.e., we have $\Lambda'_{\pp}
\subsetneq A'_{\pp} \cdot \Lambda'_{\pp}$ and
\[ [\mathrm{Stab}_{\GL_{r'}(F'_{\pp})}(A'_{\pp} \cdot
\Lambda'_{\pp}) : {\cal K}'_{\pp}] \geq C \cdot |k(\pp)|^{1/r} \]
because each finite non-trivial $A_{\pp}$-module has at least
$|k(\pp)|$ elements.

Since, for each prime $\pp$ satisfying (ii), we have 
${\cal K}' = {\cal K}'_{\pp} \times {\cal K}'^{(\pp)}$ for some subgroup 
${\cal K}'^{(\pp)} \subset \GL_{r'}(F' \otimes \AZ_F^{f,\,\pp})$, we
conclude that
\[ i(X) \geq C \cdot |k(\pp)|^{n_3(X)/r} \geq C \cdot q^{n_3(X)/r}, \]
where $n_3(X)$ is the number of primes of $F$ for which (ii) holds,
but (iii) does not hold. If $n_2$ is the number of primes of $F$, for
which (ii) does not hold, then we conclude
\[ n(X) = n_2 + n_3(X) \leq n_2 - r \cdot \log_q(C) + r \cdot \log_q(i(X)).\]
This finishes the proof of (a) because $n_2$ is independent of $X$.

\textbf{Proof of (b):} Let $X$ be a Drinfeld modular subvariety of $S$
with separable reflex field $F'$ over $F$. We denote the normal
closure of the extension $F'/F$ by~$E'$. To give a lower bound for
$m(X,N)$ we note that all primes $\pp$ of $F$ which completely split
in $E'$ satisfy condition (i). We bound the number of such primes
with fixed degree using an effective version of \v{C}ebotarev's theorem.

For the application of \v{C}ebotarev's theorem we fix some notations. 
We denote the constant extension degree of $E' / F$  by $n$ and its
geometric extension degree by $k$. Since we assumed $F$ to have field
of constants $\FZ_q$, the field of constants of $E'$ is $\FZ_{q^n}$ and
$k = [E' / \FZ_{q^n} \cdot F]$.
We furthermore fix a separating transcendence element~$\theta$ 
of $F / \FZ_q$ (i.e., an element $\theta$ of $F$ such that $F /
\FZ_q(\theta)$ is finite and separable) and set $d := [F / \FZ_q(\theta)]$.   

The effective version of \v{C}ebotarev's theorem in~\cite{FrJa}
(Proposition 6.4.8) says that for all $i \geq 1$ with $n | i$
\[ \left||C_i(E'/F)| - \frac{q^i}{ik} \right| < \frac{2}{ik} \left((k +
  g(E'))q^{i/2} + k(2g(F) + 1)q^{i/4} + g(E') + dk \right) \]
where 
\begin{eqnarray*}
 C_i(E'/F) & := & \{ \pp\ \text{place of}\ F \, | \, k(\pp) = \FZ_{q^i},\ 
\pp\ \text{completely splits in}\ E'\ \text{and}\\
 & & \ \, \pp\ \text{is unramified over}\ \FZ_q(\theta) \}. 
\end{eqnarray*}
We apply this for all $X$ with predegree $D(X) \geq q^{4Nr!}$. Because of $n
\leq [E' / F] \leq r!$, for these $X$ we have $q^n \leq
D(X)^{\frac1{4N}}$. Therefore there are $j \geq 1$ with $n | j$ and
$q^j < D(X)^{\frac1{N}}$ and we can define
\[ i := \max \{ j \geq 1 \, :\, n|j,\, q^j < D(X)^{\frac1{N}} \}. \] 
Our choice of $i$ ensures that
\[ m(X,N) \geq |C_i(E'/F)|. \]
By our choice of $i$ and $X$ we have $q^i < D(X)^{\frac1{N}}$,
$q^{n+i} \geq D(X)^{\frac1{N}}$ and $q^n \leq
D(X)^{\frac1{4N}}$. Hence we have the bounds
\[  q^i < D(X)^{\frac1{N}},\ \ q^i = \frac{q^{n+i}}{q^n} \geq D(X)^{\frac{3}{4N}}. \]
Furthermore Lemma~\ref{lemma:gCl} and \ref{lemma:gnormalclosure} imply
\begin{eqnarray*}
g(F') & \leq & C_1 + C_2 \cdot \log_q(D(X)), \\
g(E') & \leq & C_3 + C_4 \cdot g(F'). 
\end{eqnarray*}
Since $d$ is independent of $X$ and $1 \leq n,k \leq r!$ for all $X$, 
the above conclusion of \v{C}ebotarev's theorem and these
bounds imply
\[ m(X,N) \geq \frac{C_1' \cdot
  D(X)^{\frac{3}{4N}}}{\log_q(D(X))} - \frac{C_2' + C_3'\log_q(D(X))}{\log_q(D(X))}
\left( D(X)^{\frac{1}{2N}} + D(X)^{\frac{1}{4N}} + 1 \right) \]
with $C_1',\,C_2',\,C_3' > 0$ independent of $X$. On
the other hand, our statement (a) gives the bound
\[ n(X) \leq C_5 + C_6 \cdot \log_q(D(X)) \]
with $C_5,\,C_6$ independent of $X$. Since
$x^{\frac{1}{2N}}(\log_q(x))^2 = o(x^{\frac{3}{4N}})$ for $x
  \rightarrow \infty$, these bounds imply the existence of
 an $M > 0$ such that $m(X,N) > n(X)$ for all $X$ with $D(X) > M$.
\end{proof}

%%%%%%%%%%%%%%%%%%%%%%%%%%%%%%%%%%%%%%%%%%%%%%%%%%%%%%%%%%%%%%%%%%%%%%%
\section{The Andr\'e-Oort Conjecture for Drinfeld modular varieties} \label{ch:induction}
%%%%%%%%%%%%%%%%%%%%%%%%%%%%%%%%%%%%%%%%%%%%%%%%%%%%%%%%%%%%%%%%%%%%%%%
%%%%%%%%%%%%%%%%%%%%%%%%%%%%%%%%%%%%%%%%%%%%%%%%%%%%%%%%%%%%%%%%%%%%%%
\subsection{Statement and first reduction} \label{sec:firstred}
%%%%%%%%%%%%%%%%%%%%%%%%%%%%%%%%%%%%%%%%%%%%%%%%%%%%%%%%%%%%%%%%%%%%%%

\begin{conjecture}[(Andr\'e-Oort Conjecture for Drinfeld modular
  varieties)] \label{conj:AO}
Let $S$ be a Drinfeld modular variety and $\Sigma$
a set of special points of $S$. Then each irreducible
component over $\C$ of the Zariski closure of $\Sigma$ is a special
subvariety of $S$.
\end{conjecture}

Our main result is the following theorem:

\begin{theorem} \label{th:AO}
Conjecture~\ref{conj:AO} is true if the reflex fields of all special
points in $\Sigma$ are separable over $F$.
\end{theorem}

Since the reflex field of a special point in $S_{F,\cal K}^r$ is of degree~$r$
over $F$, special points with inseparable reflex field over $F$ can only occur if $r$
is divisible by $p = \mathrm{char}(F)$. Hence, Theorem~\ref{th:AO} implies
\begin{korollar}
Conjecture~\ref{conj:AO} is true if $r$ is not a multiple of $p = \mathrm{char}(F)$.
\end{korollar}

Theorem~\ref{th:AO} follows from the following crucial statement whose
proof we give in the next subsection:
\begin{theorem} \label{th:separable}
Let $S$ be a Drinfeld modular variety and $Z \subset
S$ an $F$-irreducible subvariety.
Suppose that $\Sigma$ is a set of Drinfeld modular subvarieties of~$S$,
all of the same dimension $d < \dim Z$ and with separable reflex field
over $F$, whose union is Zariski dense in $Z$.
Then, for almost all $X \in \Sigma$, there is a Drinfeld modular 
subvariety~$X'$ of $S$ with $X \subsetneq X' \subset Z$.
\end{theorem}
\textbf{Remark:}
By Proposition~\ref{prop:sub_contain}, the proper inclusion $X \subsetneq X'$ implies
that $\dim X < \dim X'$ because the reflex field of $X'$ is properly contained in the
reflex field of $X$.

\begin{satz} \label{prop:reduction}
Theorem~\ref{th:separable} implies Theorem~\ref{th:AO}.
\end{satz}
\begin{proof}[Proof of Proposition~\ref{prop:reduction}] 
We can assume w.l.o.g. that the Zariski closure $Y$ of $\Sigma$ is irreducible
over $\C$. Since each special point in $\Sigma$ is defined over $F^{\sep}$,
the Zariski closure $Y$ of $\Sigma$ is also defined over $F^{\sep}$. Hence,
we can consider the subvariety $Z := \Gal(F^{\sep} / F) \cdot Y$, which is
$F$-irreducible by Proposition~\ref{prop:Firred}. The union~$\Sigma'$ of
all $\Gal(F^{\sep} / F)$-conjugates of the elements of~$\Sigma$ 
is Zariski dense in $Z$. Proposition~\ref{prop:GalFDmsv} implies that
$\Sigma'$ is a union of Drinfeld modular subvarieties of dimension~$0$ 
with separable reflex field over $F$.

Hence, we can apply Theorem~\ref{th:separable} with $d=0$ and find
a finite subset $\tilde{\Sigma} \subset \Sigma$ such that for all
$X \in \Sigma \setminus \tilde{\Sigma}$, there is a Drinfeld modular
subvariety $X'$ with $X \subsetneq X' \subset Z$. We denote the set of these
Drinfeld modular subvarieties $X'$ by $\Sigma'$. Since $\tilde{\Sigma}$ is
finite, the union of all subvarieties in $\Sigma'$ is Zariski dense in $Z$.

Note that Proposition~\ref{prop:sub_contain} implies that all
elements $X'$ of $\Sigma'$ are of positive dimension. Therefore there
is a $d' > 0$ with $d' \leq \dim Z$ such that the Zariski closure of
the union of all subvarieties of dimension~$d'$ in $\Sigma'$ is of
codimension~$0$ in $Z$. We let $\Sigma''$ be the set of all
$\Gal(F^{\sep}/F)$-conjugates of the subvarieties of dimension~$d'$ 
in~$\Sigma'$. Since $Z$ is $F$-irreducible, this is a set of Drinfeld
modular subvarieties of $S$, all of the same dimension~$d' > 0$, whose
union is Zariski dense in $Z$.

If $d' = \dim Z$, then $Y$ is an irreducible component over $\C$ of
an element in $\Sigma''$ and therefore special. If $d' < \dim Z$, we
apply Theorem~\ref{th:separable} with $d = d' > 0$ another time to get
a set of Drinfeld modular subvarieties of dimension $d'' > d'$ whose union
is Zariski dense in $Z$. We iterate this process until we eventually get
such a set with $d'' = \dim Z$, which implies that $Y$ is special.
\end{proof}

%%%%%%%%%%%%%%%%%%%%%%%%%%%%%%%%%%%%%%%%%%%%%%%%%%%%%%%%%%%%%%%%%%%%%
\subsection{Inductive proof in the separable case}
%%%%%%%%%%%%%%%%%%%%%%%%%%%%%%%%%%%%%%%%%%%%%%%%%%%%%%%%%%%%%%%%%%%%%
The proof of Theorem~\ref{th:separable} requires the results from
Subsection~\ref{sec:exgoodprimes} about the existence of good primes and 
the following theorem. We first give an inductive proof of the latter
theorem using our results about existence of suitable Hecke
correspondences from Subsection~\ref{sec:exHecke} and our geometric
criterion in Theorem~\ref{th:geomcrit}.  

\begin{theorem} \label{th:indsep}
Let $S = S_{F,\cal K}^r$ be a Drinfeld modular variety and $X \subset S$ a
Drinfeld modular subvariety over
$F$ which is contained in an $F$-irreducible subvariety $Z \subset
S$ with $\dim Z > \dim X$. Suppose that $\pp$ is a good prime 
for $X \subset S$ and
\[ \deg(X) > |k(\pp)|^{(r-1)\cdot(2^s-1)} \cdot \deg(Z)^{2^s} \]
for $s := \dim Z - \dim X$. Then there is a Drinfeld modular
subvariety $X'$ of $S$ with $X \subsetneq X' \subset Z$.
\end{theorem}
\textbf{Remark:} 
The degree $\deg(X)$ makes sense here because 
$\cal{K}$ is amply small by condition (i) in 
Definition~\ref{def:goodprime}.
 
\begin{proof} In this proof, by ``irreducible component'' we always mean 
an irreducible component over~$\C$. We assume that $X =
\incl(S_{F',\cal K'}^{r'})$. Note that $F'$ is separable over $F$
by the remark after
Definition~\ref{def:goodprime}.

We prove the following statements for all $n \geq 1$:

\begin{properties}
\item
If the theorem is true for $s = n$ and $Z$ Hodge-generic 
(i.e., no irreducible component of $Z$ lies in a
proper Drinfeld modular subvariety of $S$, see 
Definition~\ref{def:Hodge_generic}), then it is true for $s = n$ and
general $Z$.

\item
If the theorem is true for all $s$ with $1 \leq s < n$ and general
$Z$, then it is true for $s = n$ and $Z$ Hodge-generic.
\end{properties}

These two statements imply the theorem by induction over $s$.

\textbf{Proof of (i):} We assume that the theorem is true for $s = n$
and $Z$ Hodge-generic and have to show that it is true for $s = n$ if
$Z$ is not Hodge-generic. In this case, there is an irreducible component of
$Z$ which is contained in a proper Drinfeld modular subvariety of
$S$. Since $\Gal (F^{\sep} / F)$ acts transitively on the irreducible
components of $Z$ (Proposition~\ref{prop:Firred}) and $\Gal (F^{\sep} / F)$ 
acts on the set of Drinfeld modular subvarieties of~$S$ 
(Proposition \ref{prop:GalFDmsv}), also the other
irreducible components of $Z$ are contained in a proper
Drinfeld modular subvariety of $S$. In particular, this is the case
for some chosen irreducible component $Z'$ of $Z$ which contains an
irreducible component $V$ of~$X$. 

We now consider a minimal Drinfeld modular
subvariety $Y = \inclz(S_{F'',\cal K''}^{r''})$ of $S$ with $Z' \subset 
Y \subsetneq S$. By Proposition~\ref{prop:sub_contain}, the reflex
field $F''$ of $Y$ is contained in $F'$ and is therefore also separable
over $F$. Since $Y$ is defined over~$F''$, the
$F''$-irreducible component $Z'' := \Gal (F^{\sep} / F'') \cdot Z'$ of
$Z$ is contained in~$Y$. Furthermore, the $F'$-irreducibility of $X$ (see
Corollary~\ref{cor:subirreducible}) implies 
\[ X = \Gal (F^{\sep} /
F') \cdot V \subset \Gal (F^{\sep} / F'') \cdot V \subset \Gal
(F^{\sep} / F'') \cdot Z' = Z'' \subset Y. \]
We now set $\tilde{X} := (\inclz)^{-1}(X)$ and $\tilde{Z} :=
(\inclz)^{-1}(Z'')$. These are subvarieties of $S_{F'',\cal
  K''}^{r''}$ with
\[ \tilde{X} \subset \tilde{Z} \subset S_{F'',\cal K''}^{r''}. \]
and 
\[ \dim \tilde{Z} - \dim \tilde{X} = \dim Z - \dim X = n. \]
The subvariety $\tilde{Z} = (\inclz)^{-1}(Z'')$ is $F''$-irreducible 
because $Z'' \subset \inclz(S_{F'',\cal K''}^{r''})$ is
$F''$-irreducible and $\inclz$ is a closed immersion defined over
$F''$ by Proposition~\ref{prop:closedimm}. 
 
By Corollary~\ref{cor:sub_sub} and minimality of $Y$, the subvariety
$\tilde{Z} \subset S_{F'',\cal K''}^{r''}$ 
is Hodge-generic and $\tilde{X}$ is a Drinfeld modular subvariety 
of $S_{F'',\cal K''}^{r''}$ with separable reflex field~$F'$ over~$F''$. 
Furthermore, by Proposition~\ref{prop:reduceHodge}, there is a
prime~$\pp''$ of $F''$ above $\pp$ with $k(\pp) = k(\pp'')$ such
that $\pp''$ is good for $\tilde{X} \subset S_{F'',\cal K''}^{r''}$.

Proposition~\ref{prop:degree} (ii) implies
\begin{eqnarray*} \deg \tilde{X} & = & \deg X, \\
\deg \tilde{Z} & = & \deg Z'' \leq \deg Z.
\end{eqnarray*}
Because of $k(\pp) = k(\pp'')$ and $r'' < r$ the assumption
\[ \deg(\tilde{X}) > |k(\pp'')|^{(r''-1)\cdot (2^n - 1)} \cdot
\deg(\tilde{Z})^{2^n} \]
is satisfied. So if Theorem~\ref{th:indsep} is true for $Z$
Hodge-generic and $s = n$ then 
there is a Drinfeld modular subvariety $\tilde{X'}$
of $S_{F'',\cal K''}^{r''}$ with $\tilde{X} \subsetneq \tilde{X'}
\subset \tilde{Z}$ and $X' := \inclz(\tilde{X'})$ is the desired
Drinfeld modular subvariety of $S$ with $X \subsetneq X' \subset
Z$. This concludes the proof of (i).

\textbf{Proof of (ii)}: We assume that the theorem is true for all $s$
with $1 \leq s < n$ and have to show that it is true for $Z$
Hodge-generic and $\dim Z - \dim X = n$. Since $\pp$ is a good prime
for $X$, we
can apply Theorem~\ref{th:exHecke} and find a $g \in \GL_r(\AZ_F^f)$
with the following properties:
\begin{propertiesabc}
\item $X \subset T_gX$,
\item $\deg T_g = [{\cal K} : {\cal K} \cap g^{-1}{\cal
    K}g] = |k(\pp)|^{r-1}$,
\item For all $k_1,k_2 \in {\cal K}_{\pp}$, the cyclic subgroup 
  of $\mathrm{PGL}_r(F_{\pp})$ generated by
  the image of $k_1\cdot g_{\pp} \cdot k_2$ is unbounded.
\end{propertiesabc}
Because of (a) and $X \subset Z$ we have
\[ X \subset Z \cap T_gZ. \]

Lemma~\ref{lemma:bezout} together with Proposition~\ref{prop:degree} and 
property~(b) of our $g \in \GL_r(\AZ_F^f)$ give us the upper bound
\[ \deg (Z \cap T_gZ) \leq \deg Z \cdot \deg T_gZ \leq (\deg Z)^2 \cdot \deg T_g
= (\deg Z)^2 \cdot |k(\pp)|^{r-1}. \]
With the assumption on $\deg X$ and $n = \dim Z - \dim X \geq 1$ we conclude
\[ \deg X > |k(\pp)|^{(r-1)\cdot (2^n-1)} \cdot \deg(Z)^{2^n} \geq \deg (Z \cap T_gZ).\]
Therefore $X$ cannot be a union of irreducible components of 
$Z \cap T_gZ$. Note that $Z \cap T_gZ$ is defined over $F$, hence also 
over the reflex field~$F'$ of $X$. Since $X$ is $F'$-irreducible, 
there is an $F'$-irreducible component $Y'$ of $Z \cap T_gZ$ with 
$X \subset Y'$. We have $X \subsetneq Y'$ because $X$ is not a union
of irreducible components (over $\C$) of $Z \cap T_gZ$.

Now we set $Y := \Gal(F^{\sep} / F) \cdot Y'$. This is an $F$-irreducible
component of $Z \cap T_gZ$ which contains $X$ with $\dim X < \dim Y$. 
We distinguish two cases:

\textbf{Case 1:} $Y = Z$\\
Because of $Y \subset Z \cap T_gZ$ this is only possible if 
$Z \subset T_gZ$. Since $Z$ is $F$-irreducible and Hodge-generic, 
property (c) from above holds and $\cal K$ is amply small, we can
apply our geometric criterion (Theorem~\ref{th:geomcrit}) 
and conclude that $Z = S$. So $X':= Z = S$ satisfies the conclusion of
the theorem.

\textbf{Case 2:} $Y \subsetneq Z$\\
Set $s' := \dim Y - \dim X$. Since $Y$ and $Z$ are $F$-irreducible, 
we have $1 \leq s' < n = \dim Z - \dim X$. Hence, by our assumption,
we can apply the theorem to $X \subset Y \subset S$ and the prime~$\pp$
provided that the inequality of degrees
\[ \deg X > |k(\pp)|^{(r-1) \cdot (2^{s'}-1)} \cdot \deg(Y)^{2^{s'}} \]
holds.

To check the latter, note that $Y$ is a union of irreducible components 
(over $\C$) of $Z \cap T_gZ$ because it is an $F$-irreducible component 
of $Z \cap T_gZ$, whence
\[ \deg Y \leq \deg (Z \cap T_gZ) \leq |k(\pp)|^{r-1} \cdot (\deg Z)^2. \]
Therefore we indeed have
\begin{eqnarray*}
 |k(\pp)|^{(r-1) \cdot (2^{s'}-1)} \cdot \deg(Y)^{2^{s'}} & \leq &
 |k(\pp)|^{(r-1) \cdot (2^{n-1}-1)} \cdot \deg(Y)^{2^{n-1}} \\
& \leq & |k(\pp)|^{(r-1) \cdot (2^{n-1}-1)} \cdot |k(\pp)|^{(r-1) \cdot 2^{n-1}} 
\cdot (\deg Z)^{2^n} \\
& = & |k(\pp)|^{(r-1) \cdot (2^n - 1)} \cdot (\deg Z)^{2^n} < \deg X.
\end{eqnarray*}
So we find a Drinfeld modular subvariety $X'$ of $S$ with 
$X \subsetneq X' \subset Y \subset Z$ as desired.
\end{proof}

\begin{proof}[Proof of Theorem~\ref{th:separable}] We first reduce
ourselves to the case $S=S_{F,\cal K}^r$ with ${\cal K}$ amply small. If
${\cal K}$ is not amply small, there is an amply small open subgroup
${\cal L} \subset {\cal K}$ with corresponding canonical 
projection $\pi_1:
S_{F,{\cal L}}^r \rightarrow S_{F,\cal K}^r$. We choose an
$F$-irreducible component~$\tilde{Z}$ of $\pi_1^{-1}(Z)$ with $\dim Z
= \dim \tilde{Z}$ and set
\[ \tilde{\Sigma} := \{\tilde{X} \subset \tilde{Z}\ F'\text{-irreducible component of }\pi_1^{-1}(X)\, |\, X \in \Sigma \text{ with 
reflex field }F' \}. \]
Since Drinfeld modular subvarieties with reflex field $F'$ are
$F'$-irreducible by Corollary~\ref{cor:subirreducible}, all $\tilde{X}
\in \tilde{\Sigma}$ are Drinfeld
modular subvarieties of $S_{F,{\cal L}}^r$ by
Lemma \ref{lemma:changeK}. They are all
contained in $\tilde{Z}$ and their union is Zariski dense in
$\tilde{Z}$ by our assumption on~$\Sigma$. If
Theorem~\ref{th:separable} is true for $\cal K$ amply small, we
conclude that for almost all $\tilde{X} \in \tilde{\Sigma}$, there is
a Drinfeld modular subvariety $\tilde{X'}$ of $S_{F,{\cal L}}^r$
with $\tilde{X} \subsetneq \tilde{X'} \subset \tilde{Z}$. For such an
$\tilde{X'}$, again by
Lemma~\ref{lemma:changeK}, $X' := \pi_1(\tilde{X'})$ is
a Drinfeld modular subvariety of $S_{F,\cal K}^r$. 
Hence for almost all $X \in \Sigma$, there is a
Drinfeld modular subvariety $X'$ with
$X \subsetneq X' \subset Z$.

So we now assume that $\cal K$ is amply small. By 
Theorem~\ref{th:primeexists} with $N = 2(r-1)\cdot (2^s-1) + r^2\cdot
2^{s+1}$ for $s := \dim Z
- d$, for
almost all $X = \incl(S_{F',\cal K'}^{r'}) \in \Sigma$, there exists a
prime~$\pp$ of $F$ with the properties

\begin{properties}
\item
there is a prime $\pp'$ of $F'$ above $\pp$ with local degree
  $[F'_{\pp'} / F_{\pp}] = 1$,

\item
${\cal K} = {\cal K}_{\pp} \times {\cal K}^{(\pp)}$ with ${\cal
  K}_{\pp} \subset \GL_r(F_{\pp})$ a maximal compact subgroup 
and ${\cal K}^{(\pp)} \subset \GL_r(\AZ_F^{f,\,\pp})$,

\item
${\cal K}'_{\pp} := (b_{\pp}{\cal K}_{\pp}b_{\pp}^{-1}) 
\cap \GL_{r'}(F'_{\pp})$ is a maximal compact
subgroup of $\GL_{r'}(F'_{\pp})$,

\item
$|k(\pp)|^{2(r-1)\cdot (2^s-1) + r^2\cdot 2^{s+1}} < D(X)$ for $s := \dim Z
- d$.
\end{properties}
Furthermore, by Theorem~\ref{th:Dunbounded} we have
\begin{list}{(\roman{ctr})}{\usecounter{ctr}\setlength{\topsep}{0mm}
\setlength{\leftmargin}{12mm}\setlength{\labelsep}{5mm}
\setlength{\labelwidth}{10mm}}
\setcounter{ctr}{4}
\item
$ D(X) > \frac{\deg(Z)^{2^{s+1}}}{C^2} $
\end{list}
for almost all $X \in \Sigma$ with $C$ the constant 
from Proposition~\ref{prop:degD}. 

By Proposition~\ref{prop:goodprime}, for all $X = \incl(S_{F',\cal
  K'}^{r'})$ and $\pp$ with (i)-(v) there is a subgroup 
$\tilde{\cal K} \subset {\cal K}$ and a Drinfeld modular
subvariety $\tilde{X} \subset S_{F,\tilde{\cal K}}^r$ such that 

\begin{properties}
\setcounter{ctr}{5}
\item
$\pi_1(\tilde{X}) = X$ for the canonical projection $\pi_1:
S_{F,\tilde{\cal K}}^r \rightarrow S_{F,\cal K}^r$,

\item
$\pp$ is good for $\tilde{X} \subset S_{F,\tilde{\cal K}}^r$,

\item
$[{\cal K} : \tilde{\cal K}] < |k(\pp)|^{r^2}$.
\end{properties}

Furthermore, for such an $\tilde{X} \subset S_{F,\tilde{\cal K}}^r$,
we choose an $F$-irreducible component~$\tilde{Z}$ of $\pi_1^{-1}(Z)$
with $\tilde{X} \subset \tilde{Z}$. Since $\pi_1$ is finite of degree~$[{\cal K} : {\cal
  \tilde{K}}]$ by
Theorem~\ref{prop:projetale}, we have $\dim \tilde{Z} = \dim Z >
\dim X = \dim \tilde{X}$ and
\begin{eqnarray*}
  \deg \tilde{Z} & \leq & \deg \pi_1^{-1}Z = [{\cal K} : {\cal
  \tilde{K}}]\cdot \deg Z < |k(\pp)|^{r^2} \cdot \deg Z, \\ 
  \deg \tilde{X} & \geq & \deg \pi_1(\tilde{X}) = \deg X
\end{eqnarray*}
by Proposition~\ref{prop:degree}. Therefore, using
Proposition~\ref{prop:degD}, we get the inequality
\begin{eqnarray*} 
 \deg \tilde{X} & \geq & \deg X \geq C \cdot D(X) = D(X)^{1/2} \cdot (C
 \cdot D(X)^{1/2})\\
& \stackrel{(iv),(v)}{>} & {|k(\pp)|^{(r-1)\cdot (2^s-1) + r^2\cdot 2^s} \cdot
  \deg(Z)^{2^{s}}}  \geq |k(\pp)|^{(r-1)\cdot (2^s-1)} \cdot
\deg(\tilde{Z})^{2^s}. 
\end{eqnarray*}
Therefore $\tilde{X} \subset \tilde{Z}\subset S_{F,\tilde{\cal K}}^r$
together with $\pp$ satisfy the assumptions of
Theorem~\ref{th:indsep}. So we find a Drinfeld modular subvariety
$\tilde{X'}$ of $S_{F,\tilde{\cal K}}^r$ with $\tilde{X} \subsetneq
\tilde{X'} \subset \tilde{Z}$ and $X' := \pi_1(\tilde{X'})$ is a
Drinfeld modular subvariety of $S_{F,\cal K}^r$ with $X \subsetneq X'
\subset Z$.
\end{proof} 

\begin{acknowledgements}
The results presented in this article made up my PhD thesis at ETH Z\"urich
in Switzerland. I would like to express 
my deepest gratitude to my advisor Richard Pink for guiding and supporting me 
in course of my PhD studies and carefully revising early drafts of this work. 
In particular I am indebted to him for providing some crucial ideas
which ensured the success of this article. I also thank Gebhard B\"ockle for
going through the manuscript and providing me with useful comments.
\end{acknowledgements}

\end{document}